\documentclass[sn-mathphys]{sn-jnl}
\theoremstyle{thmstyleone}%
\newtheorem{theorem}{Theorem}

\theoremstyle{thmstyletwo}%
\newtheorem{remark}{Remark}%
 \usepackage{subfig}
\theoremstyle{thmstylethree}%
\newtheorem{definition}{Definition}%
\usepackage [latin1]{inputenc}

\usepackage{enumerate}
\usepackage{mathtools}
\newtheorem{assumption}{Assumption}
\newtheorem{lemma}{Lemma}
\newtheorem{corollary}{Corollary}
\begin{document}

\title{General inertial smoothing proximal gradient algorithm for the relaxation of matrix rank minimization problem}


\author[1]{\fnm{Jie} \sur{Zhang}}\email{zjieabc@163.com; xmyang@cqnu.edu.cn}

\author*[2]{\fnm{Xinmin}\sur{Yang}}\email{xmyang@cqnu.edu.cn}


\affil[1]{\orgdiv{College of Mathematics}, \orgname{Sichuan University}, \orgaddress{\city{Chengdu}, \postcode{610065}, \country{China}}}
	
\affil*[2]{\orgdiv{School of Mathematical
		Sciences}, \orgname{Chongqing Normal University}, \orgaddress{\city{Chongqing}, \postcode{401331}, \country{China}}}




\abstract
{We consider the exact continuous relaxation model of matrix rank minimization problem
   proposed by Yu and Zhang (Comput.Optim.Appl. 1-20, 2022). Motivated by the inertial techinique, we propose a general inertial smoothing proximal gradient algorithm(GIMSPG)  for this kind of problems.  It is shown that the singular values of any accumulation point have a common support set and the nonzero singular values have a unified lower bound. Besides,  the zero singular values of the accumulation point  can be achieved within finite iterations.
 Moreover, we prove that any accumulation point of the sequence generated by the GIMSPG algorithm is a lifted stationary point of the continuous relaxation model under the flexible parameter constraint. Finally, we carry out numerical experiments on random data and image data respectively to illustrate the efficiency of the GIMSPG algorithm.}


\maketitle

\section{Introduction}\label{sec1}
In the recent years, much work has been dedicated to the  matrix rank minimization problem which emerge in
 many applications, especially in computer vision \cite{Zheng_Y} and  matrix completion \cite{Recht}. Lots of models, methods and its variants have been studied, one can refer to the literatures\cite{Ma_S,Cai_J,Mesbahi,Lai_M,Fornasier_M,Ma_T,Ji_S,Lu_Z,He_Y,Zhao_Q}, in detail, the matrix rank minimization problem can be represented as
 \begin{align}\label{1.1}
 	\mathop{\min}\limits_{X\in \mathbb{R}^{m\times n}} \quad \mathcal{F}(X):=f(X)+\lambda\cdot \mathrm{rank}(X),
 \end{align}
 where $\mathrm{rank}(X)$ denotes the number of nonzero singular values.

 In this paper, we concentrate on the exact continuous relaxation of matrix rank minimization problem proposed in
 \cite{Yu_Q_Zhang_X}, that is the following  nonconvex relaxation problem,
\begin{align}\label{1.2}
\mathop{\min}_{{X\in \mathbb{R}^{m\times n}}}\limits \quad \mathcal{F}(X):=f(X)+\lambda \Phi(X),
\end{align}
  where $m\ge n$ and $f:\mathbb{R}^{m\times n}\to [0,\infty)$ is convex and not necessairly smooth, $\lambda$ is a positive parameter.
   $\Phi(X)=\sum_{i=1}^{n} \phi(\sigma_i(X))$ is an exact continuous relaxation of the matrix rank minimization problem.
The capped-$\ell_1$ function  $\phi$  with given $v>0$ is
 \begin{align}\label{1.3}
\phi(t)=\mathop{\min\left\{1, t/v\right\}},\, t\ge 0.
\end{align}
It is observed that  $\phi$ in (\ref{1.3}) can be seen as a DC function, that is
\begin{align}\label{2.1}
\phi(t) =\frac{  t }{v}-\max\left\{\theta_1(t),\theta_2(t)\right\},
\end{align}
with $\theta_1(t)=0, \,\theta_2(t)=t/v-1 (t\ge 0)$
and $\mathcal{D}(t)=\{i\in\{1,2\}: \theta_i(t)=\max\{\theta_1(t),\theta_2(t)\}\} $.
 When the matrix $X$ is  to be diagonal, the problem (\ref{1.2}) is reduced to the exact continuous relaxation model of $\ell_0$ regularization problem which was given in
 \cite{Bian_W_Chen_X}, i.e.,
\begin{align}\label{1.2a}
\mathop{\min}\limits_{x\in \mathbb{R}^n} \quad \mathcal{F}(x):=f(x)+\lambda \Phi(x),
\end{align}
  where   $\Phi(x)=\sum_{i=1}^{n} \phi(x_i)$ and $\phi$ is defined as $\phi(t)=\mathop{\min\left\{1, \vert t\vert/v\right\}} $ for $t\in \mathbb{R}$. In \cite{Bian_W_Chen_X}, one proposed the efficient smoothing proximal gradient method  to solve this kind of problem  with box constraint.

  In \cite{Zhang_J}, the authors established the smoothing proximal gradient method with extrapolation (SPGE) algorithm for solving  (\ref{1.2a}). The extrapolation coefficient can be obtained
$\sup\beta_k=1$. Further, under more strict condition of the extrapolation coefficient which includes the extrapolation coefficient in sFISTA  with the fixed restart \cite{ODonoghue}, it is proved that any accumulation point of the sequence generated by SPGE is a lifted stationary point of (\ref{1.2a}). Besides, the convergence rate based on the proximal residual is developed.
 Since the penalty item of problem (\ref{1.2}) is nonconvex for a fixed $d$ in (\ref{3.1}), the framework of SPGE algorithm in \cite{Zhang_J} can not be directly applied to the matrix case.

 As we know, incorporating inertial item which is also called extrapolation  item to the  proximal gradient method  is a popular technique to improve the efficiency of proximal gradient algorithm
    for solving the following composition optimization problem,
  \begin{align}\label{1.1a}
  \min_{x\in \mathbb{R}^n} {\mathcal F}(x)=f(x)+g(x),
  \end{align}
   where $f \colon\mathbb{R}^n\to \mathbb{R}$ is a smooth with Lipschitz continuous gradient and possibly nonconvex function, $g\colon \mathbb{R}^n\to (-\infty,+\infty]$ is a proper closed, and convex function.
     The composition structure that one item is smooth convex and the other is convex makes the proximal gradient method \cite{Parikh_N_Boyd_S} widely used. Specifically, the updating scheme can be read as:
     $$  x^{k+1} =  {{\rm  prox}_{t_kg}}(x^k-t_k\nabla f(x^k)),
     $$
  where $t_k$ is the stepsize. Throughout this paper, the proximal mapping \cite{Parikh_N_Boyd_S} of $\lambda g$ is defined as
   $${{\rm  prox}_{\lambda g}}(u): = \arg\min\limits_{x\in \mathbb{R}^n}\limits\left\{g(x)+\frac{1}{2\lambda}{\Vert x-u\Vert}^2  \right\},
     $$
     where $\lambda>0$ and $u\in\mathbb{R}^n$.
   The accelerated proximal gradient method for convex optimization problems have been well studied \cite{Nesterov_Y0,Nesterov_Y1}. In particular,
     Beck and Teboulle \cite{Beck_A_Teboulle_M} established  the remarkable Fast Iterative Shrinkage-Thresholding Algorithm (FISTA)  for the convex case which was based on the Nesterov's method \cite{Nesterov_Y1,Nesterov_Y}. Some variants of FISTA have been developed by choosing appropriate extrapolation parameters, we refer the readers to \cite{Chambolle_A_Dossal_C,Liang_J_and_C_B} and their references therein.   In \cite{Wu_Z_Li_M}, under the proper parameter constraints,
   Wu and Li established the proximal gradient algorithm with different extrapolation (PGe) coeifficients on the proximal step and gradient step for solving the problem (\ref{1.1a}) when the smooth component is nonconvex and the nomsooth component is convex.   
The general framework is given as  follows:
   \begin{equation}
	\begin{split}\label{a}
	y^k=&x^k+\alpha_k(x^k-x^{k-1}),   \\
	 z^k=&x^k+\beta_k(x^k-x^{k-1}),   \\
 x^{k+1}=& \arg\min_{x\in \mathbb{R}^n}\left\{g(x)+\left<\nabla f(z^k),x-y^k\right>+\frac{1}{2{\lambda}_k}{\Vert x-y_k\Vert}^2\right\},
\end{split}
\end{equation}
where $\alpha_k, \beta_k$ are the extrapolation coefficients, and $\lambda_k$ is the stepsize.
In \cite{Wu_zhongming}, the authors further developed the inertial Bregman proximal gradient method for minimizing the sum of two possible nonconvex functions. In their method,
the generalized Bregman distance replaces the Euclidean distance and  two different inertial items
 on the proximal step and gradient step are taken respectively.

  For matrix case,  Toh and Yun \cite{Toh_K_C_Yun_S}  proposed the accelerated proximal gradient algorithm for the following convex problem:
\begin{align}\label{1.1aa}
	\mathop{\min}\limits_{X\in \mathbb{R}^{m\times n}}
	F(X)= f(X)+g(X),
\end{align}
 where $f$ is a convex and
smooth function with Lipschitz continuous gradient and  $g$ is a proper closed convex function. The problem arises in applications such as in matrix completion problems \cite{Toh_K_C_Yun_S}, multi-task learning \cite{A_Argyriou} and principal component analysis (PCA) \cite{Mesbahi,Xu_H}.
When the matrix is diagonal, the problem is reduced to the vector optimization problem (\ref{1.1a}).

Since optimization problem (\ref{1.2}) is a matrix generalization of the vector optimization problem  (\ref{1.2a}), it is natural for us to explore the possibility of extending some of the algorithms developed for problem (\ref{1.2a}) to
solve problem (\ref{1.2}). We observe that
  the loss function in above require a restrictive  assumption that the function $f$ is smooth, it is difficult to apply the proximal gradient method  directly for problem (\ref{1.2}). Fortunately, the smoothing approxitation methods proposed in \cite{Chen_X} can overcome the dilemma. Smoothing approximation method is an efficient method which is widely used in other problems. One can refer to \cite{Zhang_C_Chen_X,Wu_F_Bian_W_Xue_X} and their references therein. Especially, in \cite{Wei_B}, the author proposed the smoothing fast iterative shrinkage thresholding algorithm (sFISTA) for the convex problem  with the following extrapolation coefficient under the vector case, i.e.,
  \begin{align}\label{c}
  	 &\alpha_k=\beta_k=\frac{t_k-1}{t_{k+1}},  \notag\\
  	& t_{k+1}= \frac{1+\sqrt{1+4{t_k}^2\frac{\mu_k}{\mu_{k+1}}}}{2},
  \end{align}
where $t_0=1$ and $\mu_0\in(0,1)$. They gave the global convergence rate $O(lnk/k)$ on the objective function.

  Motivated by the smoothing method and the general accelerated techinique in (\ref{a}), we propose a general inertial smoothing proximal gradient method (GIMSPG) for matrix case. Thanks to the special structure of penalty item of $(\ref{1.2})$, though the subproblem is nonconvex in the GIMSPG algorithm, it has a closed form solution.
  Most recently,  Li and Bian \cite{li_wenjing} studied a   class of sparse group $\ell_0$ regularized optimization problem and gave an exact  continuous relaxation model of it. They proposed the difference of convex(DC) algorithms for solving the relaxation model which has the DC structure. Moreover, they discussed the  zero  elements of the accumulation point has finite iterations.  Inspired by this,
    under certain conditions of the parameters, we prove that
  the zero singular values of any accumulation point of the iterates generated by the GIMSPG algorithm  have the same locations and the nonzero singular values are always not less than a fixed value. Besides, the GIMSPG algorithm has the ability identifying the zero singular values of any accumulation point within finite iterations. Furthermore, we show that any accumulation point of the sequence generated by GIMSPG algorithm is a lifted stationary point of the continuous relaxation problem (\ref{1.2}).

The outline of this paper is as follows. Some preliminaries are presented in Section 2.
  Then we establish the GIMSPG algorithm framework and discuss its convergence  in Section 3. In Section 4, we conduct numerical  experiments to illustrate the efficiency of the proposed GIMSPG algorithm.  Finally,  we draw some conclusions in Section 5.

\textbf{Notations.} Let $\mathbb{R}^{m\times n}$ be the matrix space of $m\times n$ with the standard inner product, i.e., for any $X,Y\in \mathbb{R}^{m\times n}(m\ge n)$, $\langle X,Y \rangle={\rm{tr}}(X^TY)$, where ${\rm{tr}}(X^TY)$ is the trace of the matrix $X^TY$.   For any matrix $X\in \mathbb{R}^{m\times n}$, the Frobenius norm and nuclear norm are respectively denoted by $\Vert X\Vert {\Vert X\Vert}_F = \sqrt{{\rm{tr}}(X^TX)}$, $\Vert X \Vert_{*}= \sum_{i=1}^{n} \sigma_i(X)$ where $\sigma(X):=(\sigma_1(X),\ldots,\sigma_n(X))^T$ is the singular values of $X$ with $\sigma_1(X)\ge\sigma_2(X)\ge,\ldots,\ge\sigma_n(X)\ge0$. Denote $\mathcal{A}(\sigma(X))=\left\{i\in \left\{1,2,\ldots,n \right\}:\sigma_i(X)\ne 0\right\}$, $\Vert X \Vert_1
= \Vert vec(X)  \Vert_1$, where $vec(X)$ is the vectorization operation of a matrix $X$.
    Denote $\partial f(X)$ be a Clarke subgradient of $f$ at $X\in \mathbb{R}^{m\times n}$ where $f(X): \mathbb{R}^{m\times n}\to \mathbb{R}$ is a locally Lipschitz continuous function. $\mathscr{D}(x)$ is a diagonal matrix generated by vector $x$. $E_i=\mathscr{D}(e_i)$ where $e_i$ is the unit vector and the $i$-th element is 1.
     Denote $\mathbb{Q}^n$ be the set of $n \times n$ dimension unitary orthogonal matrix and $\mathbb{S}^n$ be the real and symmertric $n\times n$ matrices. For vector $x\in \mathbb{R}^n$,
   $\Vert x\Vert_1= \sum_{i=1}^{n} \vert x_i\vert$. Denote   $\mathbb{N}=\left\{0,1,2,\ldots \right\}$, $\mathbb{D}^n=\left\{d\in\mathbb{R}^n: d_i\in \left\{1,2 \right\},i=1,
   2,\ldots,n\right\}$.

\section{Preliminaries}\label{sec2}
In this section, we first recall some definitions and  primary results  which are the basis for the rest of the discussion. First, we give the form of the Clarke subdifferential of $\Phi$ at $X\in\mathbb{R}^{m\times n}$ followed by  \cite{A_S_Lewis}, that is
\begin{align*}
 \partial \Phi(X)=\{U\mathscr{D}(x)V^T: x\in\partial\sum_{i=1}^{n} \phi(\sigma_i(X)), U,V\in \mathcal{M}(X)\},
\end{align*}
 where
$\mathcal{M}(X)= \{(U,V)\in \mathbb{Q}^{m\times m}\times \mathbb{Q}^{n\times n}: U^TU\!=\!V^TV\!=\!I, X\!=\!U\mathscr{D}(\sigma(X))V^T\}$.

\begin{definition}\label{def2.1}(Lifted stationary point \cite{Yu_Q_Zhang_X})
We say that $X\in \mathbb{R}^{m\times n}$ is a lifted stationary point of (\ref{1.2}) if there exist $d_i\in \mathcal{D}(\sigma_i(X))$ for $i=1,2,\ldots,n$ such that
 \begin{align}
\lambda\sum_{i=1}^{n}\theta'_{d_i}(\sigma_i(X))E_i\in
 \{ U\partial f(X)V^T
+\frac{\lambda}{v}\mathscr{D}(\partial \Vert x\Vert_1\mid_{x=\sigma(X)} ): (U,V)\in\mathcal M(X) \},  \label{2.2}
\end{align}
where $\sigma_i(X)$ is the $i$th largest singular value of $X$.
\end{definition}


\begin{assumption} \label{assumption_1}
$f$ is Lipschitz continuous on $\mathbb{R}^{m\times n}$ with the Lipschitz  continuous constant $L_f$.
\end{assumption}
\begin{assumption} \label{assumption_2}
The parameter $v$ in (\ref{1.3}) satisfies $0<v<  \bar{v}:=\frac{\lambda}{L_f}$.
\end{assumption}
\begin{assumption} \label{assumption_3}
 $\mathcal{F}$ in (\ref{1.2})(or $\mathcal{F}_{\ell_0}$ in (\ref{1.1}) ) is level bounded on $\mathbb{R}^{m\times n}$.
\end{assumption}

\begin{lemma}\label{lemma2.2}  {\rm(\cite{Yu_Q_Zhang_X})}	
If $ \bar{X}$ is a lifted stationary point of (\ref{1.2}), then the vector
 $d^{\bar{X}} =(d_{1}^{\bar{X}},\ldots,d_{n}^{\bar{X}})\in \prod_{i=1}^{n}\mathcal{D}({\sigma_i}(\bar{X}))$ satisfying (\ref{2.2}) is unique. In particular, for $i=1,2,\ldots, n$, it has that
  \begin{equation}\label{2.3}
d_{i}^{\bar{X}}= \begin{cases}
1, &\text{ if } {\sigma_i}(\bar{X} ) <v,\\
2,&\text{ if } {\sigma_i}(\bar{X}) \ge v .
\end{cases}
\end{equation}
\end{lemma}

\begin{lemma}\label{lemma2.5}{\rm(\cite{Yu_Q_Zhang_X})}
 $\bar{X}$ is a global minimizer of (\ref{1.2}) if and only if it is a global minimizer of (\ref{1.1}) and the objective functions have the same value at $\bar{X}$.
\end{lemma}

\begin{lemma}\label{lemma 2.6}{\rm(\cite{Yu_Q_Zhang_X})}
For a given $W\in\mathbb{R}^{m\times n}$ and $\tau>0$, suppose
$U\mathscr{D}(w)V^T$ be the singular value decomposition of $W$ and $\hat{x}=prox_{\tau\Phi^d}(w)$, then $\hat X=U\mathscr{D}(\hat x)V^T$ is an optimal solution of the problem
$$ \min_{X} \{\tau\Phi^d(X) +\frac{1}{2} \Vert X-W\Vert^2 \}.
$$
\end{lemma}

\section{The general inertial smoothing proximal gradient algorithm and its convergence analysis }\label{sec3}

 In order to overcome the  nondifferentiability of the convex  loss function $f$ in (\ref{1.2}), we use a
   smoothing function defined as in \rm\cite{Yu_Q_Zhang_X} to approximate convex function  $f$.
\begin{definition}\label{def3.1}
We call $\tilde{f}:\mathbb{R}^{m\times n}\times [0, \bar{\mu}]\to {\mathbb{R}}$ with $ \bar{\mu}>0$ a smoothing function of the function $f$ in (\ref{1.2}) if $\tilde{f}(\cdot,\mu)$ is continuous differentiable in $\mathbb{R}^{m\times n}$ for any fixed $\mu>0$ and satisfies the following conditions:
\begin{enumerate}[(i)]
\item $\lim_{Z\to  X,\mu\downarrow 0}\tilde{f}(Z,\mu)=f( X) \quad \forall X\in \mathbb{R}^{m\times n}$;
 \item$\tilde{f}(X,\mu)$ is convex with respect to $X$ for any fixed $
\mu>0$;
\item $\{\lim_{Z\to X,\mu\downarrow 0}\nabla_Z\tilde{f}(Z,\mu)\}\subseteq\partial f(X) \quad \forall X\in \mathbb{R}^{m\times n}$;
\item there exists a positive constant $\kappa$ such that
$$
\vert \tilde{f}(X,\mu_2)-\tilde{f}(X,\mu_1)\vert\leq\kappa\vert \mu_1-\mu_2\vert,\quad \forall X\in\mathbb{R}^{m\times n},  \mu_1,\mu_2\in[0,\bar{\mu}],
$$
especially, $$\vert \tilde{f}(X,\mu)-f(X)\vert \leq \kappa\mu,\quad\forall X\in \mathbb{R}^{m\times n}, 0<\mu\leq \bar{\mu};$$
\item there exists a constant $\tilde{L}>0$ such that for any $\mu\in(0,\bar{\mu}]$, $\nabla_X\tilde{f}(\cdot,\mu)$ is Lipschitz continuous with Lipschitz constant $\tilde{L}{\mu}^{-1}$.
\end{enumerate}
\end{definition}
    Denote
\begin{align}\label{3.1}
\Phi^d(X):=\sum_{i=1}^{n}(\frac{ {\sigma_i}(X) }{v}-\theta_{d_i}({\sigma_i}(X))\quad d=(d_1,d_2,\ldots,d_n)^T\in \mathbb{D}^n.
\end{align}
  It can be easily got that
 \begin{align}\label{3.2}
 \Phi(X) =\mathop{\min}_{d\in\mathbb{D}^n}\Phi^d(X) \quad \forall X\in {\mathbb{R}}^{m \times n},
 \end{align}
 and  $\Phi(\bar{X})=\Phi^{d^{\bar{X}}}(\bar{X})$ with fixed $\bar{X}$
     where $d^{\bar{X}}$ is defined in (\ref{2.3}).

For the discussion convenience,  we give the notations in the following way
$$ \mathcal{\widetilde{F}}^d(X,\mu):= \widetilde{f}(X,\mu)+ \lambda\Phi^d(X)\quad {\rm{and}}  \quad
 \mathcal{\widetilde{F}}(X,\mu):= \tilde{f}(X,\mu)+ \lambda\Phi(X),
$$
where $\tilde{f}$ is a smoothing function of $f$, $\mu>0$, and $d\in \mathbb{D}^n$. By the formulation of (\ref{3.2}), we have
\begin{align}\label{3.3}
\mathcal{\widetilde{F}}^d (X,\mu)\ge \mathcal{\widetilde{F}}(X,\mu),\quad \forall d\in\mathbb D^n, X\in {\mathbb{R}}^{m\times n},\mu\in(0,\bar{\mu}].
\end{align}

 In the following, we focus on the  following nonconvex optimization problem with the given smoothing parameter $\mu>0$, and $d\in \mathbb{D}^n$:
\begin{align}\label{3.4}
 \mathcal{\widetilde{F}}^d(X,\mu) = \tilde{f}(X,\mu)+ \lambda\Phi^d(X).
\end{align}
When the matrix $X$ is the diagnoal matrix, since the penalty item of (\ref{3.4}) is piecewise
linearity function, the proximal operator of $\tau\Phi^d$ on $\mathbb{R}^{n}$ has a closed form solution with the given vectors $d\in\mathbb{D}^n, w\in \mathbb{R}^{n}$, and a positive number $\tau>0$, in detail, the following proximal operator
\begin{align}\label{3.5}
 \hat x=\arg\min_{x\in\mathbb{R}^n}\left\{\tau\Phi^d(x)+\frac{1}{2}{\Vert x-w \Vert}^2\right\}
\end{align}
can be calculated by $\hat x_i = \mathop{\max}\left\{\hat w_i-\frac{\tau}{v},0 \right\} $,
 $\forall i=1,2,\ldots,n,$
 where  \begin{equation}\label{3.6}
\hat w_i= \begin{cases}
  w_i,&\text{ if }  d_i=1 ,\\
 w_i+\tau/v,&\text{ if } d_i=2.
\end{cases}
\end{equation}

Motivated by the smoothing approximation technique, the proximal gradient algorithm and the efficiency of inertial method, we consider the following  general inetrial smoothing proxiaml gradient  method of $\mathcal{\widetilde F}^{d^k}(\cdot,\mu_k)$,  i.e.,
\begin{equation*}\label{3.7}
\begin{cases}
	Y^k=X^k+\alpha_k(X^k-X^{k-1}),\\
	Z^k=X^k+\beta_k(X^k-X^{k-1}),\\
Q_{d^k}(X,Y^k,Z^k,\mu_k)= \langle X-Y^k, \nabla\tilde{f}(Z^k,\mu_k)\rangle+\frac{1}{2}{h_k\mu^{-1}_k}{\Vert X-Y^k \Vert}^2+\lambda{{\Phi}^{d^k}}(X),
\end{cases}
\end{equation*}
where $\alpha_k, \beta_k$ are different extrapolation parameters and certain conditions are required for them, $h_k$ is the parameter depending on $\tilde{L}$. Especially, when $\alpha_k=\beta_k=0$, the GIMSPG algorithm is reduced to the MSPG algorithm. Based on (\ref{3.5}) and Lemma \ref{lemma 2.6}, we get that  the problem
 $ \mathop{\min Q_{d^k}(X,Y^k,Z^k,\mu_k)}$ has a  minimizer $\hat X=U\mathscr{D}( \hat x)V^T$ with $ \hat x= {\rm{prox}}_{\tau \Phi^d}( w^k) $, $\tau=\lambda {h^{-1}_k}\mu_k$ and $W^k=U\mathscr{D}({w^k}) V^T,$ $W^k=Y^k-{h^{-1}_k}\mu_k\nabla\tilde{f}(Z^k,\mu_k)$.

 \begin{assumption}\label{assumption_4}
    (Parameters constraints)The parameters $\alpha_k,\beta_k,h_k$ in the GIMSPG algorithm satisfy the following  conditions: for any
 $0<\varepsilon\ll 1$, $\alpha_k$ is nonincreasing and $\alpha_k\in[0, \frac{1-2^{\sigma}\varepsilon}{1+2^{\sigma}})$, ${\beta_k}\in[0,1]$, $\{h_k\}$ is nonincreasing and satisfies  $h^{-1}_0\leq h^{-1}_k\leq \min\{\frac{1-\alpha_k-(\alpha_{k}+\varepsilon)\frac{\mu_k}{\mu_{k+1}}}{ (1-\beta_k)\tilde{L}},\frac{\alpha_k}{\beta_k\tilde{L}}\}.$
 \end{assumption}

\begin{algorithm}
\caption{GIMSPG  algorithm}
\label{alg1} 
\begin{algorithmic}[1]
\State {\textbf{Input:}} $X^{-1}=X^{0}\in  \mathbb{R}^{m\times n}$, ${\mu}_{-1}={\mu}_0\in(0,\bar{\mu}]$,  $\sigma\in(0,1)$, choosing the parameters $\alpha_k,\beta_k, h_k$  satisfying the Assumption \ref{assumption_4}. Set $k=0$.
\While{a termination criterion is not met}\\
\quad {\textbf{step 1:}} Let $d^k=d^{X^k}$, where $d_i^{k}$ is defined as in (\ref{2.3}).\\
  \quad {\textbf{step 2:}}
 Compute
\begin{equation}\label{3.7a}
\begin{cases}
	Y^k=X^k+\alpha_k(X^k-X^{k-1}),\\
	Z^k=X^k+\beta_k(X^k-X^{k-1}),\\
  X^{k+1}\in\arg\min\limits_{X\in \mathbb{R}^{m\times  n}}{Q_{d^k}(X,Y^k,Z^k,\mu_k)}.
\end{cases}
\end{equation}\\
\quad {\textbf{step 3:}}
If
  \begin{equation}\label{3.9}
H_{\delta_{k+1}}(X^{k+1}, X^{k},\mu_{k+1},\mu_k)-H_{\delta_{k}}(X^{k}, X^{k-1},\mu_{k},\mu_{k-1})\leq -\alpha {\mu_k}^2,
  \end{equation}

set $\mu_{k+1}=\mu_k$, otherwise, set
\begin{equation}\label{3.10}
\mu_{k+1}=\frac{\mu_0}{(k+1)^\sigma},
\end{equation}

set $ k:=k+1$, and return to Step 1.
\EndWhile
\end{algorithmic}
\end{algorithm}

\begin{remark}\label{remark_1}
 The parameter $h_k$ is bounded. Indeed, if $\frac{1-\alpha_k-(\alpha_{k}+\varepsilon)\frac{\mu_k}{\mu_{k+1}}}{(1-\beta_k)\tilde{L}}>\frac{1}{\tilde{L}}$ and $\frac{\alpha_k}{\beta_k\tilde{L}}>\frac{1}{\tilde{L}}$, then it has that  $\alpha_k+(\alpha_{k}+\varepsilon)\frac{\mu_k}{\mu_{k+1}} <\beta_k<\alpha_k$ which is impossible, then $\frac{1}{\tilde{L}}$ can be seen as a upper bound of $h_k^{-1}$.
\end{remark}
\begin{remark}\label{remark_2}
From Assumption \ref{assumption_4}, when  $\frac{1-\alpha_k-(\alpha_{k}+\varepsilon)\frac{\mu_k}{\mu_{k+1}}}{ (1-\beta_k)\tilde{L}}=\frac{\alpha_k}{\beta_k\tilde{L}}$, we have
the largest range of the stepsize. In this case, $\alpha_k=\frac{\beta_k(1-\varepsilon\frac{\mu_k}{\mu_{k+1}})}{1+\beta_k\frac{\mu_k}{\mu_{k+1}}}$. Moreover, from 
$\inf\frac{\mu_{k}}{\mu_{k+1}}=1$ and $\sup\frac{\mu_{k}}{\mu_{k+1}}=2^\sigma$,
we can take $\alpha_k =\frac{0.98\beta_k }{1+2^{\sigma}\beta_k}$ and
 $h_k=\frac{(1+2^{\sigma}\beta_k)\tilde{L}}{0.98}$. It needs to mention that the parameters in the GIMSPG algorithm can be selected adaptively. However, how to select appropriate parameters is beyond the scope of this paper.
\end{remark}

Next,  we are ready to dicuss the convergence  by constructing the auxiliary sequence. For any $k\in\mathbb{N}$, define
\begin{align}\label{3.111}
H_{\delta_{k+1}}(X^{k+1}, X^{k},\mu_{k+1},\mu_k):=\widetilde{\mathcal{F}}(X^{k+1},\mu_k)+\kappa\mu_k
+ \delta_{k+1}\mu^{-1}_{k+1}{\Vert X^{k+1}-X^k \Vert}^2,
\end{align}
where  $\{X^k\}$, $\{\mu_k \}$ are the sequences generated by GIMSPG algorithm and $\delta_{k+1}= \frac{h_{k+1}\alpha_{k+1}}{2}$.
 It needs to mention that the  sequence is  essential to the convergence analysis of the GIMSPG algorithm. Now, we start by showing that GIMSPG algorithm is well defined.

\begin{lemma}\label{lemma3.2}
The proposed GIMSPG algorithm is well defined, and the sequence $\left\{\mu_k\right\}$ generated by the GIMSPG has the property that there are infinite elements in $\mathcal{N}^s$ and $\lim_{k\to +\infty}\mu_k=0$, where
$\mathcal{N}^s:=\left\{k\in\mathbb{N}: \mu_{k+1}\ne\mu_{k}\right\}$.
 \end{lemma}
\begin{proof}
 We can easily get the result  from \cite{Zhang_J} in Lemma 4.
\end{proof}

   Then we are intend to prove that the  auxiliary sequence $H_{\delta_{k+1}}(X^{k+1}, X^{k},\mu_{k+1},\mu_k)$ is monotonically decreasing.


\begin{lemma}\label{lemma3.3}
For any $k\in \mathbb{N}$, suppose $\{X^k \}$ be the sequence generated by GIMSPG algorithm, it holds that
\begin{align}\label{3.11}
& H_{\delta_{k+1}}(X^{k+1}, X^{k},\mu_{k+1},\mu_k)-H_{\delta_{k}}(X^{k}, X^{k-1},\mu_{k},\mu_{k-1}) \notag\\
\leq &((\frac{\tilde{L}-h_k}{2}+ \frac{h_k\alpha_{k}-\beta_k\tilde{L}}{2})\mu^{-1}_k
+\frac{\alpha_{k}h_{k}\mu^{-1}_{k+1}}{2})\Vert X^{k+1}-X^{k}\Vert^2 \notag\\
&+
 \frac{\tilde{L}\beta_k(\beta_k-1) }{2}\mu^{-1}_k{\Vert X^{k}-X^{k-1}\Vert}^2.
\end{align}
 Moreover, when the parameters  $\alpha_k, \beta_k,h_k$ satisfy Assumption \ref{assumption_4},
  it holds that
 \begin{align}\label{3.1111}
 	H_{\delta_{k+1}}(X^{k+1}, X^{k},\mu_{k+1},\mu_k)-H_{\delta_{k}}(X^{k}, X^{k-1},\mu_{k},\mu_{k-1})
 	\leq \frac{-\varepsilon h_k }{2}\mu^{-1}_{k+1}{\Vert X^{k+1}-X^{k}\Vert}^2,
 \end{align}
and the sequence
$ H_{\delta_{k+1}}(X^{k+1}, X^{k},\mu_{k+1},\mu_k)$
is  nonincreasing.
\end{lemma}


\begin{proof}
 From the optimality of  subproblem (\ref{3.7a}), we have
 \begin{align}\label{3.13}
 &\langle X^{k+1}-Y^k, \nabla \tilde{f}(Z^k,\mu_k)\rangle+\frac{1}{2}h_k{\mu_k}^{-1}{\Vert X^{k+1}-Y^k \Vert}^2+\lambda{{\Phi}^{d^k}}(X^{k+1})     \notag\\
 \leq & \langle X^k-Y^k, \nabla \tilde{f}(Z^k,\mu_k)\rangle+\frac{1}{2}h_k{\mu_k}^{-1}{\Vert X^k-Y^k \Vert}^2+\lambda{{\Phi}^{d^k}}(X^k).
 \end{align}
Since $\nabla\tilde{f}(\cdot,\mu_k)$ is Lipschitz continuous with modulus $\tilde{L}{\mu^{-1}_k}$,  it follows from the Definition \ref{def3.1}(v) that
\begin{align}\label{3.14}
\tilde{f}(X^{k+1},\mu_k)\leq \tilde{f}(Z^{k},\mu_k)+\langle X^{k+1}-Z^k, \nabla \tilde{f}(Z^k,\mu_k)\rangle + \frac{1}{2}\tilde{L}{\mu^{-1}_k}{\Vert X^{k+1}-Z^k \Vert}^2.
\end{align}
Moreover, by the convexity of $\tilde{f}(\cdot,\mu_k)$, it holds that
\begin{align}\label{3.14a}
 \tilde{f}(Z^{k},\mu_k) + \langle X^{k}-Z^k, \nabla \tilde{f}(Z^k,\mu_k) \rangle \leq \tilde{f}(X^{k},\mu_k).
\end{align}
Combining (\ref{3.13}), (\ref{3.14}) and (\ref{3.14a}), we have
\begin{align}\label{3.14aaa}
& \tilde{{\mathcal F}}^{d^k}(X^{k+1},\mu_k)-\widetilde  {\mathcal F}^{d^k}(X^{k},\mu_k)  \\ \notag
\leq&   h_k{\mu^{-1}_k}\langle Y^{k}-X^k, X^{k+1}-X^{k} \rangle
-\frac{1}{2}h_k{\mu^{-1}_k} {\Vert {X}^{k+1}-X^k \Vert}^2+\frac{\tilde{L}\mu^{-1}_k}{2} {\Vert {X}^{k+1}-Z^k \Vert}^2.\notag
\end{align}
Denote $\Delta_k:= X^k-X^{k-1}$, we have
 $\alpha_k\Delta_k= Y^k-X^k$, $\beta_k\Delta_k=Z^k-X^k$, $\beta_k\Delta_k-\Delta_{k+1}=Z^k-X^{k+1}$.
Then it follows from (\ref{3.14aaa}) that,
\begin{align}
&\widetilde{{\mathcal {F}}}^{d^k}(X^{k+1},\mu_k)-\widetilde{{\mathcal{F}}}^{d^k}(X^{k},\mu_k) \notag \\
 \leq& h_k{\mu_k}^{-1} \langle \alpha_k\Delta_k, \Delta_{k+1} \rangle +\frac{\tilde{L} \mu^{-1}_k}{2} {\Vert\beta_k\Delta_k- \Delta_{k+1}\Vert}^2 - \frac{1}{2} h_k{\mu_k}^{-1}{\Vert \Delta_{k+1}\Vert}^2\notag \\
 =& \frac{(\tilde{L}-h_k)\mu^{-1}_k}{2}{\Vert \Delta_{k+1}\Vert}^2+\frac{\tilde{L}}{2}\beta^2_k\mu^{-1}_k{\Vert \Delta_{k}\Vert}^2 +
 (h_k\alpha_k-\beta_k\tilde{L})\mu^{-1}_k \langle \Delta_k, \Delta_{k+1} \rangle  \notag \\
 \leq& (\frac{ \tilde{L}-h_k }{2}+\frac{h_k\alpha_k-\beta_k\tilde{L}}{2})\mu^{-1}_k
 {\Vert \Delta_{k+1}\Vert}^2+(\frac{\tilde{L}\beta^2_k}{2}+\frac{h_k\alpha_k-\beta_k\tilde{L}}{2})\mu^{-1}_k
 {\Vert \Delta_{k}\Vert}^2,
\end{align}
where the second inequality comes from the Assumption \ref{assumption_4} and Cauchy-Schwartz inequality.

Letting $d^k=d^{X^k}$ and according to $ \mathcal{\widetilde{F}}^{d^k}(X^{k+1},\mu_k)\ge \mathcal {\widetilde{ F}}(X^{k+1},\mu_k)$, we have
\begin{align}\label{3.15}
&\widetilde{\mathcal F}(X^{k+1},\mu_k)
+\frac{\alpha_{k+1}h_{k+1}\mu^{-1}_{k+1}}{2}{\Vert \Delta_{k+1}\Vert}^2
-(\widetilde{\mathcal F}(X^{k},\mu_k)+\frac{\alpha_{k }h_{k }\mu^{-1}_{k}}{2}{\Vert \Delta_{k}\Vert}^2 ) \notag \\
&\leq  ((\frac{\tilde{L}-h_k}{2}+\frac{h_k\alpha_k-\beta_k\tilde{L}}{2})\mu^{-1}_k
   +\frac{\alpha_{k+1}h_{k+1}\mu^{-1}_{k+1}}{2})
     {\Vert \Delta_{k+1}\Vert}^2  \notag\\
&\quad+((\frac{\tilde{L}\beta^2_k+h_k\alpha_k-\beta_k\tilde{L}}{2})\mu^{-1}_k-\frac{\alpha_kh_{k}\mu^{-1}_{k}}{2})
 {\Vert \Delta_{k}\Vert}^2.
\end{align}
By the Definition \ref{def3.1}(iv), we easily have that
 \begin{align}\label{3.15a}
 	\widetilde {\mathcal{F}}(X^{k},\mu_k)\leq \widetilde {\mathcal{F}}(X^{k},\mu_{k-1})+\kappa(\mu_{k-1}-\mu_k).
 \end{align}
Combining (\ref{3.15}), (\ref{3.15a}) with the nonincreasing of $h_k$, $\alpha_k$, we get
\begin{align*}
 & H_{\delta_{k+1}}(X^{k+1}, X^{k},\mu_{k+1},\mu_k)-H_{\delta_{k}}(X^{k}, X^{k-1},\mu_{k},\mu_{k-1}) \notag \\
\leq& ((\frac{\tilde{L}-h_k}{2}+\frac{h_k\alpha_k-\beta_k\tilde{L}}{2})\mu^{-1}_k
+\frac{\alpha_{k}h_{k}\mu^{-1}_{k+1}}{2})
{\Vert \Delta_{k+1}\Vert}^2+
  \frac{\tilde{L}(\beta^2_k-\beta_k)}{2}\mu^{-1}_{k}{\Vert \Delta_{k}\Vert}^2.
  \end{align*}
According to the parameters constraint in Assumption \ref{assumption_4}, we easily have that  $(\tilde{L}-h_k+h_k\alpha_k-\beta_k\tilde{L})\mu^{-1}_k+\alpha_{k}h_{k}\mu^{-1}_{k+1
}<-h_k\mu^{-1}_{k+1}\varepsilon$, $\beta^2_k-\beta_k\leq 0$.
 Hence, the inequality (\ref{3.1111}) holds.
So the desired results are obtained.
\end{proof}

  \begin{corollary}\label{corollary_1}
For any $k\in \mathbb{N}$, under the Assumption \ref{assumption_4}, there exists $\zeta\in\mathbb{R}$ satisfying
$$\zeta:= \mathop{\lim_{k\to+\infty}}H_{\delta_{k+1}}(X^{k+1}, X^{k},\mu_{k+1},\mu_k). $$
 Further, it holds that
\begin{align}\label{3.16}
\mathop{\lim_{k\to+\infty}}H_{\delta_{k+1}}(X^{k+1}, X^{k},\mu_{k+1},\mu_k)
&= \mathop{\lim_{k\to+\infty}}\widetilde{\mathcal F}(X^{k},\mu_{k-1})=\mathop{\lim_{k\to+\infty}}\mathcal{F}(X^{k})=\zeta.
 \end{align}
and the sequence $\left\{ X^k\right\}$ is bounded.
\end{corollary}

\begin{proof}
 By the definition of $H_{\delta_{k+1}}$ from (\ref{3.111}), we obtain that
   \begin{align*}
   H_{\delta_{k+1}}(X^{k+1}, X^{k},\mu_{k+1},\mu_k) \ge\widetilde{\mathcal{F}}(X^{k+1},\mu_k)+\kappa\mu_k \ge \mathcal{F}(X^{k+1})
   &\ge\mathop{\min_{X\in\mathbb{
   		 R}^{m\times n}}}\mathcal{F}(X)\\
   	 &=\mathop{\min_{X\in\mathbb{R}^{m\times n} }}\mathcal{F}_{\ell_0}(X),
   	 \end{align*}
where the last equality holds due to the result that the global minimizers of problem (\ref{1.2}) and problem (\ref{1.1}) is equivalent from Lemma \ref{lemma2.5},
  further, according to the fact that $\left\{ H_{\delta_{k+1}}(X^{k+1}, X^{k},\mu_{k+1},\mu_k) \right\}$ is nonincreasing, then there exists $\zeta\in\mathbb{R}$ such that
 \begin{align} \label{3.16a}
 	\mathop{\lim_{k\to+\infty}} H_{\delta_{k+1}}(X^{k+1}, X^{k},\mu_{k+1},\mu_k)=\zeta,
   \quad
H_{\delta_{k+1}}(X^{k+1}, X^{k},\mu_{k+1},\mu_k)\ge \zeta,  \forall k\in\mathbb{N}.
 \end{align}
 Next, note that
 \begin{align}\label{3.19a}
&\vert \widetilde {\mathcal F}(X^{k+1},\mu_{k}) +\kappa\mu_{k}-\zeta\vert \notag\\
=&\vert H_{\delta_{k+1}}(X^{k+1}, X^{k},\mu_{k+1},\mu_{k})-\zeta
 -\frac{\alpha_{k+1}h_{k+1}\mu^{-1}_{k+1}}{2} {\Vert \Delta_{k+1}\Vert}^2\vert  \notag\\
\leq  &\vert H_{\delta_{k+1}}(X^{k+1}, X^{k},\mu_{k+1},\mu_{k})-\zeta\vert
+\frac{\alpha_{k+1}h_{k+1}  }{ \varepsilon h_k }\cdot\frac{\varepsilon h_k \mu^{-1}_{k+1} }{2}{\Vert \Delta_{k+1}\Vert}^2  \notag\\
\leq  & \vert  H_{\delta_{k+1}}(X^{k+1}, X^{k},\mu_{k+1},\mu_{k})-\zeta\vert
+ \frac{\alpha_{k+1}h_{k+1}}{ \varepsilon h_k} ( H_{\delta_{k }}(X^{k }, X^{k-1},\mu_{k},\mu_{k-1})- \zeta)
\end{align}
where the first inequality holds  due to the traingle inequality and the second one derives from (\ref{3.1111}) and (\ref{3.16a}).
When $k$ is sufficiently large, the right of the above inequality tends to 0 by means of (\ref{3.16a}) and the boundedness  of $\frac{\alpha_{k+1}h_{k+1}}{ \varepsilon h_k}$. Therefore, it holds that
$$
\mathop{\lim_{k\to+\infty}}H_{\delta_{k+1}}(X^{k+1}, X^{k},\mu_{k+1},\mu_k)
 = \mathop{\lim_{k\to+\infty}}\widetilde {\mathcal{F}}(X^{k+1},\mu_{k})=\mathop{\lim_{k\to+\infty}} \mathcal F(X^{k})=\zeta.
$$

From the nonincreasing property of the sequence
 $\left\{H_{\delta_{k+1}}(X^{k+1}, X^{k},\mu_{k+1},\mu_k)  \right\}$ again,  it holds that
\begin{align*}
   \mathcal{F}(X^{k+1}) \leq \widetilde{\mathcal{F}}(X^{k+1}, \mu_k)+\kappa \mu_k
    & \leq {H_{\delta_{k+1}}(X^{k+1}, X^{k},\mu_{k+1},\mu_k)} \\
   &\leq
    H_{\delta_{0}}(X^{0}, X^{-1},\mu_{0},\mu_{-1})
    <+\infty.
    \end{align*}
    Then we get that $\left\{ X^k\right\}$ is bounded for any $k\in \mathbb{N}$ from the level bounded of ${\mathcal{F}}$ by the Assumption \ref{assumption_3}.
\end{proof}
\begin{corollary}\label{corollary_2} Under the parameters constraint in Assumption \ref{assumption_4},  it holds that
 \begin{align*}
 \sum_{k=0}^{+\infty}\mu^{-1}_{k} {\Vert X^{k}-X^{k-1}\Vert}^2<+\infty, \quad \sum_{k=0}^{+\infty}\mu^{-1}_{k-1} {\Vert X^{k}-X^{k-1}\Vert}^2
 <+\infty.
 \end{align*}
Moreover, $\lim_{k\to +\infty} \Vert X^{k}-X^{k-1}\Vert=0 $.
 \end{corollary}
 \begin{proof}
  Summing up the inequality (\ref{3.1111}) from $k=0$ to $k=K$ for any $K\in \mathbb{N}$, it yields that
\begin{align*}
 &\sum_{k=0}^{K}\frac{\varepsilon h_k}{2}\mu^{-1}_{k+1 }{\Vert X^{k+1}-X^{k}\Vert}^2\\
\leq&  H_{\delta_0}(X^{0},X^{-1},\mu_{0},\mu_{-1})
-H_{\delta_{K+1}}(X^{K+1},X^K,\mu_{K+1},\mu_K),
\end{align*}
when $K$ is sufficiently large, from (\ref{3.16a}) and the boundedness of $h_k$, we have
\begin{align*}
\sum_{k=0}^{+\infty}\mu^{-1}_{k+1}{\Vert X^{k+1}-X^{k}\Vert}^2  <+\infty.
\end{align*}
Since $\{\mu_k\}$ is a nonincreasing sequence,   it can be easily got that
\begin{align*}
 \sum_{k=0}^{+\infty}\mu^{-1}_{k}{\Vert X^{k+1 }-X^{k}\Vert}^2\leq \sum_{k=0}^{+\infty}\mu^{-1}_{k+1 }{\Vert X^{k+1 }-X^{k}\Vert}^2 <+\infty.
\end{align*}
Further, from $\mu_k\leq\bar{\mu}< 1$ for any $k\in\mathbb{N}$, we have
$$ \lim_{k\to +\infty} \Vert X^{k+1}-X^{k}\Vert=0.$$
So the proof is completed.
\end{proof}

\begin{theorem}\label{theorem3}
	Suppose $\{X^k\}$ be the sequence generated by the GIMSPG algorithm. Then, we have that
	\begin{enumerate}[\rm(i)]
		\item for any  $k\in\mathcal{N}^s$, $d^k$ in Algorithm \ref{alg1} only changes finite number of times;
	   \item  for any accumulation point $\bar{X}$  and $X^*$ of $\{X^k, k\in\mathcal{N}^s\}$, it follows that $\mathcal{A}^c( \sigma(\bar{X}) )= \mathcal{A}^c(\sigma(X^*))$,
		moreover, there exists  $J\in\mathbb{N}$, for any $i\ge J$, it holds that
			$$ \Vert  \sigma(X^{k_i})_{\mathcal A^{c}(\sigma(\bar{X}))}-  \sigma(\bar{X})_{\mathcal A^{c}(\sigma(\bar{X}))}\Vert=0,
		$$
				where $\{X^{k_i}\} $ is a subsequence of $\left\{X^k: k\in\mathcal{N}^s  \right\}$ such that  $\lim_{i\to +\infty} X^{k_{i}}=\bar{X}$.		
	\end{enumerate}
\end{theorem}
\begin{proof}
	(i)	From the first order optimality condition of (\ref{3.7a}), there exist $U,V\in\mathcal{M}(X^{k+1})$ such that
	\begin{align}\label{3.20a}
		 U(\nabla\tilde{f}(Z^k,\mu_k)\!+\!h_k\mu^{-1}_k(X^{k+1}\!-\!Y^k))V^T\!+\! \frac{\lambda}{v}\partial {\Vert \sigma(X^{k+1})\Vert_1}\!-\!\lambda\sum_{j=1}^{n}\nabla\theta_{d^k_j}(\sigma_j(X^{k+1}))E_j =\textbf{0}.
	\end{align}
From (\ref{3.1111}) and  for any $k\in\mathcal{N}^s$, the inequality (\ref{3.9}) does not hold, then we get
$$\frac{\varepsilon h_k }{2}\mu^{-1}_{{k+1}} {\Vert X^{k+1}-X^{k}\Vert}^2\leq \alpha{\mu^2_{k}}  .$$ That is,
$$
\sqrt{\frac{\varepsilon h_k}{2}}\mu^{-1}_{k+1}\Vert X^{k+1}-X^{k}\Vert\leq \sqrt {\alpha\mu^2_{k}\mu^{-1}_{k+1}}.
$$
According to $\lim_{k\to+\infty} \mu_k=0$, the nonincreasing of $\mu_k$  and the boundedness of $h_k$, we have that
\begin{align}\label{3.18a}
	\lim_{k\to +\infty} \mu^{-1}_{k+1} \Vert X^{k+1}-X^{k}\Vert =0,\quad \lim_{k\to +\infty} \mu^{-1}_{k} \Vert X^{k+1}-X^{k}\Vert =0.
\end{align}
Further, since
$$\mu^{-1}_k\Vert X^{k+1}-Y^k \Vert\leq \mu^{-1}_k\Vert X^{k+1}-X^k \Vert+ \mu^{-1}_k\Vert X^{k}-X^{k-1} \Vert,  $$
 we have $\lim_{k\to +\infty} \mu^{-1}_{k} \Vert X^{k+1}-Y^{k}\Vert =0$.  From the boundedness of $h_k$, we get that
\begin{align}\label{3.180a}
 \lim_{k\to +\infty}h_k\mu^{-1}_{k} \Vert X^{k+1}-Y^{k}\Vert =0,
\end{align}
	 then there exists $K\in\mathbb{N}$ such that for any $k\ge K$, it has that
	\begin{align}\label{3.20aa}
	 h_k\mu^{-1}_k\Vert X^{k+1}-Y^k \Vert
		< \varepsilon,
	\end{align}	
	where   $\varepsilon=\frac{1}{2}(\frac{\lambda}{v}-L_f).$
	
	If there exist $k_0\ge k$ and $j\in\{1,2,...,n \}$ such that $ \sigma_j(X^{k_0})<v$, then $d^{k_0}_j=1$ from (\ref{2.3}) which yields that $\nabla \theta_{d^{k_0}_j}(\sigma_j
	(X^{k_0+1}))=0$.
	
	Next, we will prove $\sigma_j
	(X^{k_0+1})=0$ by contradiction.
	If $\sigma_j
	(X^{k_0+1})\neq0$,  from (\ref{3.20aa}), Assumption \ref{assumption_2} and Definition \ref{def3.1}(iii), then it holds that
	$$   (U\nabla \tilde{f}(Z^k,\mu_k)V^T)_{jj}+h_k\mu^{-1}_k(U(X^{k+1}-Y^k)V^T)_{jj}+ \frac{\lambda}{v}\neq 0, $$
	 which contradicts (\ref{3.20a}). In conclusion, if there exist $k_0\ge k$ and $j\in\{1,2,\ldots,n \}$ satisfying
	 $\sigma_j(X^{k_0})<v$, then $\sigma_j(X^{k_0+1})=0$.  Then we get that $\sigma_j(X^{\tilde{k}})\equiv 0$ for any $\tilde k> k_0$.  Therefore, for sufficiently large $\tilde{k}$, it holds that $\sigma_j(X^{\tilde{k}})\equiv 0$ or $\sigma_j(X^{\tilde{k}})\ge v $ for $j=1,\ldots,n$.
So $d^k$  only changes finite number of times.
	
	(ii)
	 Suppose $\bar{X}$  and $X^*$ be any two accumualtion points of $\{X^k\}$, then we get
	   $\mathcal{A}^c(\sigma(\bar{X}))= \mathcal{A}^c(\sigma(X^*))$ by (i) of  this Theorem.  From the boundedness of $\{X^k\}$,   then there exists a subsequence $\{X^{k_i}, k_i\in\mathcal{N}^s\}$ of $\{X^k,k\in\mathcal{N}^s \}$ such that
	$\lim_{i\to +\infty}X^{k_i} = \bar{X}$. Then there exists $J\in\mathbb{N}$, for any $i\ge J$, it has that
	$$ \Vert  \sigma(X^{k_i})_{\mathcal A^{c}(\sigma(\bar{X}))}-  \sigma(\bar{X})_{\mathcal A^{c}(\sigma(\bar{X}))}\Vert=0,
	$$
	so the proof is completed.
\end{proof}

In the following, we establish the subsequence convergence result.
 \begin{theorem}\label{thm1}
 	Under the Assumption \ref{assumption_4},
any accumulation point of $ \left\{X^k: k\in\mathcal{N}^s \right\}$ is a lifted stationary point of problem (\ref{1.2}).
 \end{theorem}
\begin{proof}
 By the boundedness of $\left\{ X^k: k\in\mathcal{N}^s \right\}$ from  Corollary \ref{corollary_1}, letting $\bar{X}$ be an accumulation point of $\left\{ X^k: k\in\mathcal{N}^s \right\}$, then there exists a subsequence $\left\{ X^{k_i} \right\}_{k_i\in\mathcal{N}^s}$ of $\left\{ X^k \right\}_{k\in\mathcal{N}^s}$ satisfying $\lim_{i\to +\infty}X^{k_i}\to \bar{X}$.

Using the first order necessary optimality condition of (\ref{3.7a}), we get that
\begin{align}\label{3.18}
\textbf{0}\in \nabla\tilde{f}(Z^{k_i}, \mu_{k_i})+h_{k_i}\mu^{-1}_{k_i}(X^{k_i+1}-Y^{k_i})+\lambda\xi^{k_i}   \quad \forall \;\xi^{k_i}\in\partial\Phi^{d^{k_i}}(X^{k_i+1}).
\end{align}
  By virtue of $\lim_{i\to+\infty}X^{k_i}=\bar{X}$ and the fact that elements in $ \left\{d^{k_i}: i\in\mathbb{N}\right\}$ are finite, there is a subsequence $ \left\{k_{i_j} \right\}$ of $\left\{k_i \right\}$ and $\bar{d}\in \mathcal{D}(\sigma(\bar{X}))$ such that $d^{k_{i_j}}=\bar{d}$, $\forall j\in\mathbb{N}$. Further, from  $\lim_{j\to+\infty}\Vert X^{k_{i_j}+1}-X^{k_{i_j}} \Vert= 0$, (\ref{3.18a}) and
  the upper semicontinuity of $\partial\Phi^{d^{k_{i_j}}}$,
   we get that
$$\lim_{j\to+\infty}  X^{k_{i_j}+1}  =\bar{X},\quad \lim_{j\to+\infty}  Z^{k_{i_j}}  = \bar{X},\quad \lim_{j\to+\infty}\mu^{-1}_{k_{i_j}}\Vert X^{k_{i_j}+1}-Y^{k_{i_j}}\Vert=0,
$$
and
\begin{align}\label{3.19}
\left\{ \lim_{j\to+\infty}\xi^{k_{i_j}}: \xi^{k_{i_j}}\in\partial\Phi^{d^{k_{i_j}}}(X^{k_{i_j}+1}) \right\}\subseteq \partial\Phi^{\bar{d}}(\bar{X}),
\end{align}
  taking $k_i=k_{i_j}$ in (\ref{3.18}) and letting $j\to+\infty$, together (\ref{3.18a}), (\ref{3.18}), (\ref{3.19}) with the Definition (\ref{def3.1})(iii),
   it deduces that there exist $ \bar{\psi}\in \partial f(\bar{X})$ and
 $ \bar{\xi}^{\bar{d}} \in\partial\Phi^{\bar{d}}(\bar{X})$ satisfying the following
 \begin{align*}
  \textbf{0}=\bar{\psi}+\lambda \bar{\xi}^{\bar{d}} .
\end{align*}
  This completes the proof.
\end{proof}

 \begin{remark}\label{remark_3}
 	 It's well known that, in the vector case, the bregman distance is a generalized form of the Euclidean distance. One can refer to the literatures \cite{Bolte_J,Teboulle_M} which have studied the theoretical analysis and methods based on Bregman distance. Similarly, in the matrix case, for any $X,Y\in \mathbb{S}^{n}$, the Euclidean distance $\frac{1}{2}\Vert X-Y \Vert^2$ used in solving the subproblem (\ref{3.7a}) can be replaced by general Bregman regularization distance $D_{\phi}(X,Y)$  where $\phi:\mathbb{S}^n\to\mathbb{R}$ is a convex and continuously differentiable function which has been studied in \cite{Ma_S,Kulis_B}.
 	
 	 For any $X,Y\in \mathbb{S}^{n}$, the Bregman distance on matrix space is defined as
 	$$ D_\phi(X,Y): =\phi(X)-\phi(Y)- tr(\nabla^T\phi(X)(X-Y)).   $$
 	Specially,
 	when $\phi(X)=\frac{1}{2}\Vert X\Vert^2$, we have
 	$D_{\phi}(X,Y)=\frac{1}{2}\Vert X-Y\Vert^2$.	
 	If the function $\phi$ is strong convex with the strong convex modulus $\varrho$, then $D_\phi(X,Y)\ge \frac{\varrho}{2}\Vert X-Y\Vert^2$.  In this case, the subproblem in (\ref{3.7a}) is
 	generalized as follows:
 	\begin{equation}\label{3.7aa}
 		\begin{cases}
 			Y^k=X^k+\alpha_k(X^k-X^{k-1}),\\
 			Z^k=X^k+\beta_k(X^k-X^{k-1}),\\
 			X^{k+1}=\mathop {\arg\min\limits_{X\in\mathbb{S}^{n} }}{Q'_{d^k}(X,Y^k,Z^k,\mu_k)}
 		\end{cases}
 	\end{equation}
 	where $Q'_{d^k}(X,Y^k,Z^k,\mu_k)=
 	\langle X, \nabla \widetilde f(Z^k,\mu_k)-h_k\mu^{-1}_k(Y_k-X_k)\rangle+h_k \mu^{-1}_k  D_{\phi}(X,X^k)+\lambda{{\Phi}^{d^k}}(X).
 	$
 	In this case, the auxiliary sequence defined as (\ref{3.111})  is descending under the Assumption \ref{assumption_5}  and the proof is presented in Appendix \ref{sec5}.
 	The parameter constraints in  Assumption \ref{assumption_4} is generalized as follows:
 	\begin{assumption}\label{assumption_5}
 	 for any
 	$0<\varepsilon\ll 1$, $\alpha_k\in[0, \frac{\varrho-2^{\sigma}\varepsilon}{1+2^{\sigma}})$, ${\beta_k}\in[0,1]$, $\{h_k\}$ is nonincreasing and satisfies  $h^{-1}_0\leq h^{-1}_k\leq \min\{\frac{\varrho-\alpha_k-(\alpha_{k}+\varepsilon)\frac{\mu_k}{\mu_{k+1}}}{ (1-\beta_k)\tilde{ L}},\frac{\alpha_k}{\beta_k\tilde{L}}\}.$
 	\end{assumption}

 	 Particularly, if the convex  function $\phi$ is a strong convex function with the modulus $\varrho$ satisfying $\varrho\ge(1+2^{\sigma})$, then $\alpha_k\in[0,1), \beta_k\in [0,1]$. This means that the parameter $\alpha_k$ and $\beta_k$ can be chosen as the sFISTA system (\ref{c}) with fixted restart.
 \end{remark}

\section{Numerical Experiments }\label{sec4}
In this section, we aim to verify the efficiency of the GIMSPG algorithm compared with the MSPG algorithm \cite{Yu_Q_Zhang_X}, FPCA \cite{Ma_S}, SVT \cite{Cai_J} and VBMFL1 \cite{Zhao_Q}  on matrix completion problem, i.e.,
\begin{align}\label{4.2}
 \min_{X\in \mathbb{R}^{m\times n}} \mathcal{F}(X):
=\Vert P_\Omega(X-M) \Vert_1+ \lambda\cdot {\rm{rank(X)}},
\end{align}
where $M\in\mathbb{R}^{m\times n}$ and $\Omega$ is the subset of the index set of the matrix $M$, $P_{\Omega}(\cdot)$ is the projection operator which projects onto the  subspace of sparse matrices with nonzero entries confined to the index subset $\Omega$. The goal of this problem is to find the missing entries of the partially observed  low-rank matrix $M$ based on the known elements $\{M_{ij}: i,j\in\Omega\}$. All the numerical experiments are carried out on 1.80GHz Core i5 PC with 12GB of RAM.

In the following, we denote  $ \bar{X}$    the output of the GIMSPG algorithm  and use ``Iter" ``time" to present the number of iterations and the CPU time in seconds respectively. We pick $ X^0= P_{\Omega}(M)$ as the initial point.  The stopping standard is set as
$$\frac{\Vert X^{k+1}- X^k\Vert}{\Vert X^k \Vert} \leq 10^{-4}.$$
For the following $\ell_1$ loss function,
$$ f(x)= \Vert  x-b \Vert_1, \quad {\rm{with}}\;x,b\in \mathbb{R}^{n},
$$
the smoothing function is defined as
\begin{align}
\tilde{f}(x,\mu)= \sum_{i=1}^{n}\tilde{\theta}( x_i-b_i, \mu) \quad {\rm{with}} \quad
\tilde\theta(s,\mu)=\begin{cases}
\vert s \vert \;                 &{\rm{if}} \,\vert s \vert>\mu,\\
\frac{s^2}{2\mu}+\frac{\mu}{2}\; & {\rm {if}} \,\vert s\vert\leq\mu.
\end{cases}
\end{align}
In the numerical simulation, the non-Gaussian noise momdel is set as the typical two-component Gaussian mixture model(GMM), the specific form of the probability density function is
\begin{align}
	p_v(i)=(1-c)N(0,\sigma^2_A)+cN(0,\sigma^2_B),
\end{align}
where $N(0,\sigma^2_A)$ indicates as the general noise disturbance with variance $\sigma^2_A$ and $N(0,\sigma^2_B)$ represents the outliers with a large variance $\sigma^2_B$. The parameter $c$
 trade off between them.

 The  parameters in the GIMSPG algorithm are defined as follows: $\tilde{L}=1$,
  $\sigma =0.9$, $\mu_0=1$, $\alpha=1$, $\lambda=20$.
     From  Remark \ref{remark_2},
  we can take $\alpha_k =\frac{0.98\beta_k }{1+2^{\sigma}\beta_k}$ and $h_k=\frac{(1+2^{\sigma}\beta_k)
  \tilde{L}}{0.98}$.

   Moreover, in order to chose an appropriate value of $\beta_k$,
   we perform the experiments to test the influence of different values of $\beta_k$  on
   GIMSPG algorithm from the RMSE, the last value  $\mu$ of smoothing factor and $time$. The details are presented in Figure \ref{fig1}. From it, we observe that the higher value of $\beta_k$, the used time is less, but RMSE and the last value of the smoothing factor   is larger. For the sake of balance, in the following experiments, we chose $\beta_k=0.4$ in Section \ref{4.1} and \ref{4.21}. In Section \ref{4.3}, we chose $\beta_k=0.3$.
\subsection{Matrix completion on the random data}\label{4.1}
In this subsetion, we perform the numerical experiments on random generated data.
In order to make a fair  comparison, our results are averaged over 20 independent tests. Besides, the way of the generated data is same as \cite{Yu_Q_Zhang_X}, in detail, we first randomly generate matrices $M_L\in \mathbb{R}^{m\times r}$ and $M_R\in \mathbb{R}^{n\times r}$ with i.i.d. standard Gaussian entries and let $M=M_LM^T_R$. We then sample a subset $\Omega$ with sample ration $sr$ uniformly at random, where $sr=\frac{\vert \Omega\vert}{mn}$. In the following experiments, the parameters in the GMM noise are set as $\sigma^2_A=0.0001,\, \sigma^2_B=0.1,\, c=0.1$. The rank of the matrix $M$ is set as $r=30$ and the sampling ration $sr$ are given three different cases $sr=0.2,\,sr=0.6,\,sr=0.8$. The size of the square matrix $m$ increases from 100 to 200 with increament 10.

The root-mean-square error (RMSE) is defined as an error estimation criterion,
\begin{align}\label{4.4}
	\mathrm{RMSE}: = \sqrt{\frac{\Vert  \bar{X}-M \Vert^2}{mn}}.
\end{align}

 The Figure \ref{fig2} presents the last value of the objective function and the last value $\mu$ of smoothing factor for GIMSPG algorithm and MSPG algorithm under different matrix size and $sr=0.2,\,sr=0.6,\,sr=0.8$. It can be easily seen that the last value of the objective function value of GIMSPG algorithm is lower than the MSPG algorithm except $sr=0.2$.  Besides, the last value $\mu$ of smoothing factor  of GIMSPG algorithm are far lower than the MSPG algorithm, that is, the model in GIMSPG algorithm can be better approxiamtion to the original problem. Moreover,  we illustrate the efficiency of GIMSPG algorithm compared with MSPG, FPCA, SVT and VBMFL1 from RMSE and the running time $T$ two aspects under different matrix size for $sr=0.2,\,sr=0.6,\,sr=0.8$ in Tabels \ref{tab0},\ref{tab1},\ref{tab2}.  From them, we see that
  the running time of GIMSPG algorithm  is the least one in these algorithms.
    We also observe that  as the dimension increases, so does the time.
     Compared with MSPG algorithm, when the sample ratio is lower, the RMSE of GIMSPG algorithm is better. And compared with the other algorithms, the  GIMSPG algorithm always has the least value of RMSE.

   \begin{figure}[htpb]
   	\footnotesize
  	\centering

  	\begin{minipage}[t] {0.4 \textwidth}
  		\includegraphics[scale=0.05]{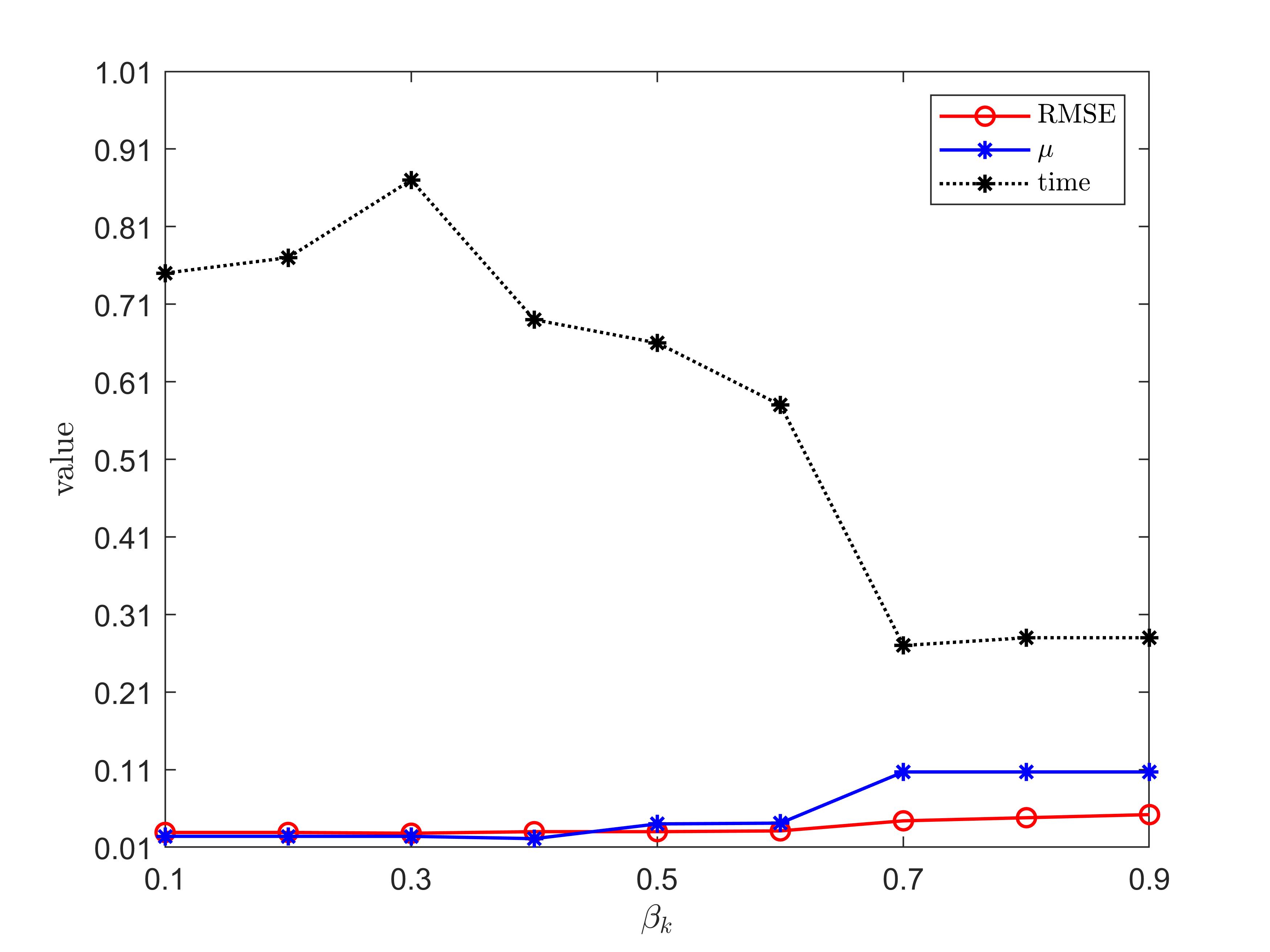}
  		\end {minipage}	
  		\centering\caption{The  $time$, the last value of $\mu$ and   RMSE  for $sr= 0.6$ and $m,n=150$}
  		\label{fig1}
  	\end{figure}
  	\begin{figure}[htpb]
  		 \centering
  		\begin{minipage}[t] {0.5\textwidth}
  			\includegraphics[scale=0.05]{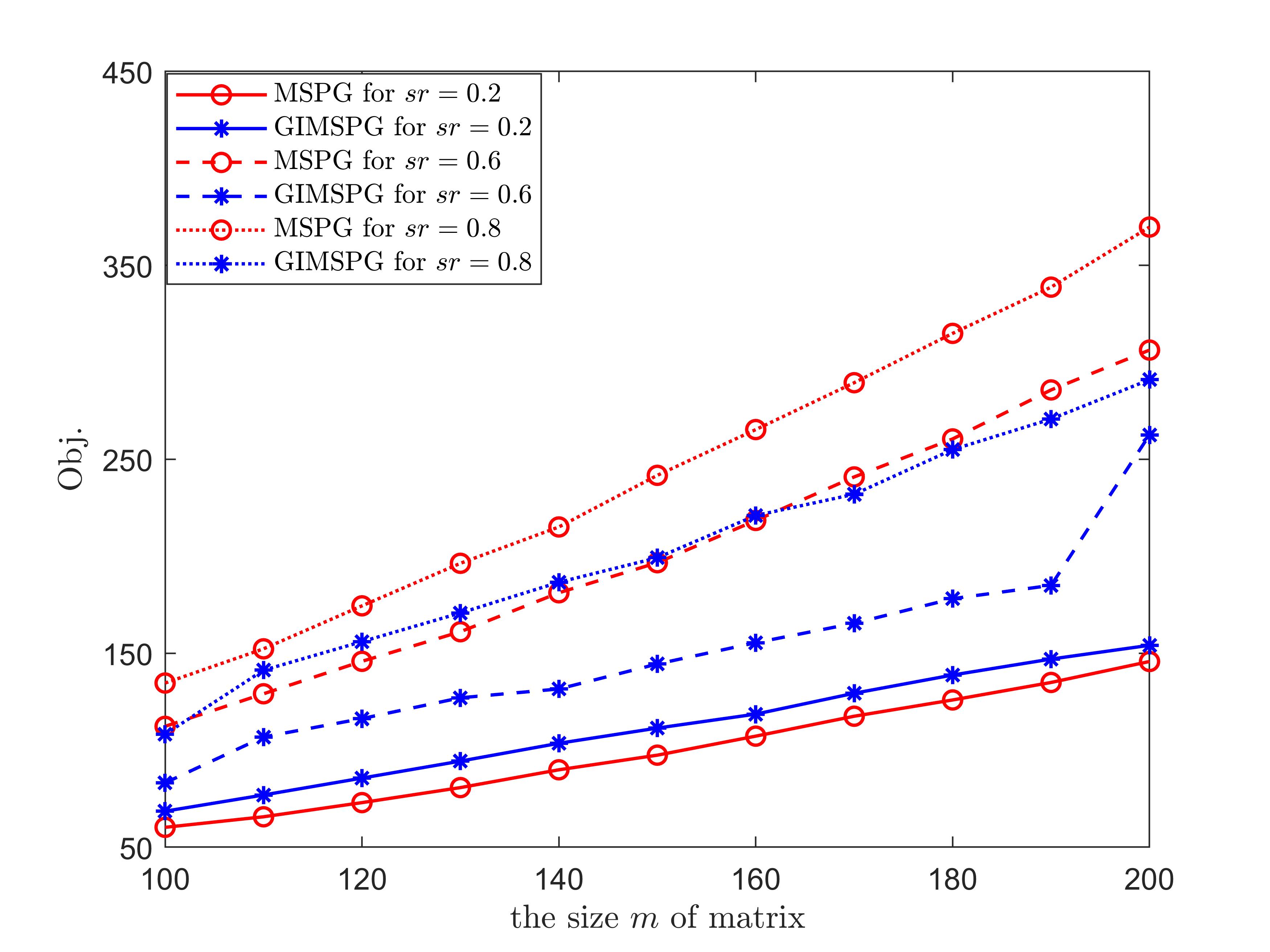}
  		\end{minipage}	
  		\centering
  		\begin{minipage}[t] {0.40\textwidth}
  			\includegraphics[scale=0.05]{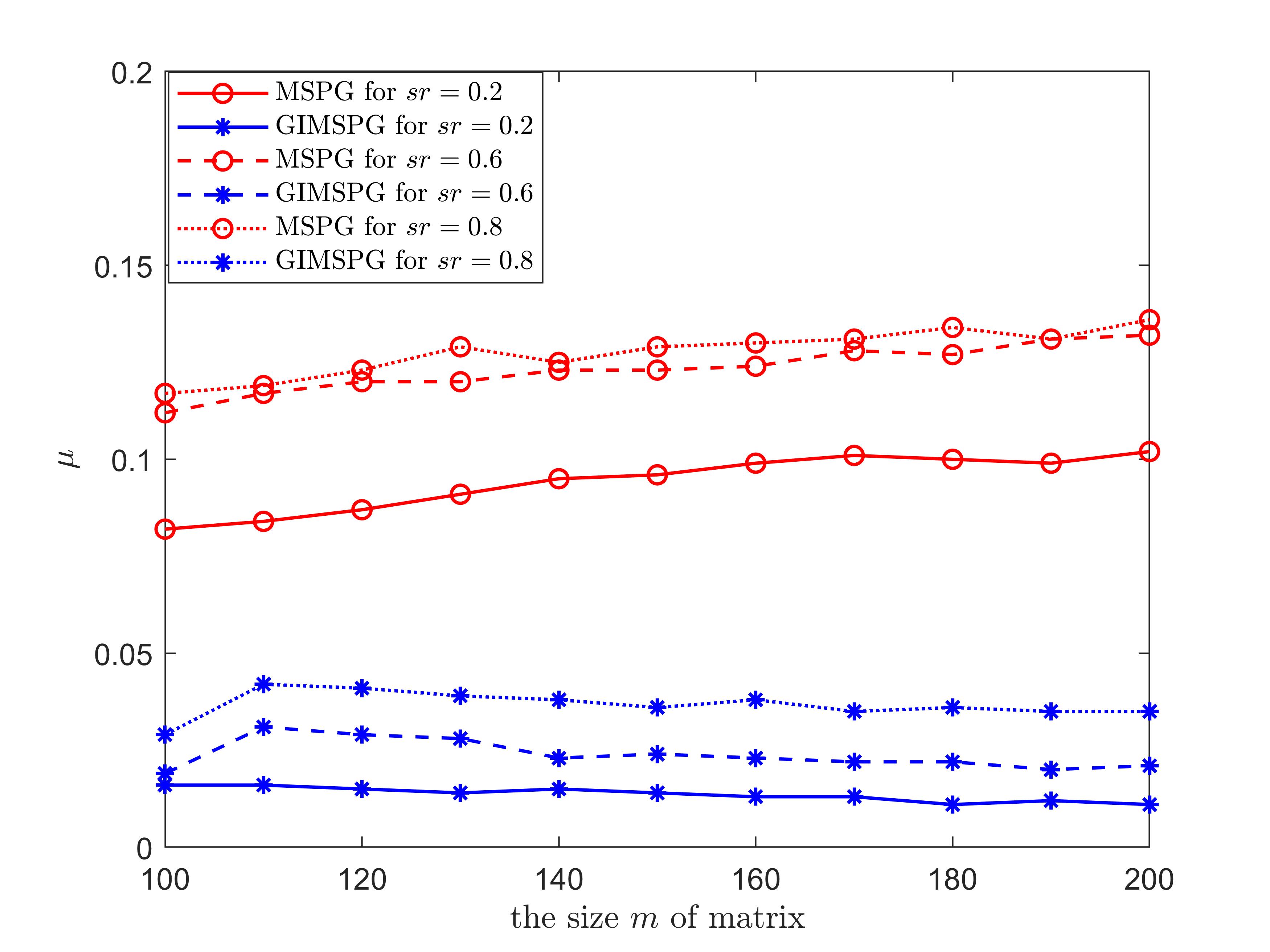}
  		\end{minipage}
  		\centering\caption{The last value of the objective function (\ref{1.2}) and smoothing factor $\mu$ for $sr=0.2,0.6,0.8$}
  		\label{fig2}
  	\end{figure}

\begin{table}[htb]
 	\centering
	\caption{Numerical results of the  random matrix problem for $sr=0.8$}
	\label{tab0}
	\setlength{\tabcolsep}{0.8mm}{
		\begin{tabular}{ccccccccccc} 
			\toprule
			\multirow{2}{*}{$m$}  &  \multicolumn{5}{l}{RMSE}& \multicolumn{5}{l}{$T$} \\
			\cmidrule(lr){2-6}\cmidrule(lr){7-11}
			&GIMSPG&MSPG& VBMFL1& FPCA&SVT & GIMSPG &MSPG& VBMFL1& FPCA&SVT  \\
		\midrule
		100   &0.030&0.031 &0.051&0.551 &0.242
		&0.32&1.01 &3.33&5.13&3.70	\\
		
		{110}	&0.029&0.030 &0.045&0.472&0.076
		&0.33&1.21&3.53&4.97&4.03
		\\
		{120}	&0.029&0.029&0.038&0.352&0.049
		&0.33&1.26&3.56&5.55&4.31
		\\
		{130}	&0.028&0.029&0.036&0.289&0.036
		&0.45&1.46&3.84&6.44&4.32
		\\
		{140} & 0.028&0.028&0.035&0.284	&0.028
		&0.51&1.78&3.56&6.72&4.51
		\\
		{150}  &0.028&0.028&0.034&0.271	&0.026
		&0.55&1.85&3.58&6.72&4.75
		\\
		{160}	&0.028&0.028&0.037&0.271&0.023
		&0.62&2.21&4.06&7.74&4.52
		\\
		{170}	&0.028&0.027&0.034&0.230&0.022
		&0.78&2.66&4.30&8.28&4.70
		\\ 	
		{180}	&0.027&0.026&0.032&0.209&0.022
		&0.84&2.89&4.16&8.86&4.41
		\\ 	
		{190}	&0.025&0.026&0.031&0.201&0.021
		&0.90&3.19&4.59&9.72&4.31
		\\ 	
		{200}	&0.025&0.025&0.030&0.196&0.020
		&1.04&3.47&4.94&10.03&4.84
		\\ 	
		\bottomrule
	\end{tabular}}
\end{table}

\begin{table}[htpb]
	 	\centering
	\caption{Numerical results of the  random matrix problem for $sr=0.6$ }
	\label{tab1}
\setlength{\tabcolsep}{0.8mm}{
	\begin{tabular}{ccccccccccc} 
		\toprule
		\multirow{2}{*}{$m$}  &  \multicolumn{5}{l}{RMSE}& \multicolumn{5}{l}{$T$} \\
		\cmidrule(lr){2-6}\cmidrule(lr){7-11}
		& GIMSPG&MSPG& VBMFL1& FPCA&SVT & GIMSPG &MSPG& VBMFL1& FPCA&SVT  \\
		\midrule
		100   &0.036&0.036 &0.124&1.640 &1.860
		&0.42&1.40 &5.11&4.71&5.23	\\
		
		{110}	&0.033&0.033 &0.080&1.360&1.380
		&0.57&1.50&5.19&5.09&4.84
		\\
		{120}	&0.033&0.033&0.062&0.870&1.324
		&0.85&1.70&5.38&5.94&5.70
		\\
		{130}	&0.031&0.031&0.054&0.610&0.972
		&0.91&1.88&5.44&6.05&8.41
		\\
		{140} & 0.030&0.031&0.051&0.500	&0.707
		&1.09&2.10&5.52&6.67&7.27
		\\
		{150}  &0.029&0.029&0.047&0.437	&0.552
		&1.87&2.32&6.21&7.17&7.47
		\\
		{160}	&0.028&0.028&0.042&0.396&0.368
		&1.90&2.41&6.36&8.09&8.55
		\\
		{170}	&0.027&0.028&0.041&0.357&0.219
		&1.99&2.99&6.59&8.64&9.01
		\\ 	
		{180}	&0.027&0.027&0.038&0.326&0.163
		&0.84&3.10&7.16&10.06&9.03
		\\ 	
		{190}	&0.026&0.027&0.035&0.310&0.127
		&2.14&4.11&7.51&10.23&10.06
		\\ 	
		{200}	&0.026&0.027&0.034&0.289&0.068
		&2.38&4.37&8.14&10.55&10.84
		\\ 	
		\bottomrule
	\end{tabular}}
\end{table}
\begin{table}[htpb]
	 	\centering
	\caption{Numerical results of the  random matrix problem for $sr=0.2$ }
	\label{tab2}
	\setlength{\tabcolsep}{0.8mm}{
	\begin{tabular}{ccccccccccc} 
		\toprule
		\multirow{2}{*}{$m$}  &  \multicolumn{5}{l}{RMSE}& \multicolumn{5}{l}{$T$} \\
		\cmidrule(lr){2-6}\cmidrule(lr){7-11}
		& GIMSPG&MSPG& VBMFL1& FPCA&SVT & GIMSPG &MSPG& VBMFL1& FPCA&SVT  \\
		\midrule
		100   &0.070&0.075&5.649&5.770 &6.536
		&0.69&1.40 &13.42&2.38&8.17	\\
		
		{110}	&0.071&0.074 &5.496&5.660&6.497
		&0.80&1.50&16.59&2.64&995
		\\
		{120}	&0.069&0.073&5.486&5.540&6.563
		&0.94&1.60&17.38&3.27&10.41
		\\
		{130}	&0.066&0.073&5.456&5.300&6.618
		&1.08&1.70&18.33&3.27&13.51
		\\
		{140} & 0.066&0.075&5.446&5.209	&6.255
		&1.50&1.98&20.84&3.73&15.41
		\\
		{150}  &0.066&0.072&5.442&5.270	&6.230
		&1.66&2.19&21.84&4.06&15.84
		\\
		{160}	&0.065&0.071&5.508&5.230&6.210
		&1.80&2.57&25.67&3.99&27.09
		\\
		{170}	&0.065&0.068&5.309&5.050&6.194
		&2.30&2.80&26.77&4.31&23.89
		\\ 	
		{180}	&0.064&0.069&5.212&5.010&5.984
		&2.60&3.39&29.16&5.25&29.01
		\\ 	
		{190}	&0.062&0.068&5.091&4.970&5.020
		&2.63&3.58&29.98&5.98&29.21
		\\ 	
		{200}	&0.062&0.067&4.239&4.680&5.064
		&3.20&4.82&32.83&6.66&29.77
		\\ 	
		\bottomrule
	\end{tabular} }
\end{table}

\subsection{Image inpainting}\label{4.21}

In this subsection, we aim to illustrate the efficiency of GIMSPG algorithm by
 solving a grayscale inpainting problem with non-Gaussian noise. The image inpainting problem is to fill the missing pixel values of the image at given pixel locations. In fact, the grayscale image can be seen as the matrix, and the image inpainting problem can be represented as the matrix completion problem if the matrix satisfies the property of the low rank.
In our experiments, two different grayscale images are considered which are from the USC-SIPI image database
 \footnote{\url{http://sipi.usc.edu/database/.}}. In detail, they are  ``Chart''  and ``Ruler". ``Chart'' is an image with
$256\times 256$ pixels and rank $r=185$. ``Ruler'' is an image with $512\times512$ with rank $r=67$.
To evaluate the performance of the GIMSPG algorithm under the non-Gaussian noise, the peak signal-to-noise ratio (PSNR) is considered as an evaluation metric which is defined as follows:
$$\mathrm{PSNR}: =10\mathrm{log}_{10} (\frac{mn}{\Vert \bar{X}-M \Vert^2_F}).
$$

 We solve the image inpainting problem under three different sample ratios $sr=0.2,\,sr=0.6,\,sr=0.8$. From the results in Table \ref{tab3}, we observe that the last value $\mu$ of smoothing factor of GIMSPG algorithm is less than MSPG algorithm in each case. The used time $T$ of GIMSPG  is the  least in these algorithms. The PSNR of GIMSPG algorithm is  higher than MSPG for $sr=0.2$ and obviously higher than algorithm in all the cases. From Figures \ref{fig4}, \ref{fig5},
  one can see that the recoverability of  GIMSPG algorithm is slightly better than that MSPG algorithm espeically in the lower sample rate and apparently better than others.

\begin{figure*}[!ht]
	\centering
	\begin{minipage}[b]{1\linewidth}
		\subfloat[Observed]{
			\begin{minipage}[b]{0.14\linewidth}
				\centering
				\includegraphics[width=\linewidth]{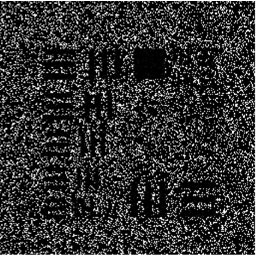}\vspace{1pt}
				\includegraphics[width=\linewidth]{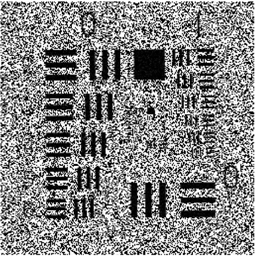}\vspace{1pt}
				\includegraphics[width=\linewidth]{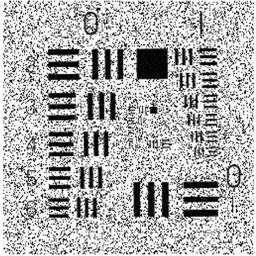}
			\end{minipage}
		}
		\hfill
		\subfloat[GIMSPG]{
			\begin{minipage}[b]{0.14\linewidth}
				\centering
				\includegraphics[width=\linewidth]{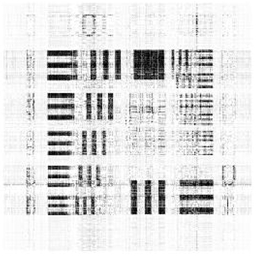}\vspace{1pt}
				\includegraphics[width=\linewidth]{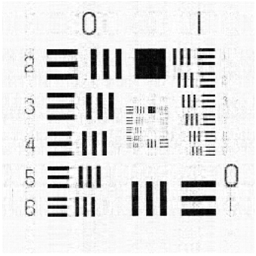}\vspace{1pt}
				\includegraphics[width=\linewidth]{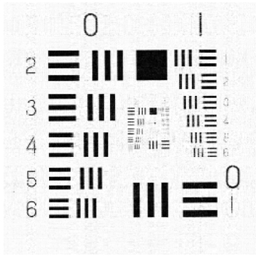}
			\end{minipage}
		}
		\hfill
		\subfloat[MSPG]{
			\begin{minipage}[b]{0.14\linewidth}
				\centering
				\includegraphics[width=\linewidth]{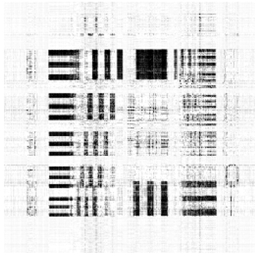}\vspace{1pt}
				\includegraphics[width=\linewidth]{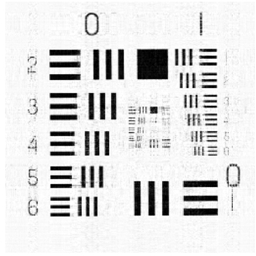}\vspace{1pt}
			
				\includegraphics[width=\linewidth]{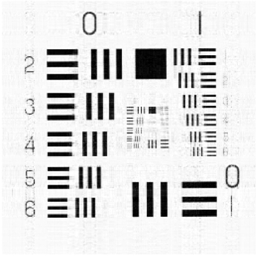}
			\end{minipage}
		}
		\hfill
		\subfloat[VBMFL1]{
			\begin{minipage}[b]{0.14\linewidth}
				\centering
				\includegraphics[width=\linewidth]{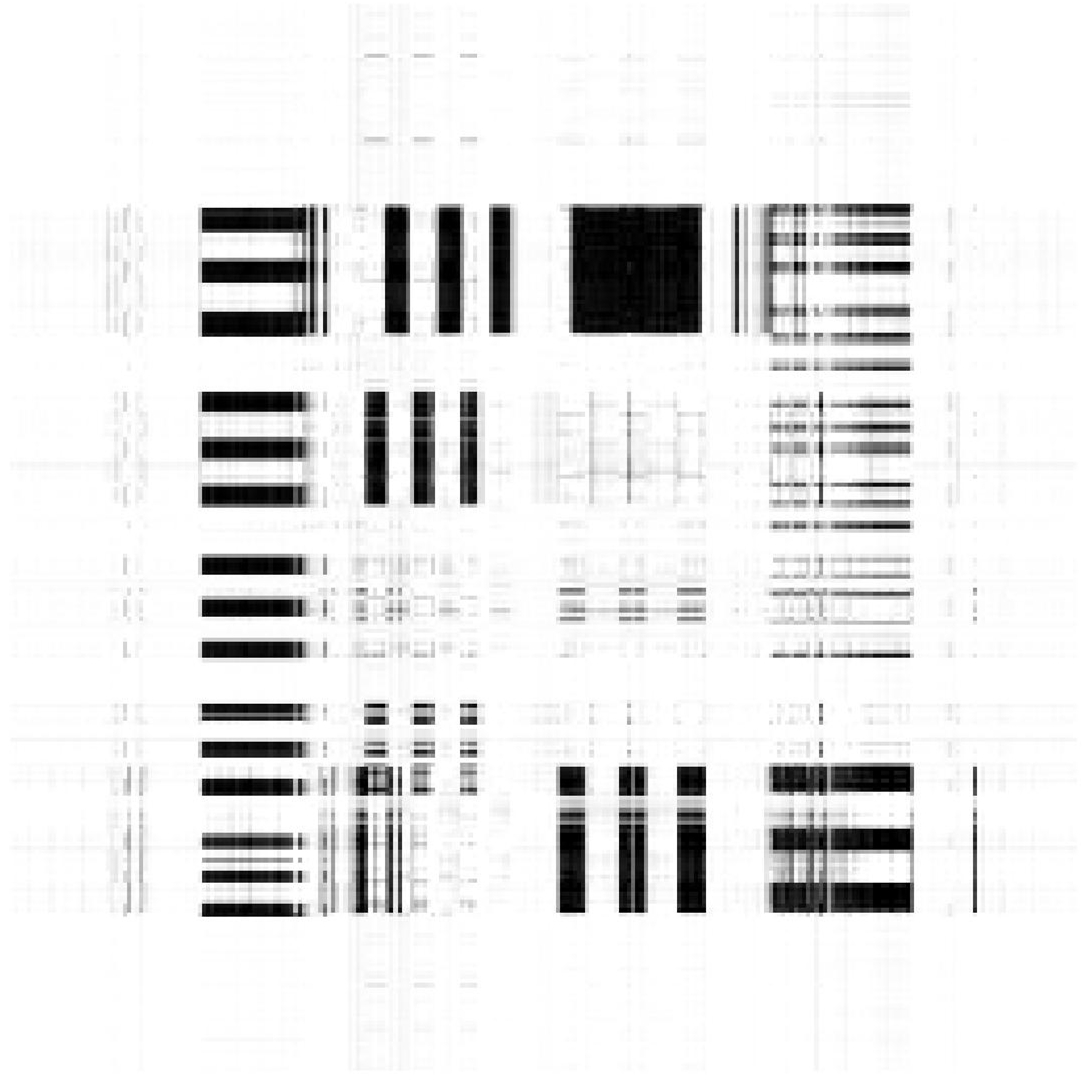}\vspace{1pt}
				\includegraphics[width=\linewidth]{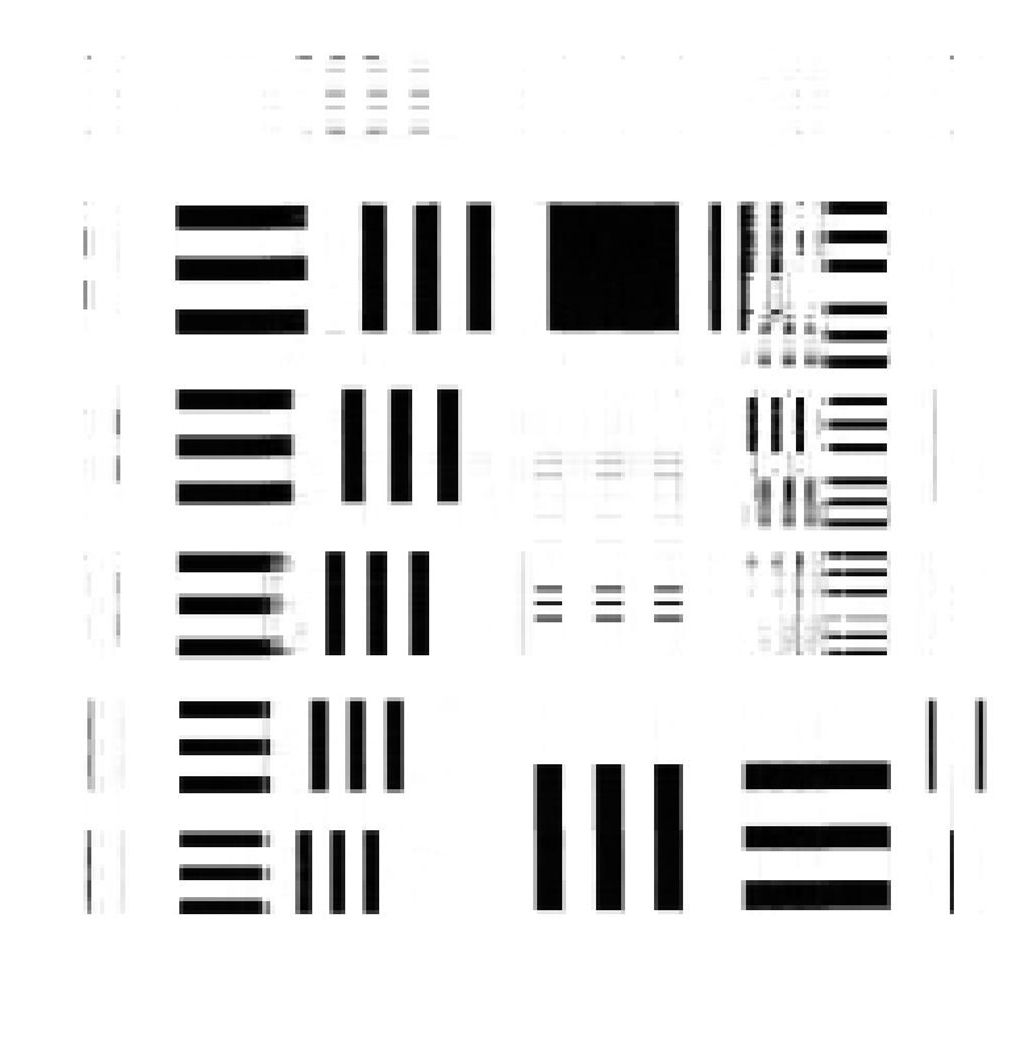}\vspace{1pt}
			
				\includegraphics[width=\linewidth]{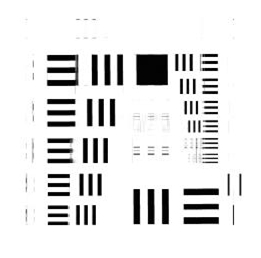}
			\end{minipage}
		}
	\hfill
	\subfloat[FPCA]{
		\begin{minipage}[b]{0.14\linewidth}
			\centering
			\includegraphics[width=\linewidth]{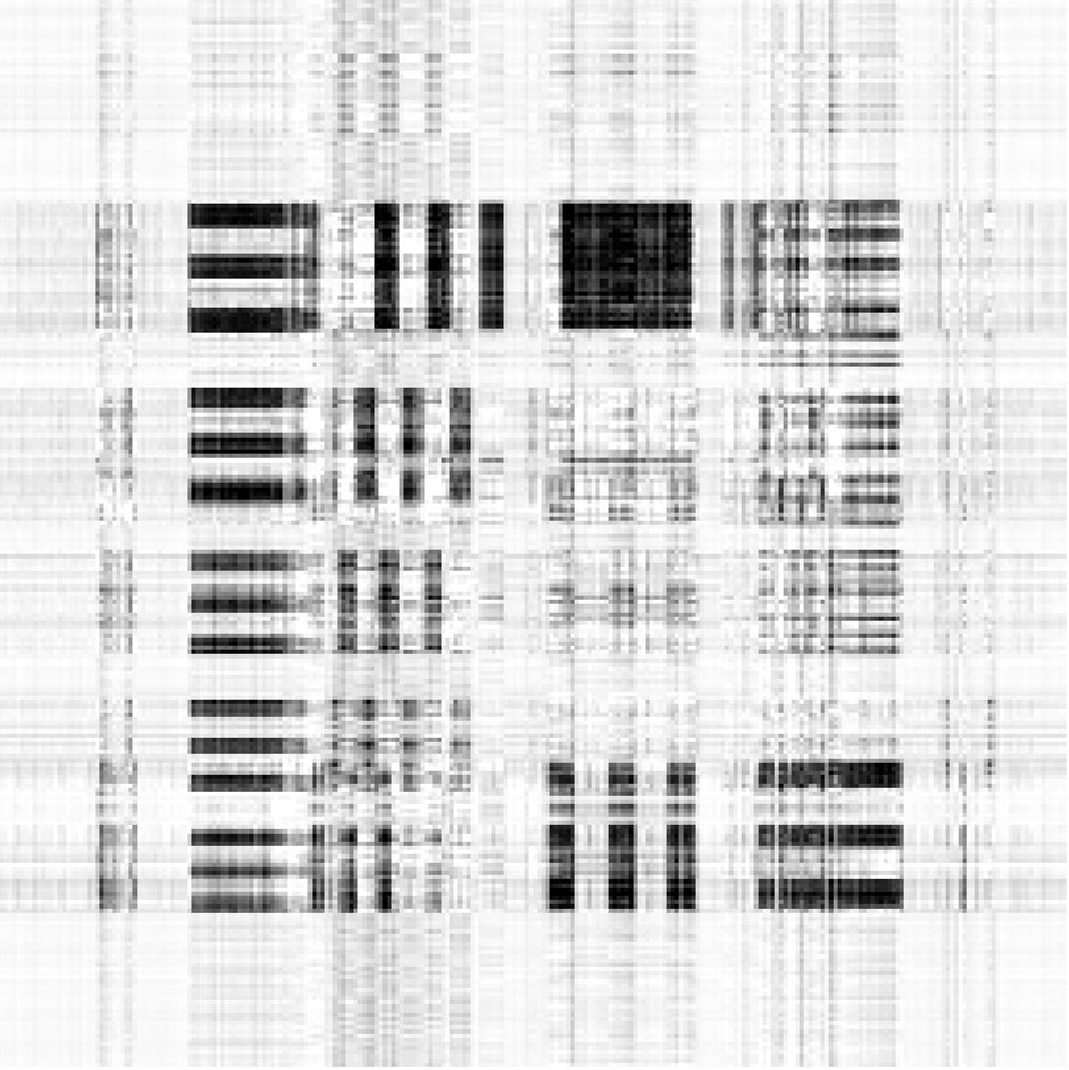}\vspace{1pt}
			\includegraphics[width=\linewidth]{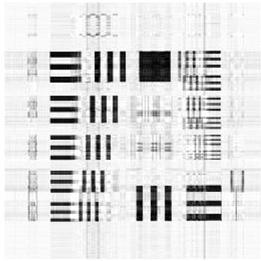}\vspace{1pt}
			
			\includegraphics[width=\linewidth]{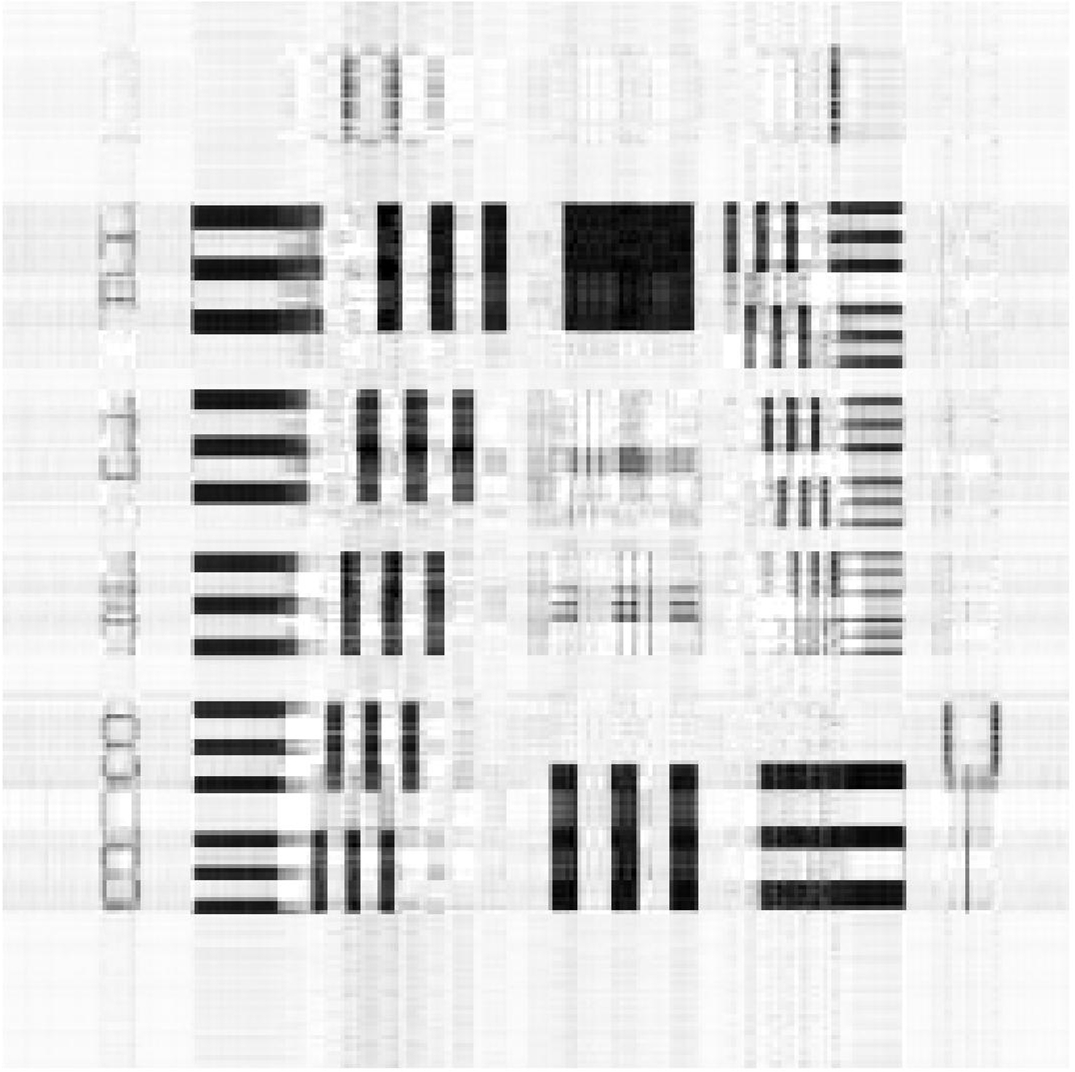}
		\end{minipage}
	}	
	\hfill
\subfloat[SVT]{
	\begin{minipage}[b]{0.14\linewidth}
		\centering
		\includegraphics[width=\linewidth]{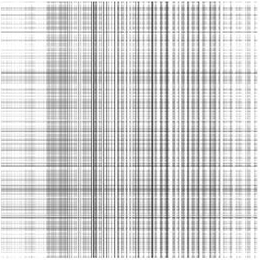}\vspace{1pt}
		\includegraphics[width=\linewidth]{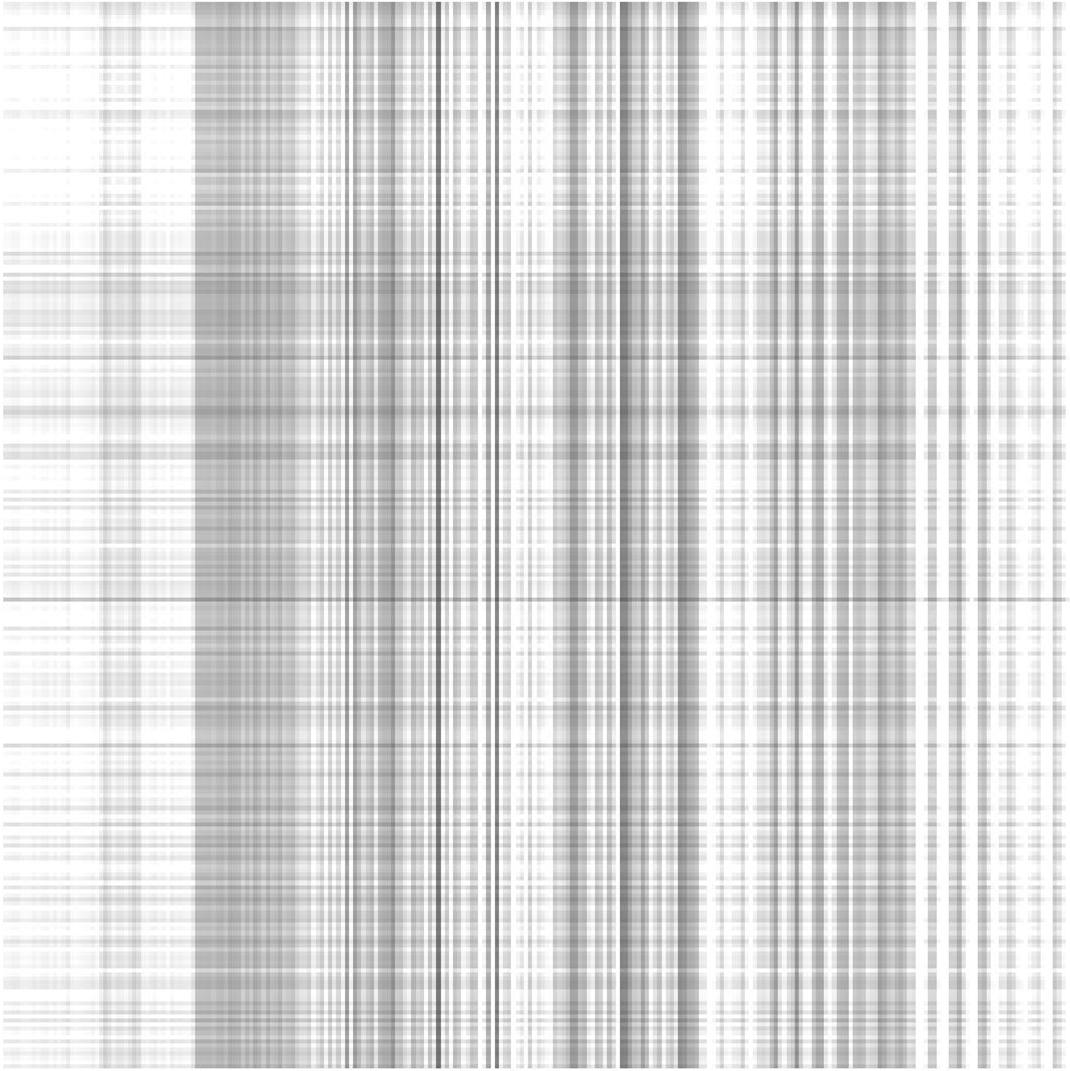}\vspace{1pt}		
		\includegraphics[width=\linewidth]{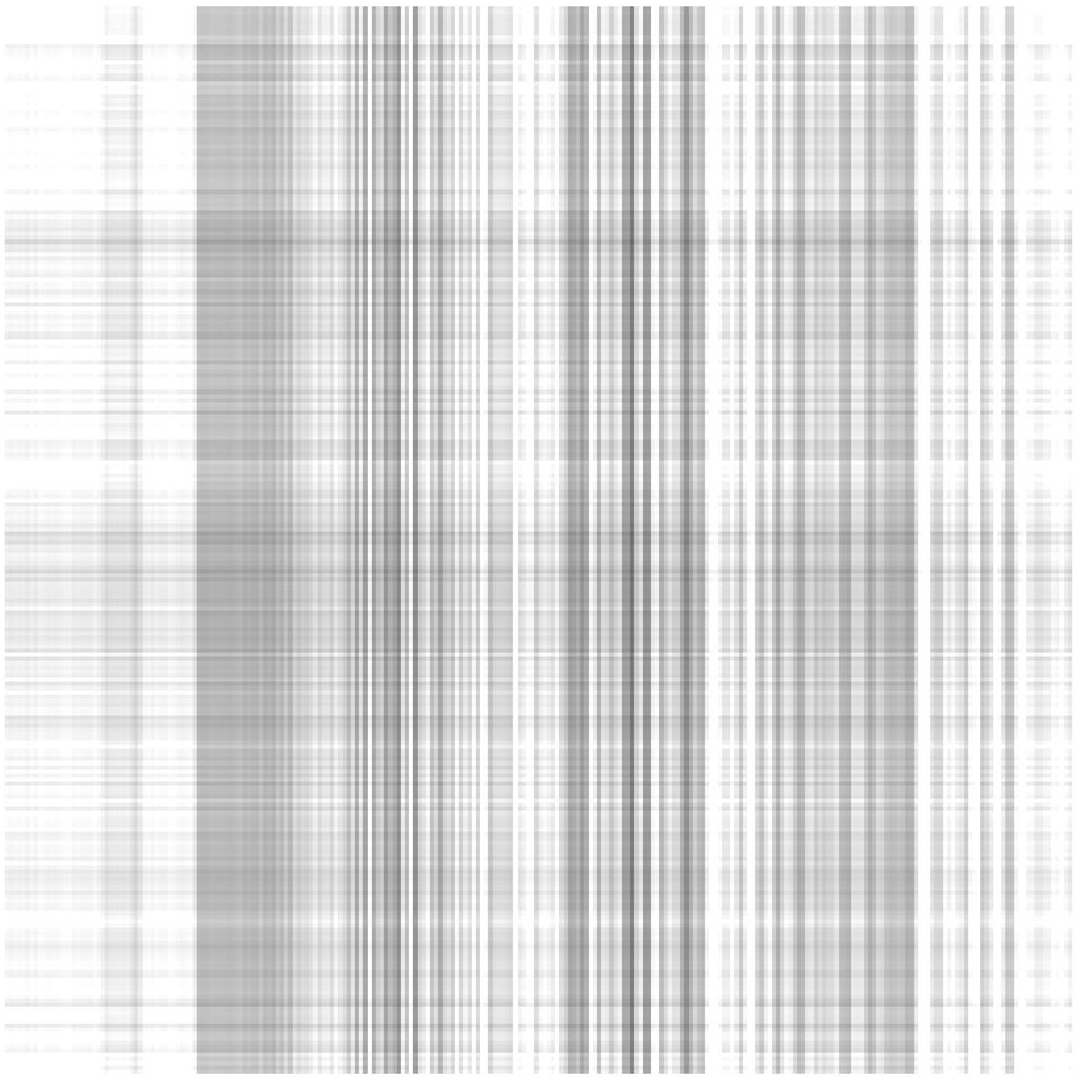}
	\end{minipage}
}	
	\end{minipage}
	\vfill
		\caption{ Results of ``Chart'' recovery. The first column is the sample images with $sr = 0.2, \,0.6,\, 0.8.$ The rest columns are the recovery images by the corresponding algorithms}
		\label{fig4}
\end{figure*}

\begin{figure*}[!ht]
	\centering
	\begin{minipage}[b]{1\linewidth}
		\subfloat[Observed]{
			\begin{minipage}[b]{0.14\linewidth}
				\centering
				\includegraphics[width=\linewidth]{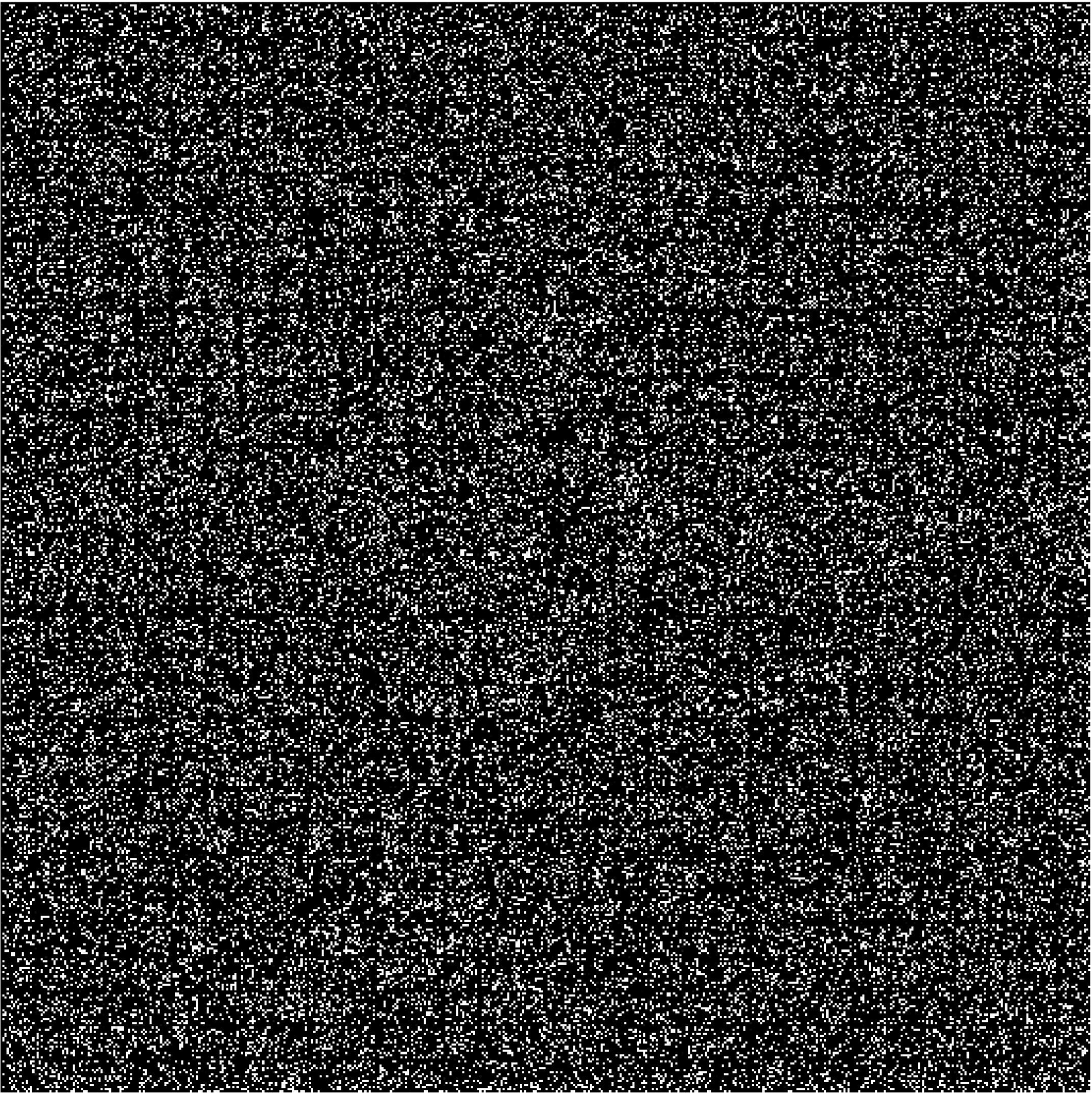}\vspace{1pt}
				\includegraphics[width=\linewidth]{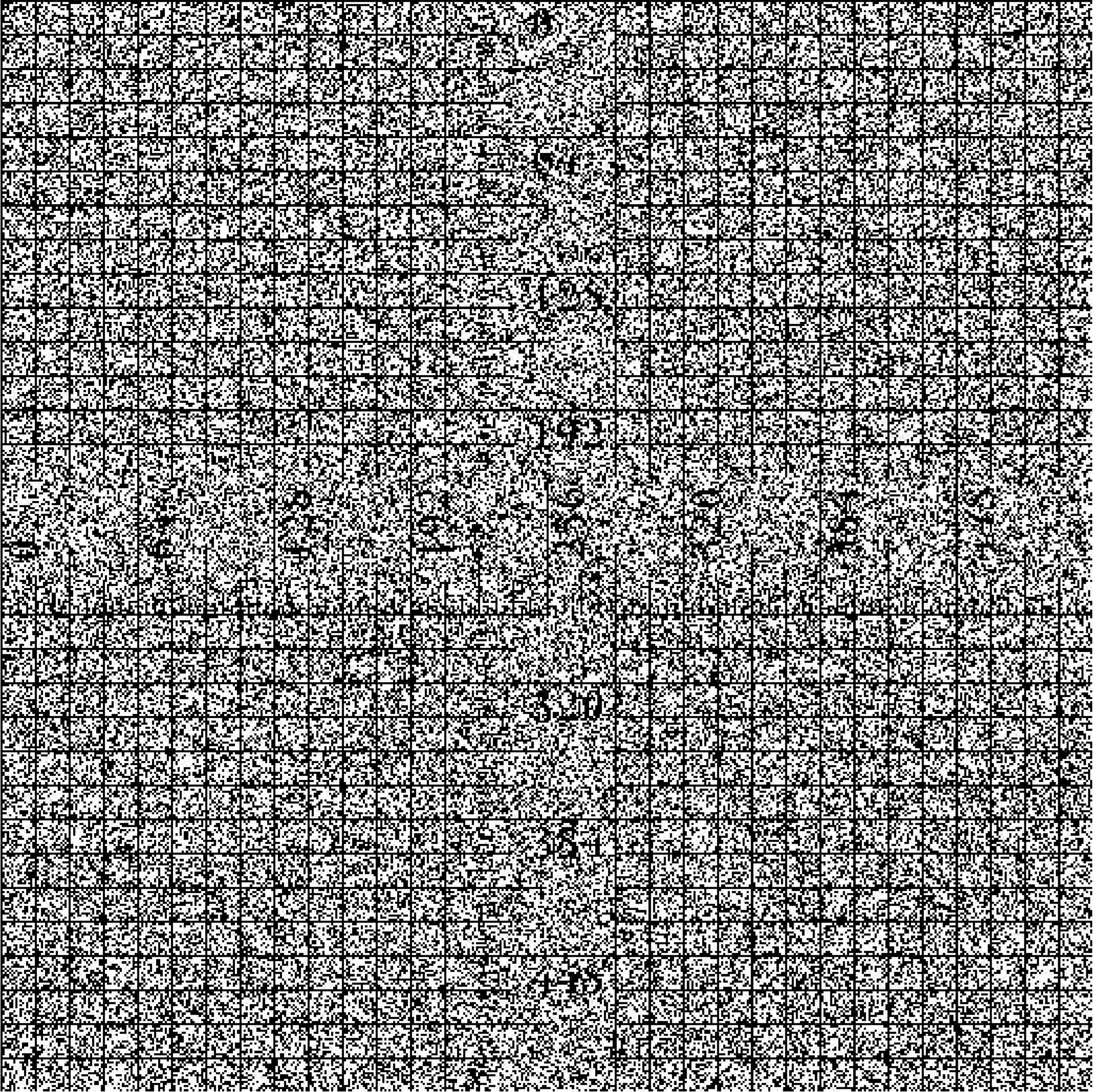}\vspace{1pt}			
				\includegraphics[width=\linewidth]{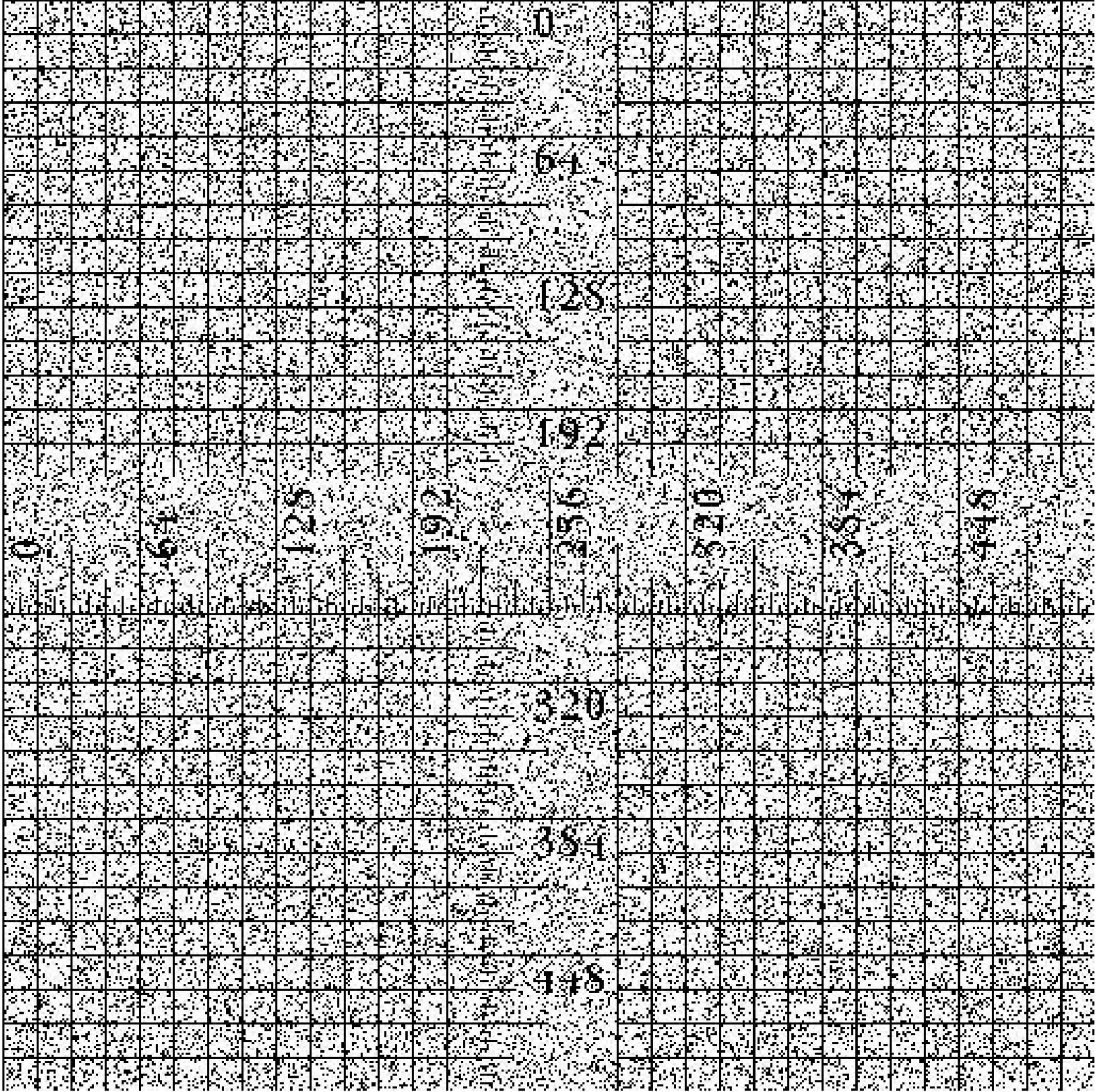}
			\end{minipage}
		}
		\hfill
		\subfloat[GIMSPG]{
			\begin{minipage}[b]{0.14\linewidth}
				\centering				
				\includegraphics[width=\linewidth]{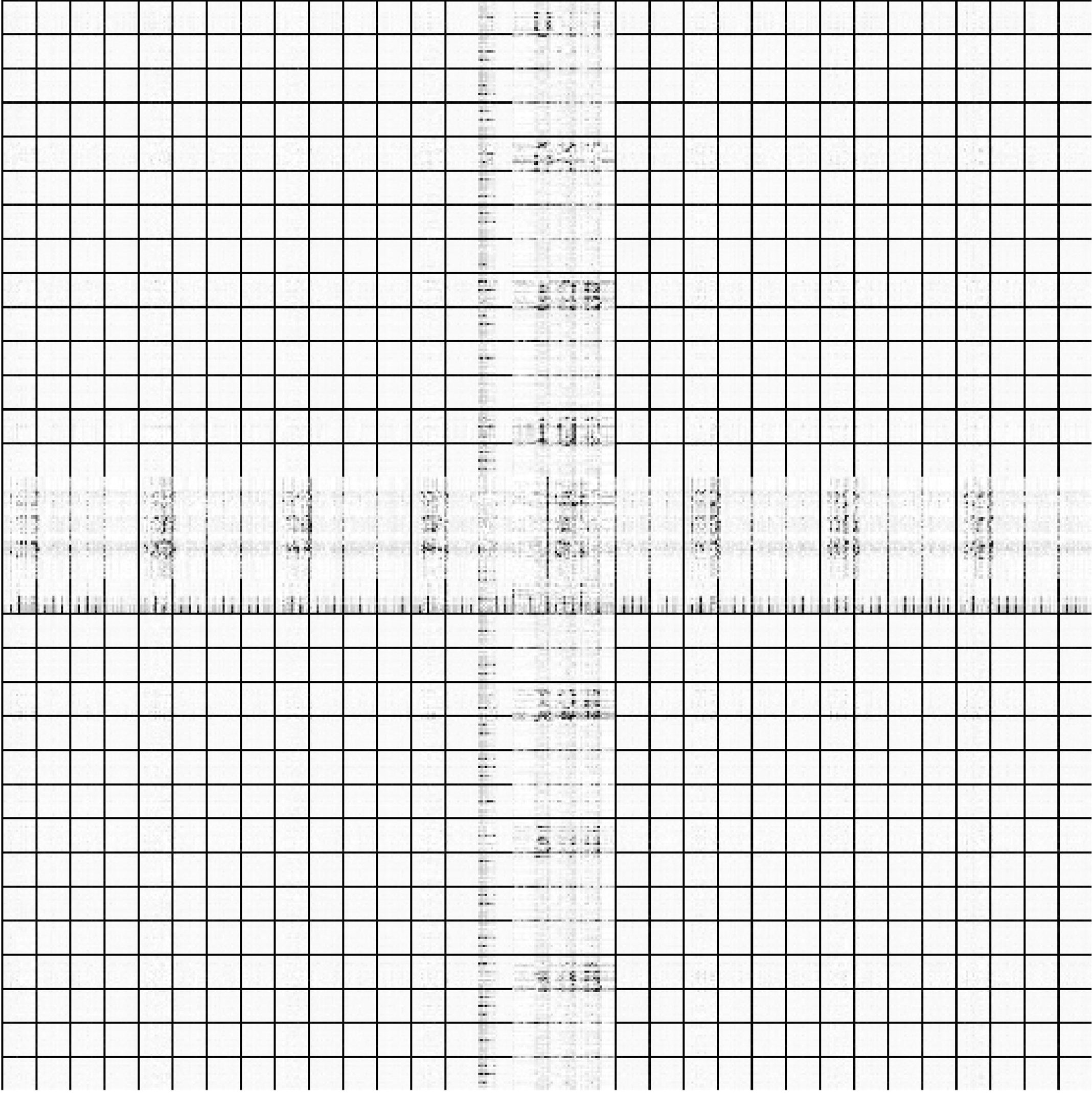}\vspace{1pt}
				\includegraphics[width=\linewidth]{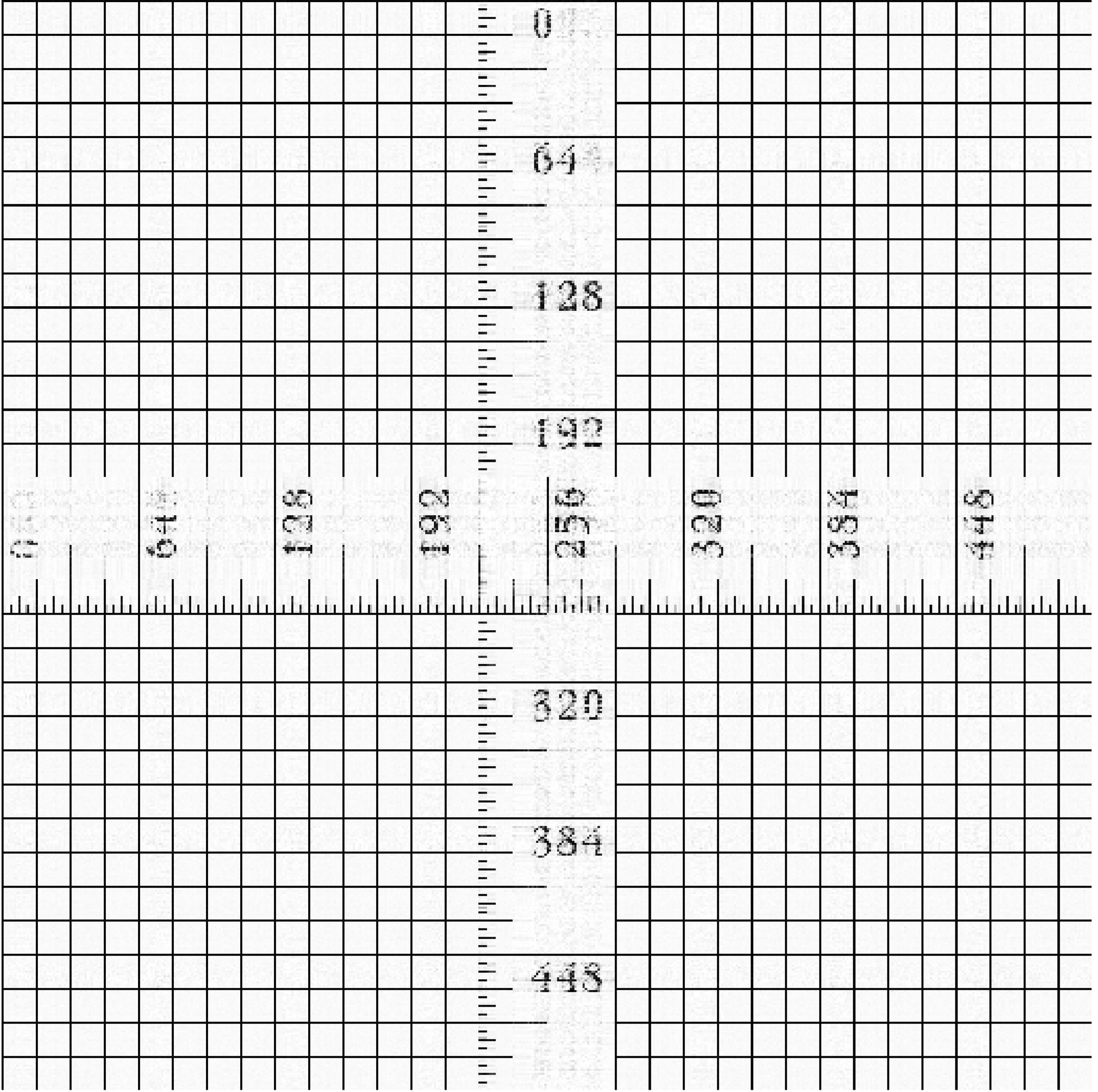}\vspace{1pt}
				\includegraphics[width=\linewidth]{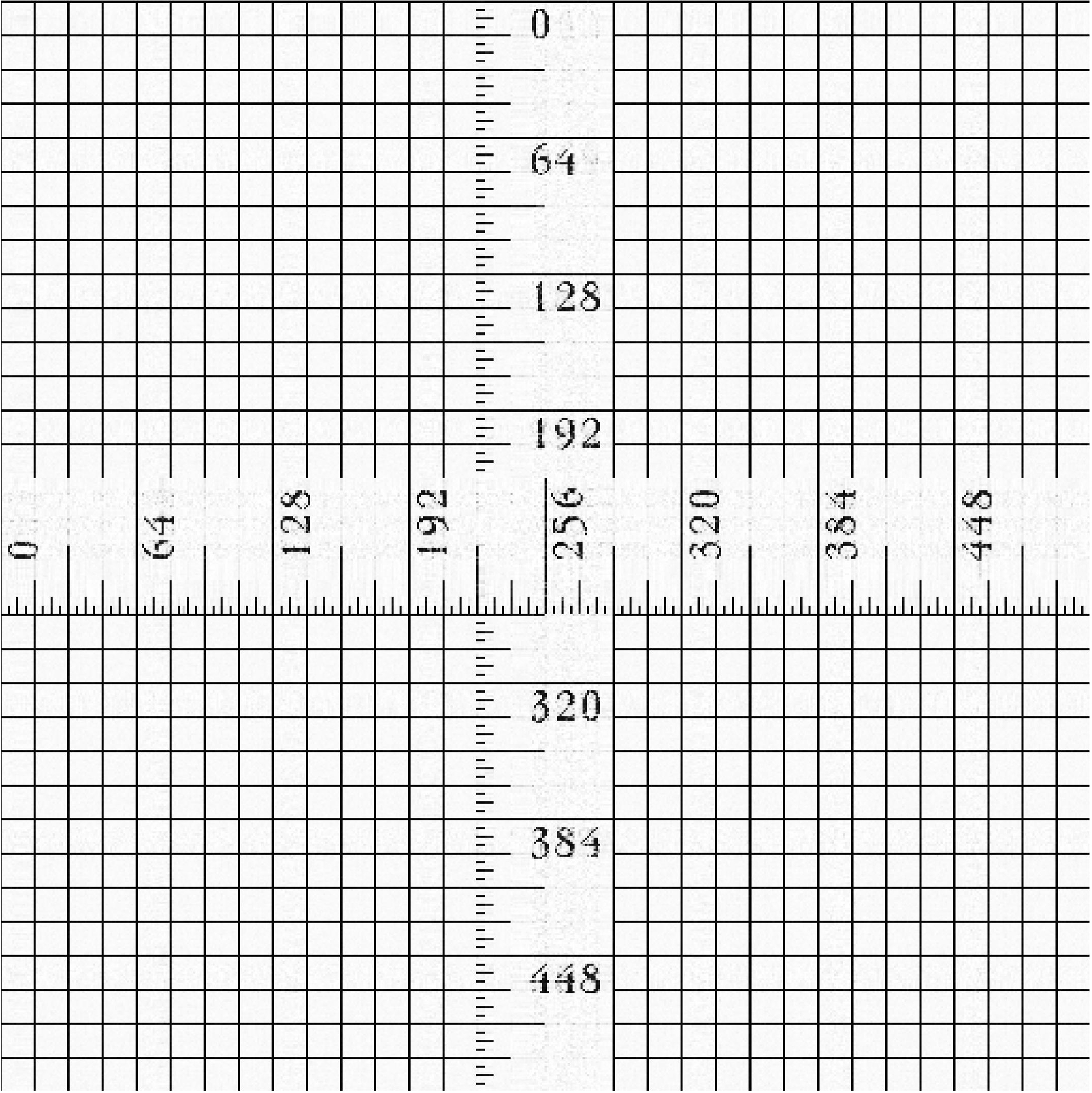}
			\end{minipage}
		}
		\hfill
		\subfloat[MSPG]{
		\begin{minipage}[b]{0.14\linewidth}
			\centering				
			\includegraphics[width=\linewidth]{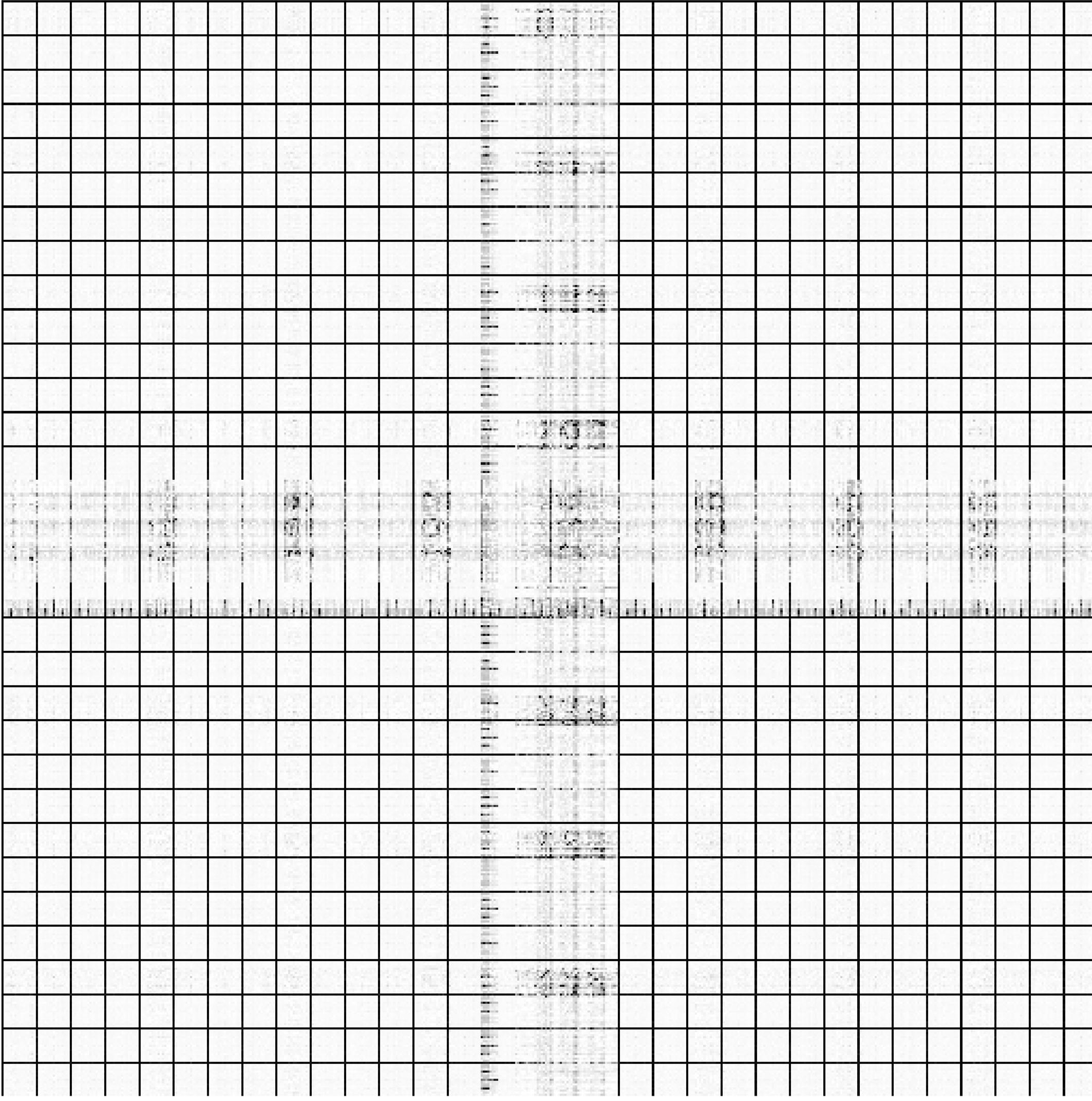}\vspace{1pt}
			\includegraphics[width=\linewidth]{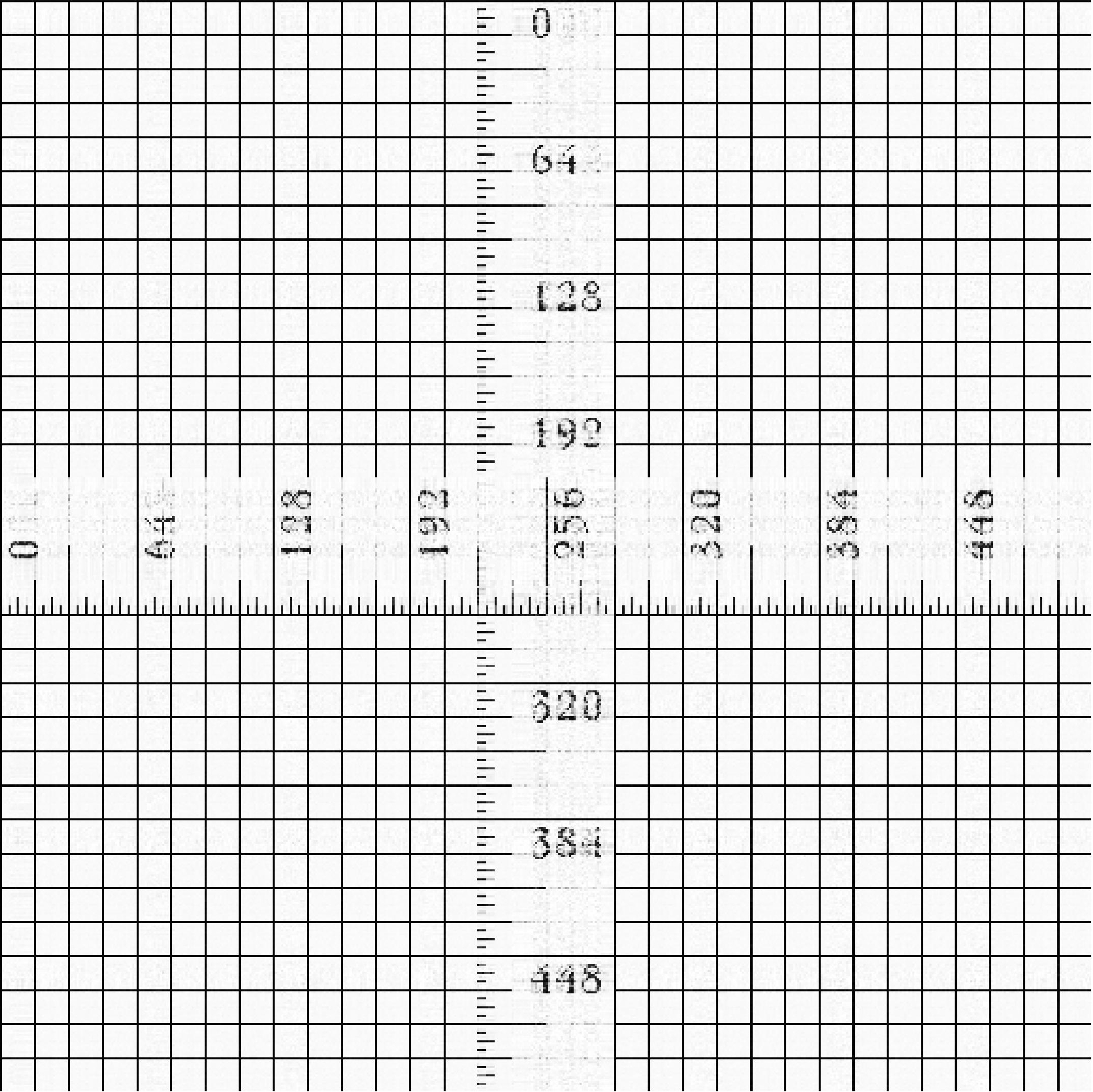}\vspace{1pt}
			\includegraphics[width=\linewidth]{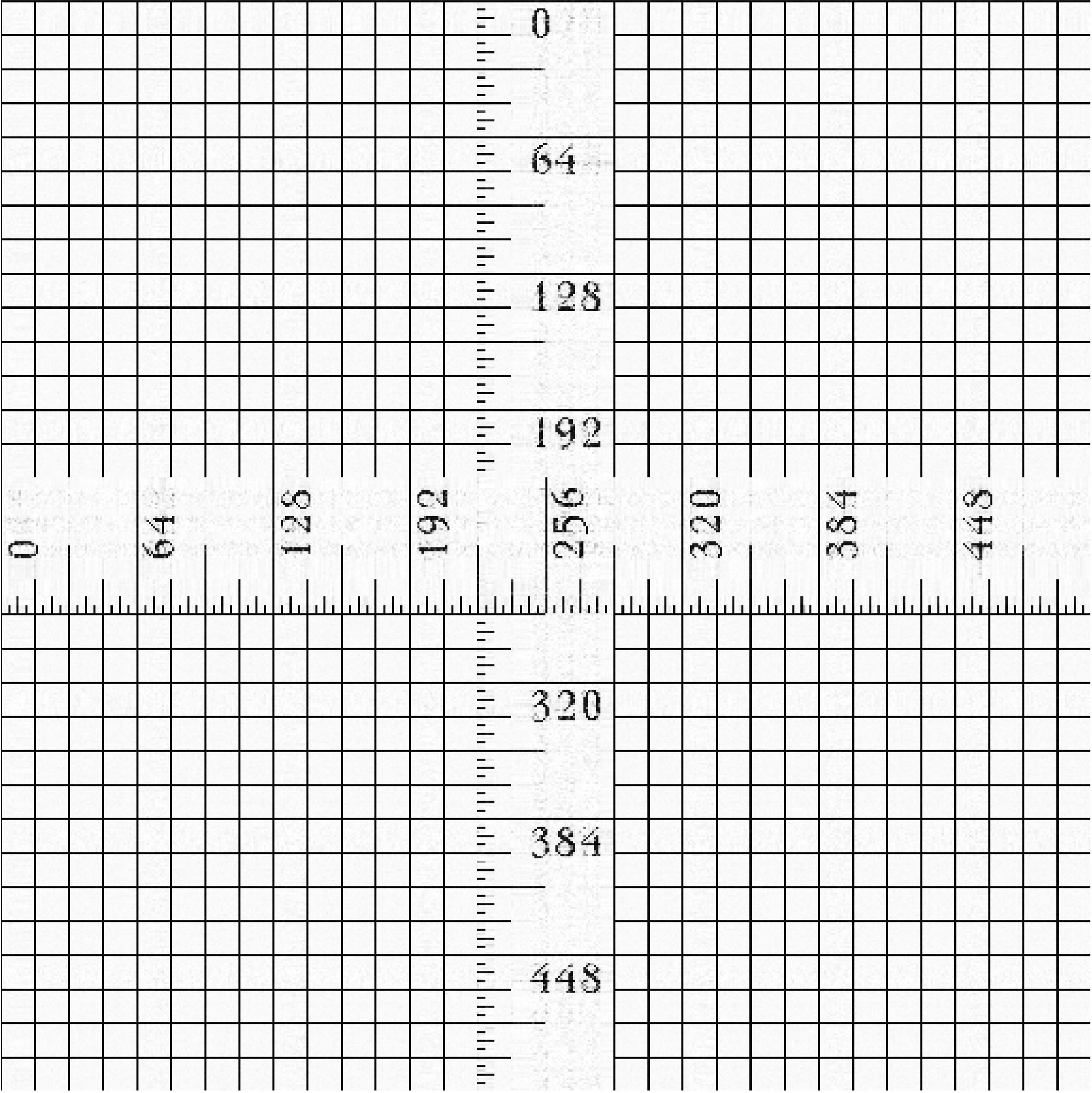}
		\end{minipage}
	}
		\hfill
			\subfloat[VBMFL1]{
			\begin{minipage}[b]{0.14\linewidth}
				\centering				
				\includegraphics[width=\linewidth]{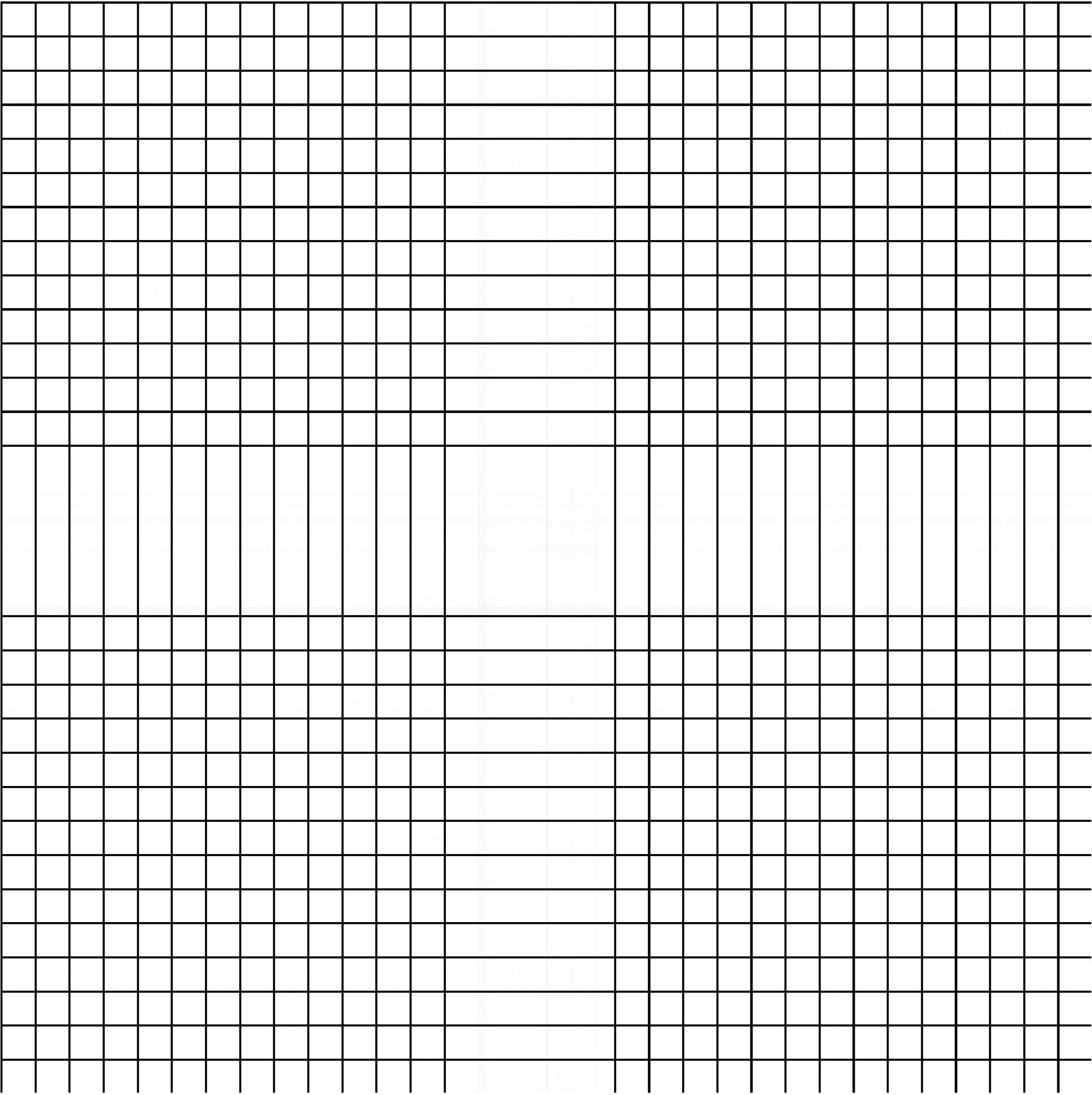}\vspace{1pt}
				\includegraphics[width=\linewidth]{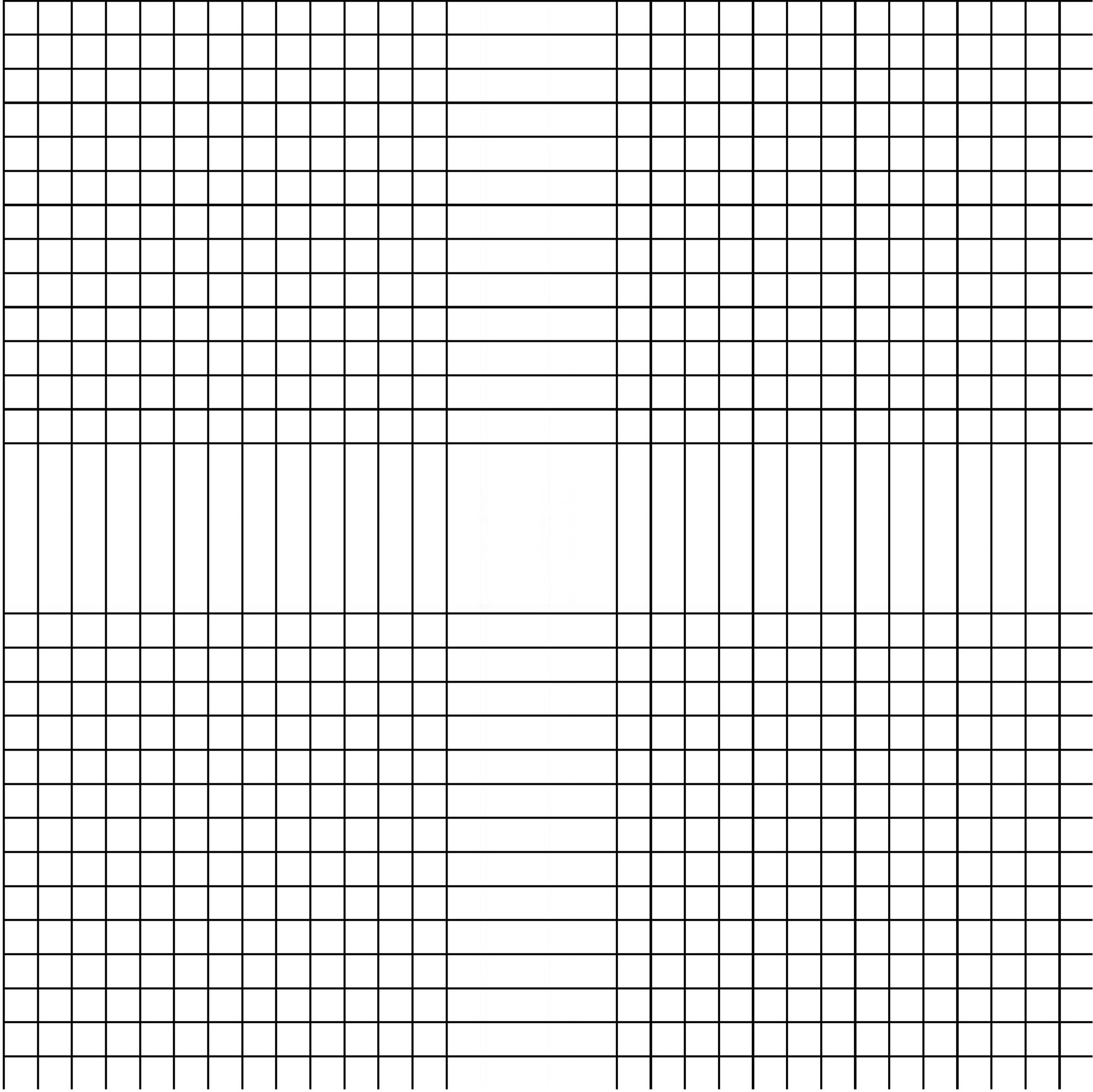}\vspace{1pt}
				\includegraphics[width=\linewidth]{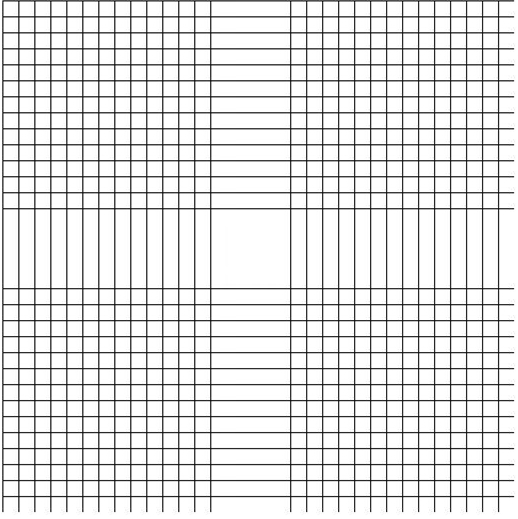}
			\end{minipage}
		}
			\hfill
		\subfloat[FPCA]{
			\begin{minipage}[b]{0.14\linewidth}
				\centering				
				\includegraphics[width=\linewidth]{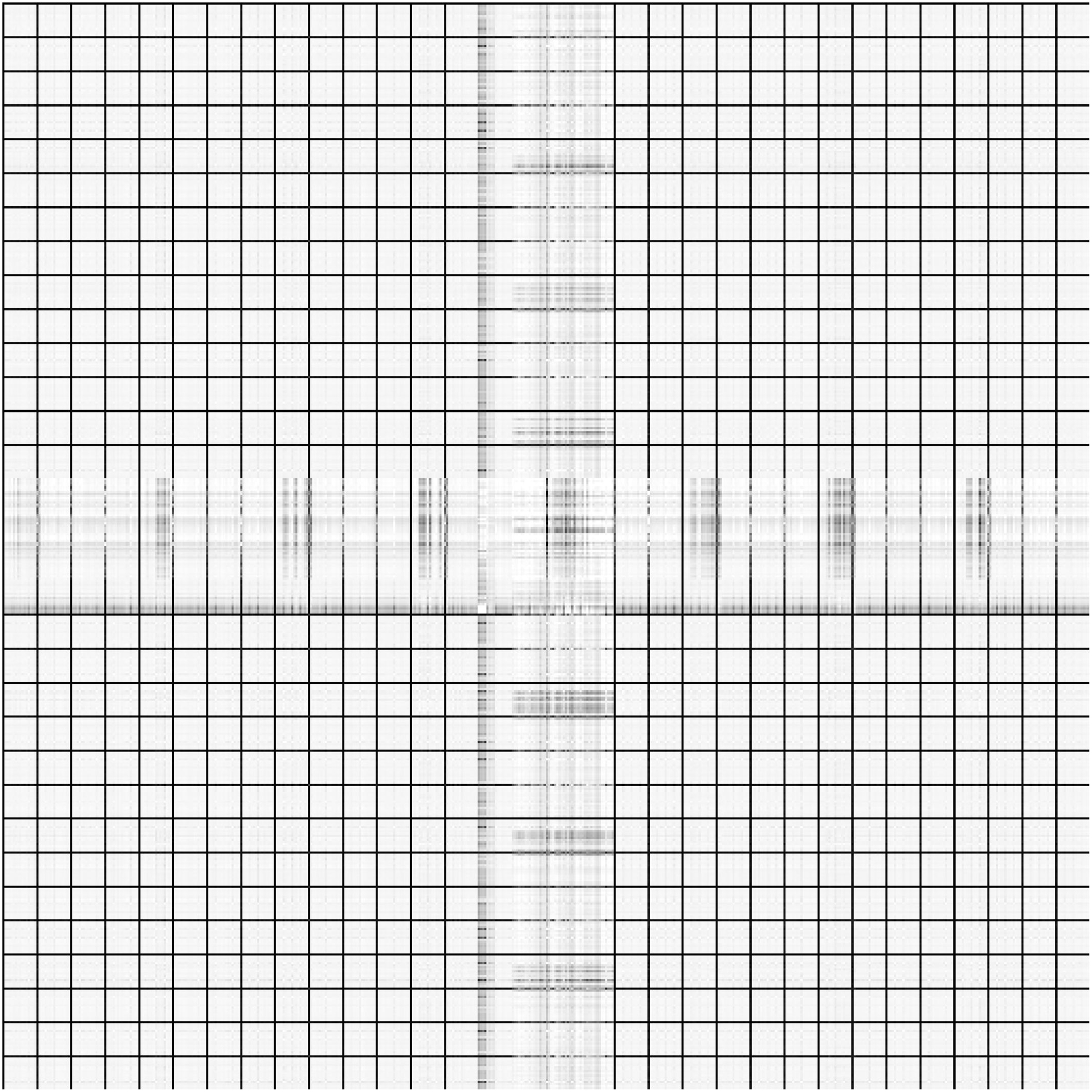}\vspace{1pt}
				\includegraphics[width=\linewidth]{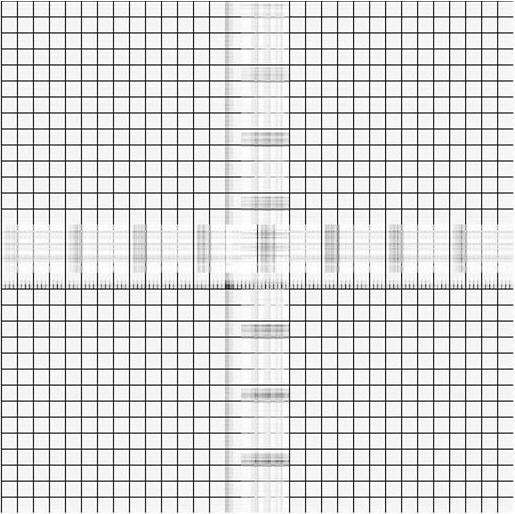}\vspace{1pt}
				\includegraphics[width=\linewidth]{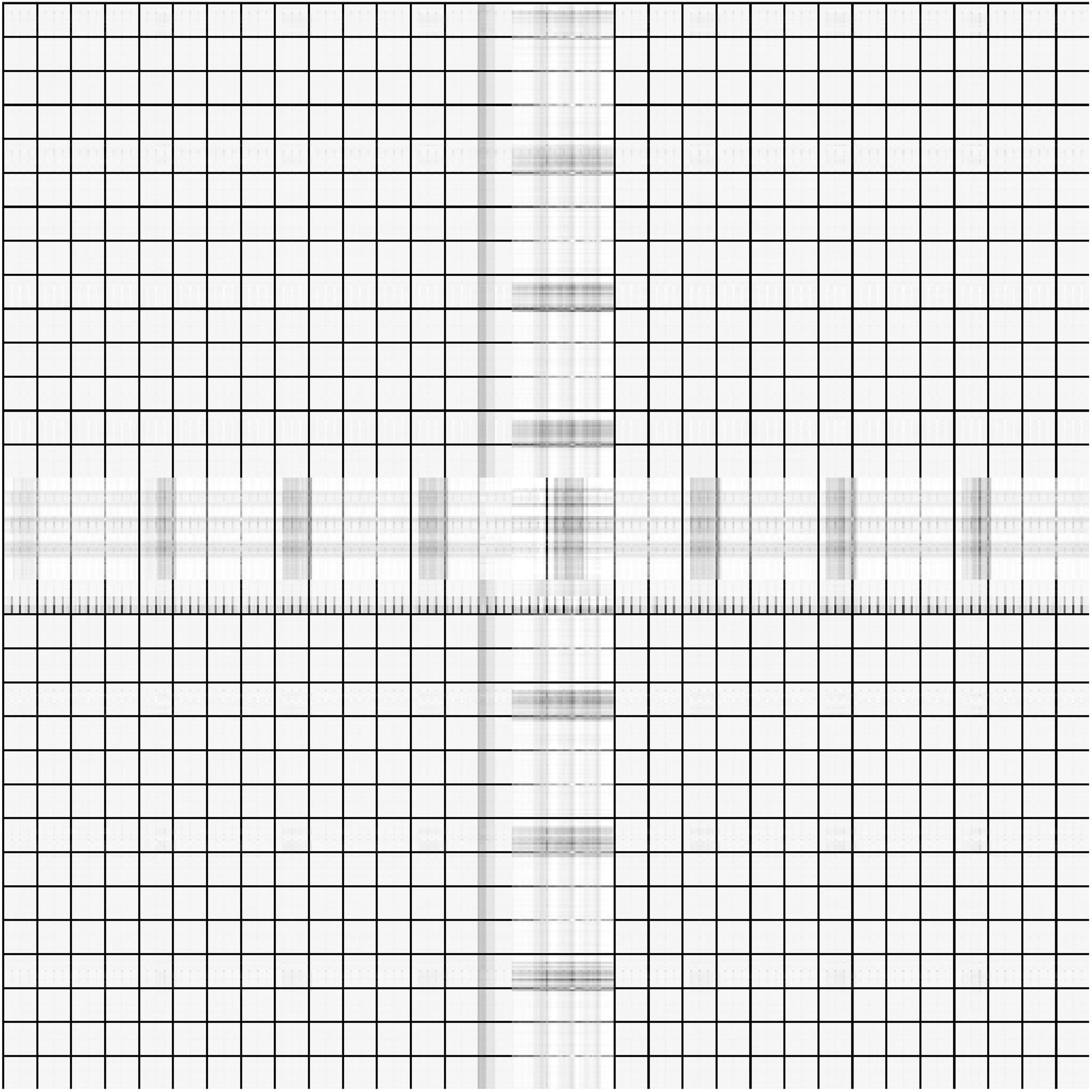}
			\end{minipage}
		}
			\hfill
		\subfloat[SVT]{
		\begin{minipage}[b]{0.14\linewidth}
			\centering				
			\includegraphics[width=\linewidth]{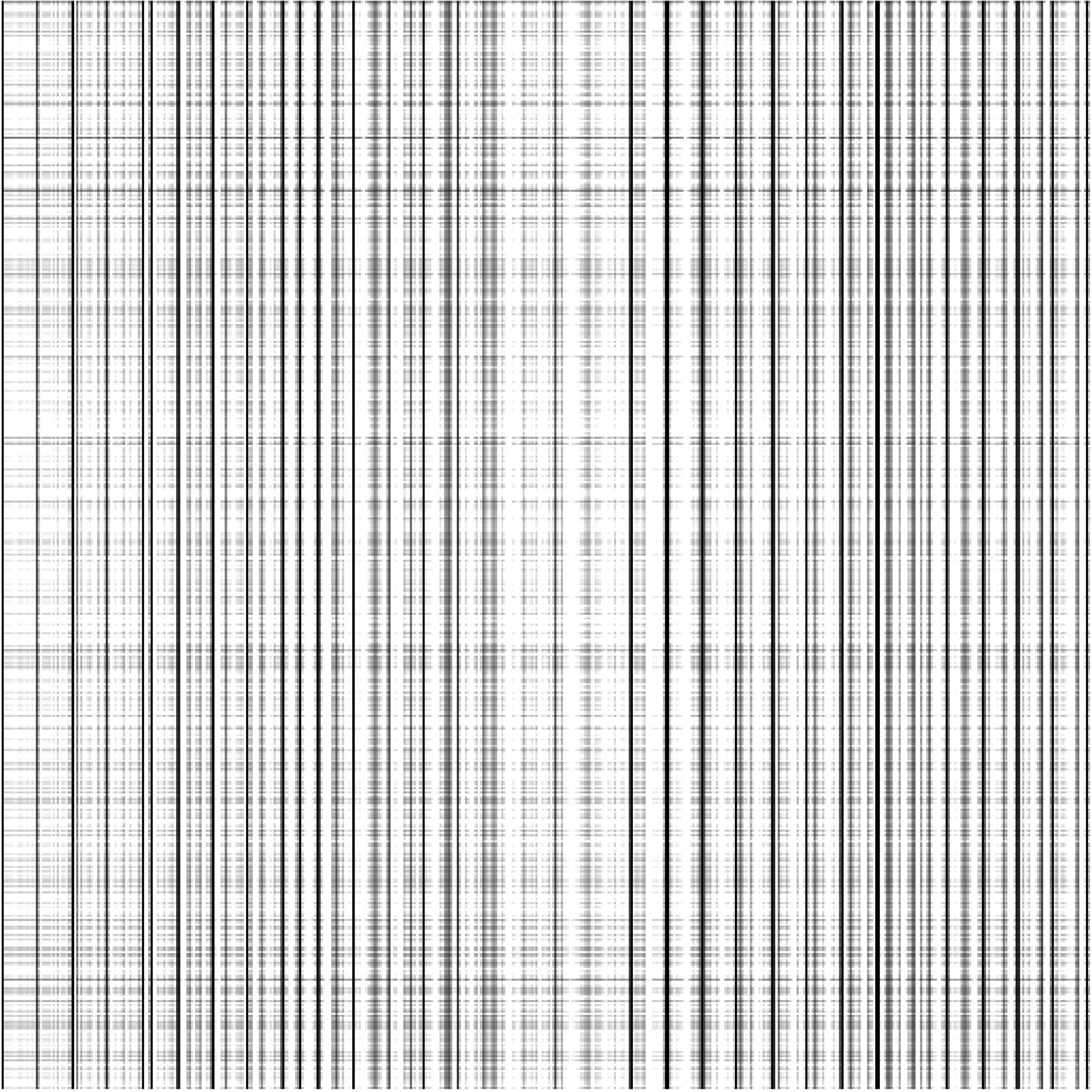}\vspace{1pt}
			\includegraphics[width=\linewidth]{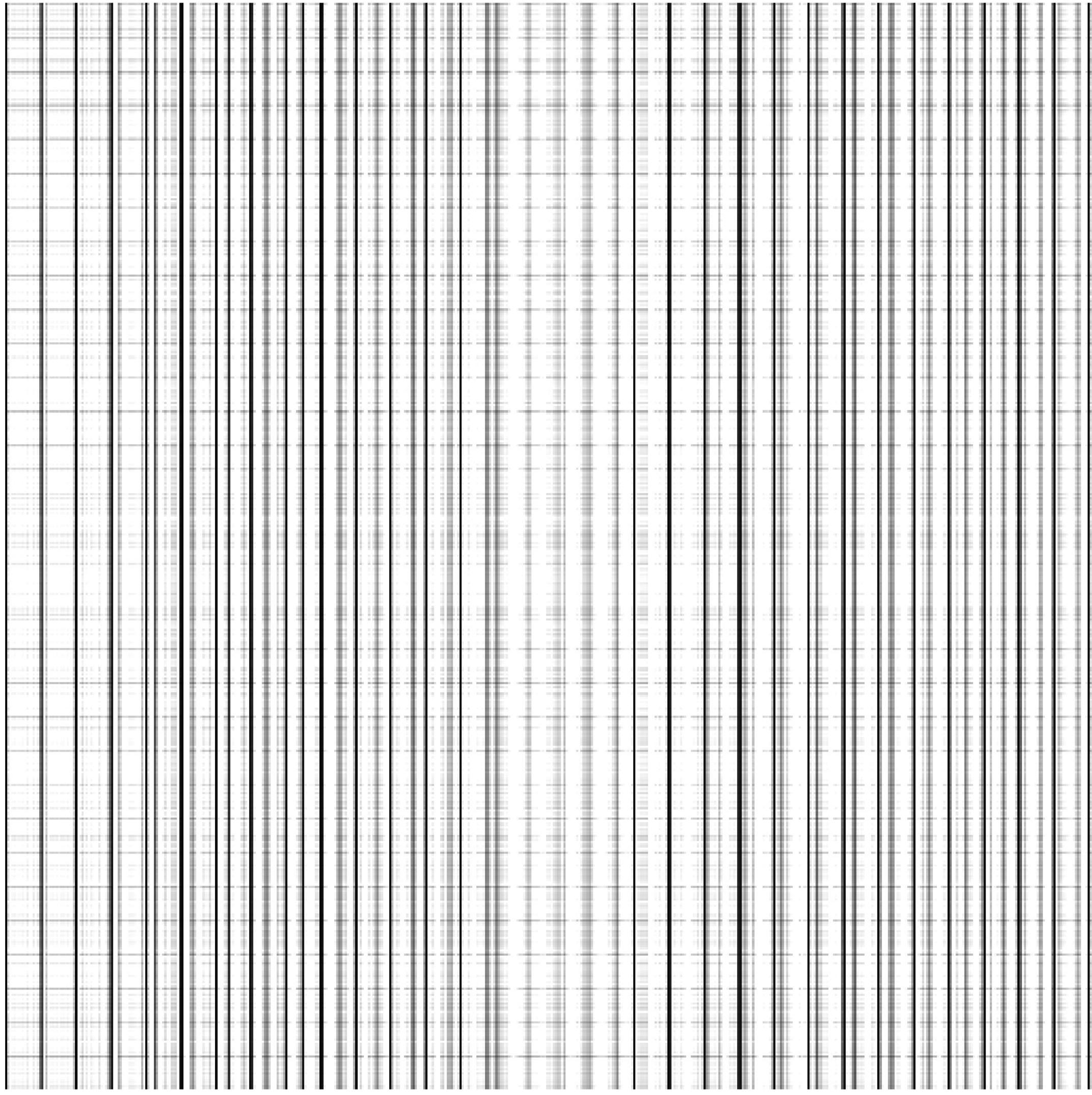}\vspace{1pt}
			\includegraphics[width=\linewidth]{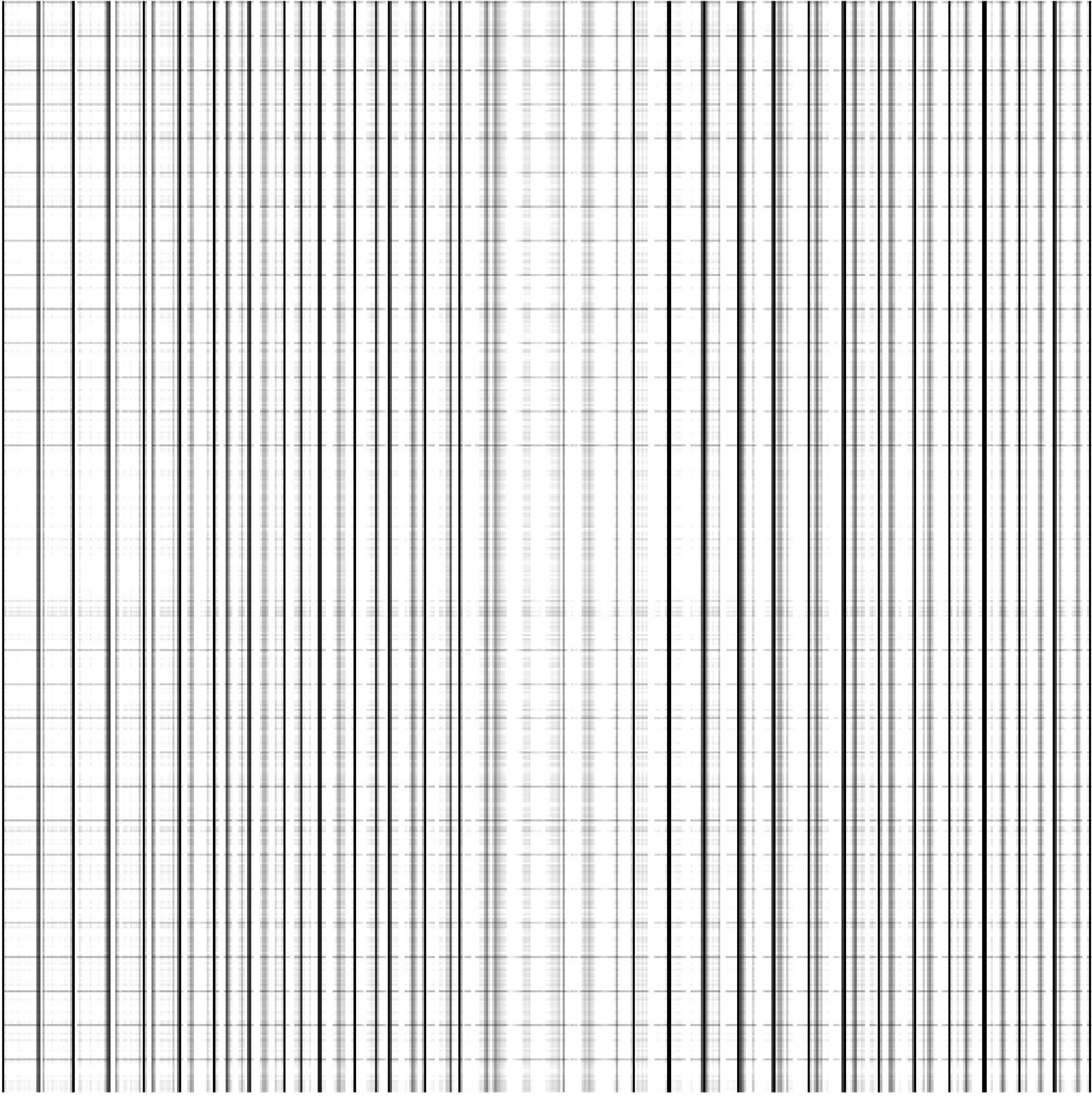}
		\end{minipage}
	}		
	\end{minipage}
	\vfill
	\caption{Results of ``Ruler'' recovery. The first column is the sample images with $sr = 0.2, \,0.6,\, 0.8.$ The rest columns are the recovery images by the corresponding algorithms}
	\label{fig5}
\end{figure*}

\begin{sidewaystable}
\sidewaystablefn%
\begin{center}
\begin{minipage}{\textheight}
\caption{Numerical results on  the  image data for $sr=0.2,\, sr=0.6,\,sr=0.8$}\label{tab3}
\setlength{\tabcolsep}{0.8mm}{
\begin{tabular*}{\textheight}{@{\extracolsep{\fill}}lccccccccccccc@{\extracolsep{\fill}}}
\toprule%
\multirow{2}{*}{Image}&\multirow{2}{*}{sr}& \multicolumn{5}{@{}c@{}}{PSNR}& \multicolumn{5}{@{}c@{}}{T}&\multicolumn{2}{@{}c@{}}{$\mu$} \\\cmidrule{3-7}\cmidrule{8-12}\cmidrule{13-14}%
& &GIMSPG&MSPG& VBMFL1& FPCA&SVT & GIMSPG &MSPG& VBMFL1& FPCA&SVT&MSPGE &MSPG  \\
 					\midrule
 					chart   &0.2&18.04&17.16&15.96 &15.60	&9.16
 					&9.78 &18.10&18.67&0.53	&2.72
 					&0.022&0.062\\  		
 					&0.6&29.13&29.01&19.80 &18.80	&10.49
 					&2.19 &12.59&33.10&14.16	&2.78
 					&0.035&0.131\\
 					
 					&0.8&30.84 &30.18&20.63&19.50&10.65
 					&1.55&3.16&47.94&13.64 &3.58
 					&0.107&0.234
 					\\
 					ruler 	&0.2&21.17&20.74&20.10&19.90&5.21
 					&25.45&55.76&5.23&26.25&5.42
 					&0.020&0.105
 					\\
 					&0.6&27.02&27.04&19.80 &18.80	&5.22
 					&11.50 &12.86&19.67&54.91	&6.31
 					&0.044&0.238\\
 					
 					&0.8&30.01&30.39&20.50&20.20&5.24
 					&7.11&14.83&49.91&49.98&7.09
 					&0.093&0.273
 					\\  	  	
\botrule
\end{tabular*}
}
\end{minipage}
\begin{minipage}{\textheight}
\caption{Numerical results of the MRI volume dataset for $sr=0.6,\,sr=0.8$}\label{tab4}
\setlength{\tabcolsep}{0.8mm}{
\begin{tabular*}{\textheight}{@{\extracolsep{\fill}}lccccccccccccc@{\extracolsep{\fill}}}
\toprule%
\multirow{2}{*}{Image}&\multirow{2}{*}{sr}& \multicolumn{5}{@{}c@{}}{PSNR}& \multicolumn{5}{@{}c@{}}{T}&\multicolumn{2}{@{}c@{}}{$\mu$} \\\cmidrule{3-7}\cmidrule{8-12}\cmidrule{13-14}%
	& &GIMSPG&MSPG& VBMFL1& FPCA&SVT & GIMSPG &MSPG& VBMFL1& FPCA&SVT&MSPGE &MSPG  \\
		\midrule
		MRI&0.6&29.13&29.01&29.58 &26.70	&13.56
		&2.19 &12.59&32.91&10.67&2.22
		&0.035&0.131\\  		
		
		&0.8&33.06&32.99&30.33&27.20&13.59
		&1.79&10.57&26.61&11.58 &2.44
		&0.029&0.137
 					\\  	  	
\botrule
\end{tabular*}
}
\end{minipage}
\end{center}
\end{sidewaystable}

\subsection{MRI Volume Dataset}\label{4.3}

  In this subsection, we perform the experiments on the MRI image which is of size
 $217 \times 181$ with 181 slices and we select the 38th slice  for the
 experiments.  The two   sampling ratios
 $sr=0.6,\,sr=0.8$ are considered.
   The parameters in GMM noise are set as $\sigma^2_A=0.0001$, $\sigma^2_B=0.1$ and $c=0.1$.

 From Figure \ref{fig6} and Table \ref{tab4}, we see that the recoverability of VBMFL1 is relatively better but it spends more time.  Particularly,  we
  notice that the PSNR of GIMSPG algorithm is higher than other algorithms except the VBMFL1 in Table \ref{tab4}. Besides, the used time $T$ of GIMSPG is the least one. Compared with MSPG, the last value $\mu$ of smoothing facor of GIMSPG algorithm is lower.
\begin{figure*}[!ht]
	\centering
	\begin{minipage}[b]{1\linewidth}
		\subfloat[Observed]{
			\begin{minipage}[b]{0.135\linewidth}
				\centering
				\includegraphics[width=\linewidth]{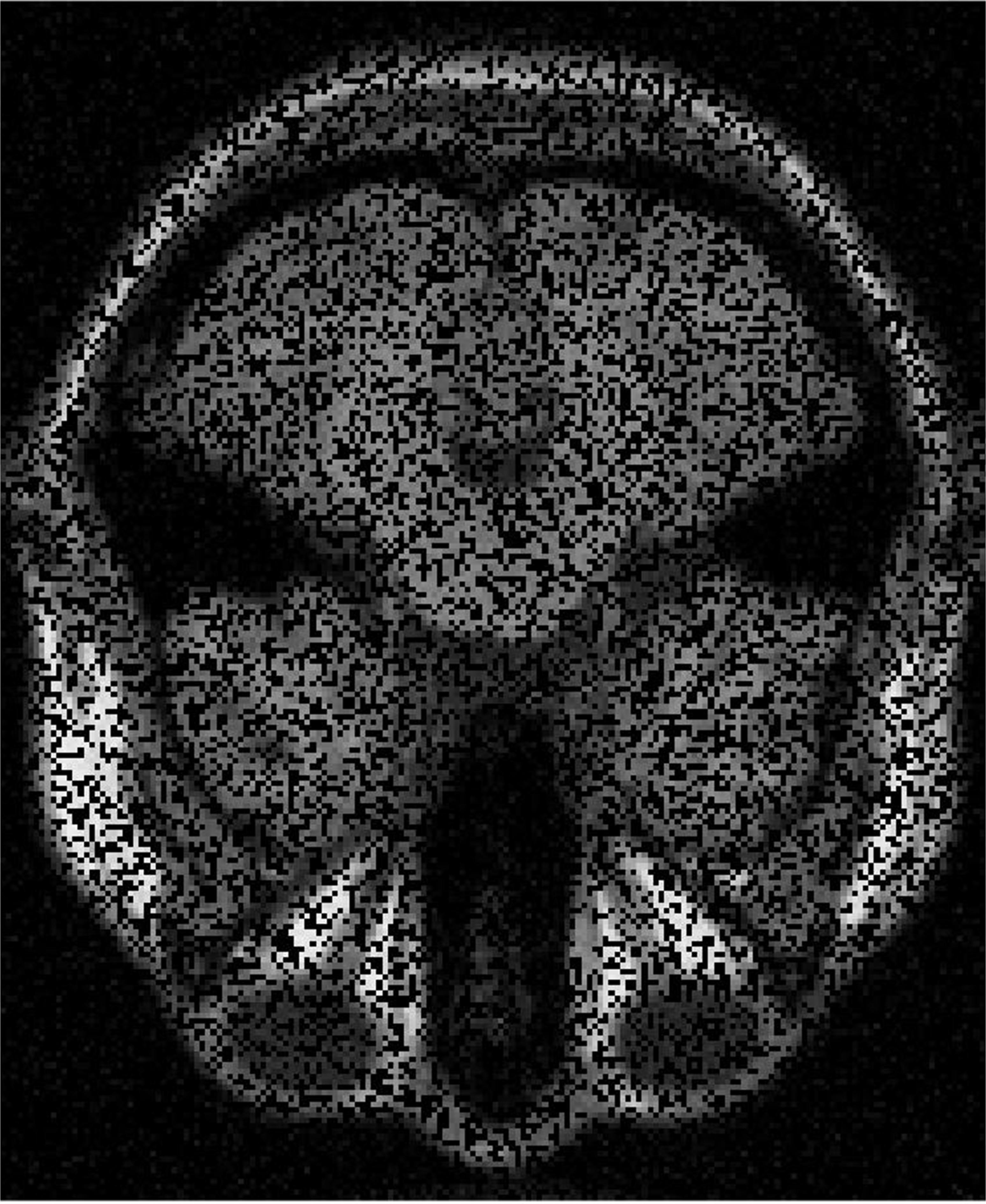}\vspace{1pt}			
				\includegraphics[width=\linewidth]{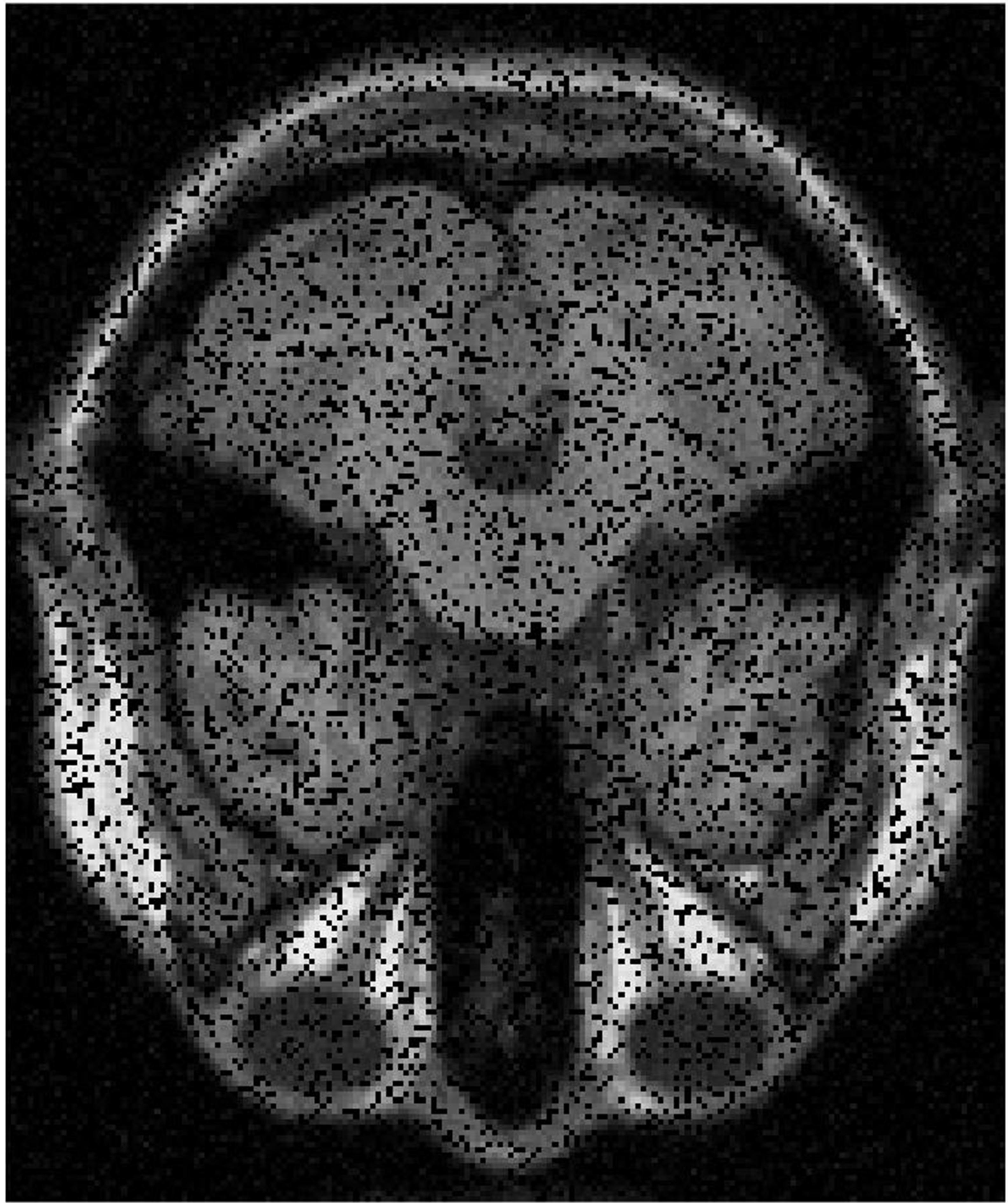}
			\end{minipage}
		}
		\hfill
		\subfloat[GIMPG]{
			\begin{minipage}[b]{0.135\linewidth}
				\centering				
				\includegraphics[width=\linewidth]{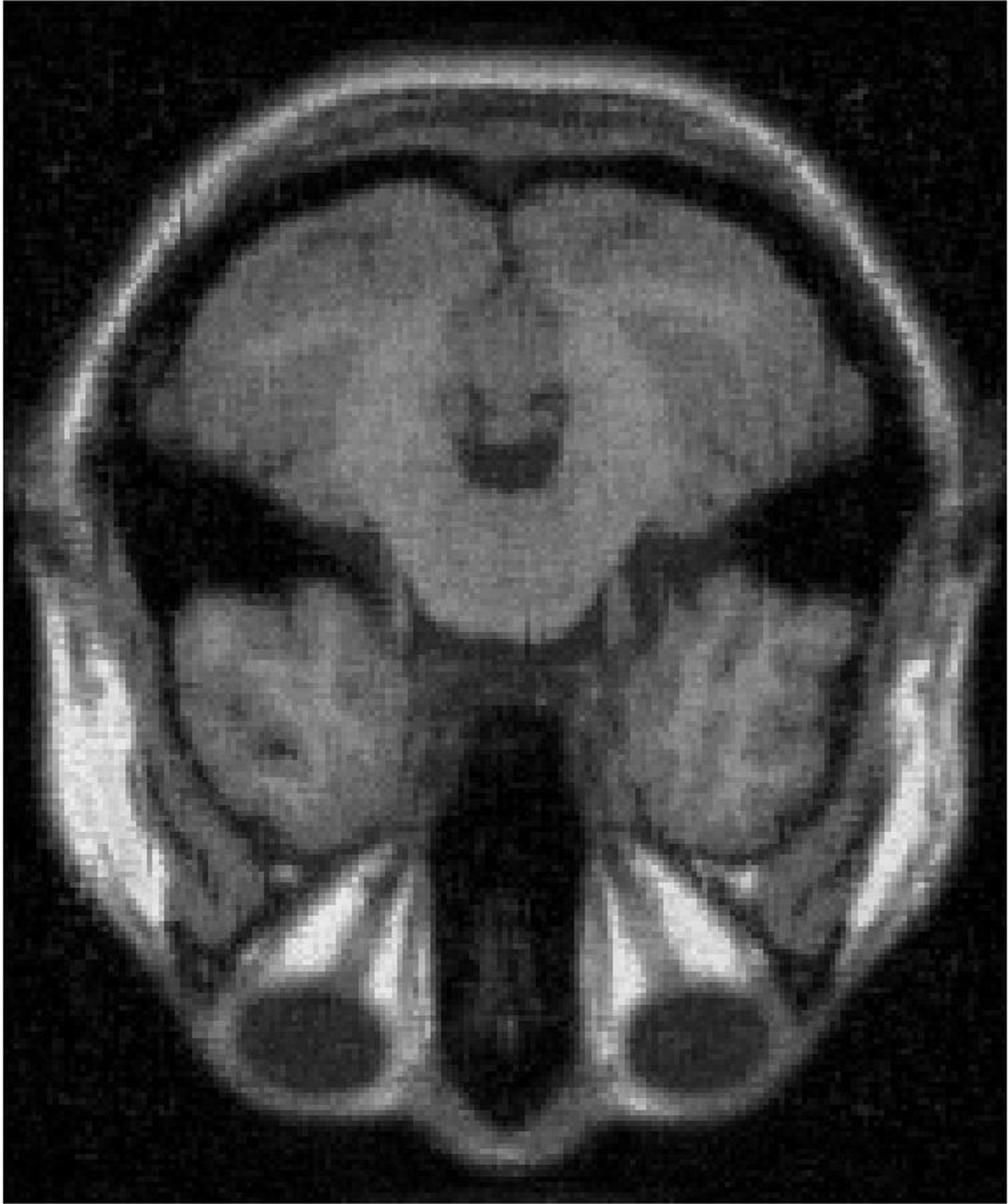}\vspace{1pt}
				\includegraphics[width=\linewidth]{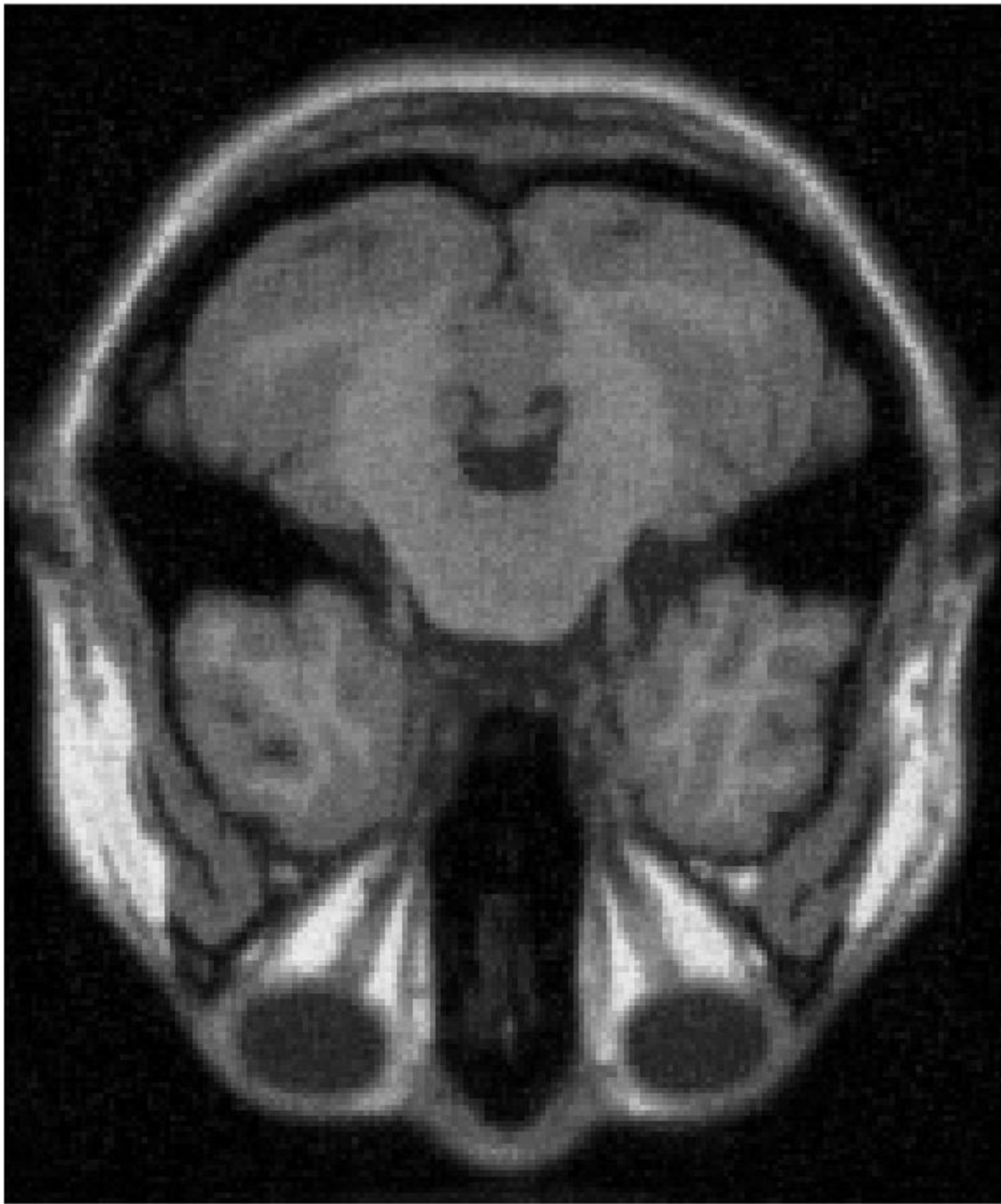}
			\end{minipage}
		}
		\hfill
		\subfloat[MSPG]{
			\begin{minipage}[b]{0.135\linewidth}
				\centering				
				\includegraphics[width=\linewidth]{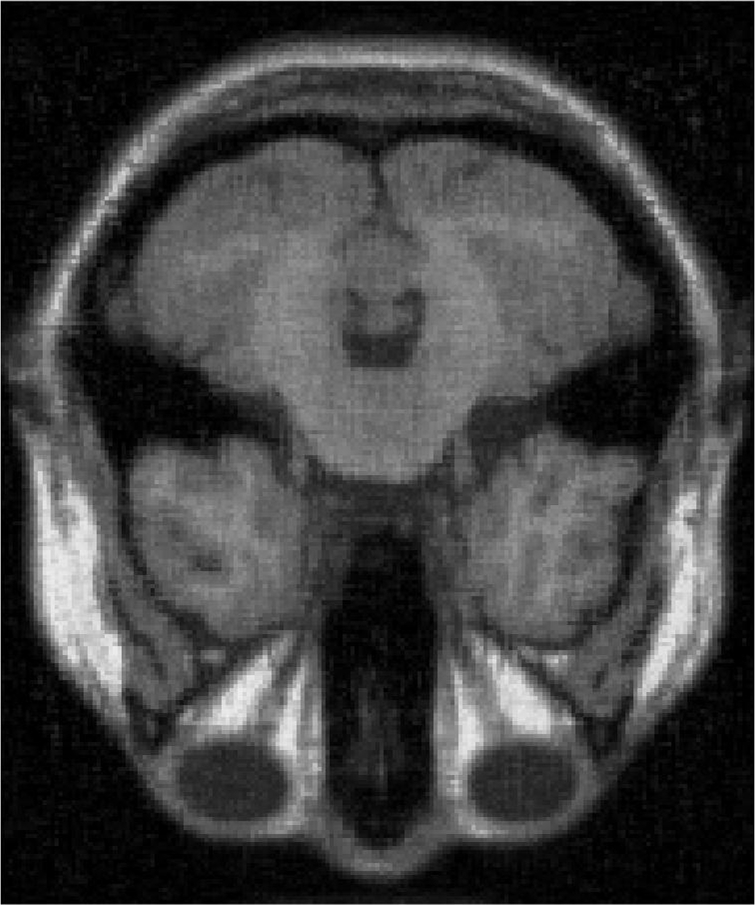}\vspace{1pt}
				\includegraphics[width=\linewidth]{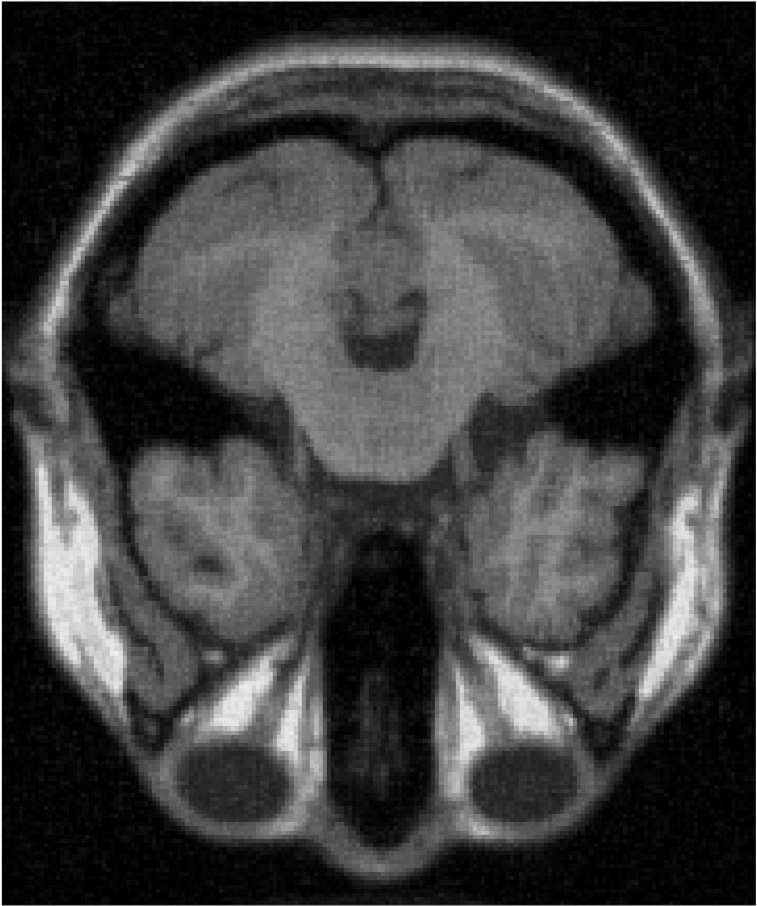}
			\end{minipage}
		}
		\hfill
		\subfloat[VBMFL1]{
			\begin{minipage}[b]{0.135\linewidth}
				\centering				
				\includegraphics[width=\linewidth]{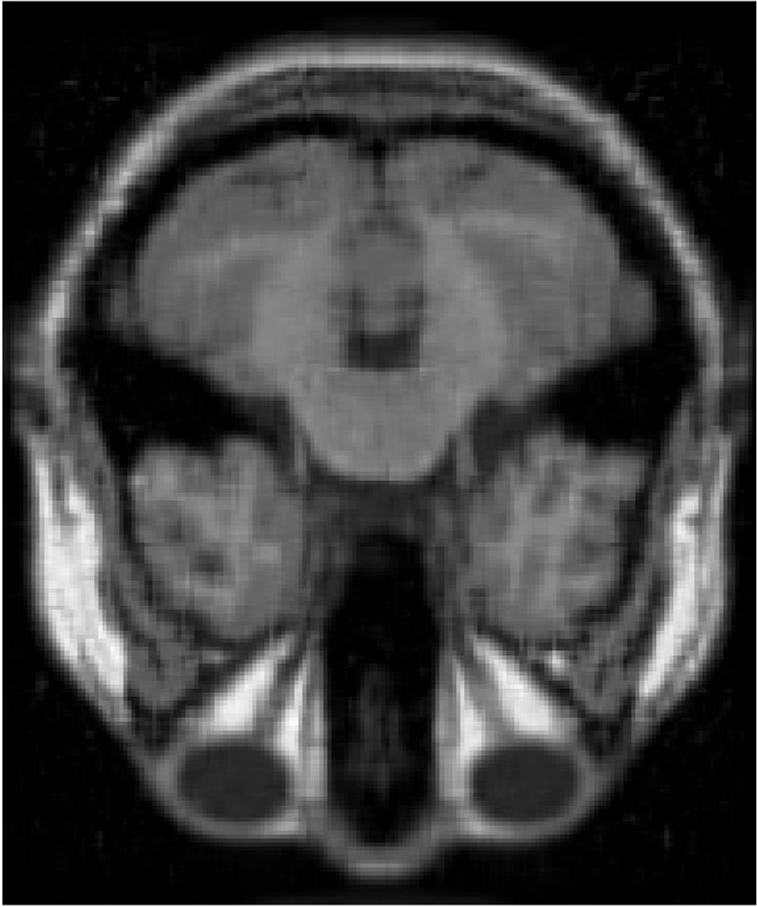}\vspace{1pt}
				\includegraphics[width=\linewidth]{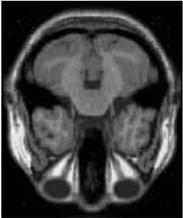}
			\end{minipage}
		}
		\hfill
		\subfloat[FPCA]{
			\begin{minipage}[b]{0.135\linewidth}
				\centering				
				\includegraphics[width=\linewidth]{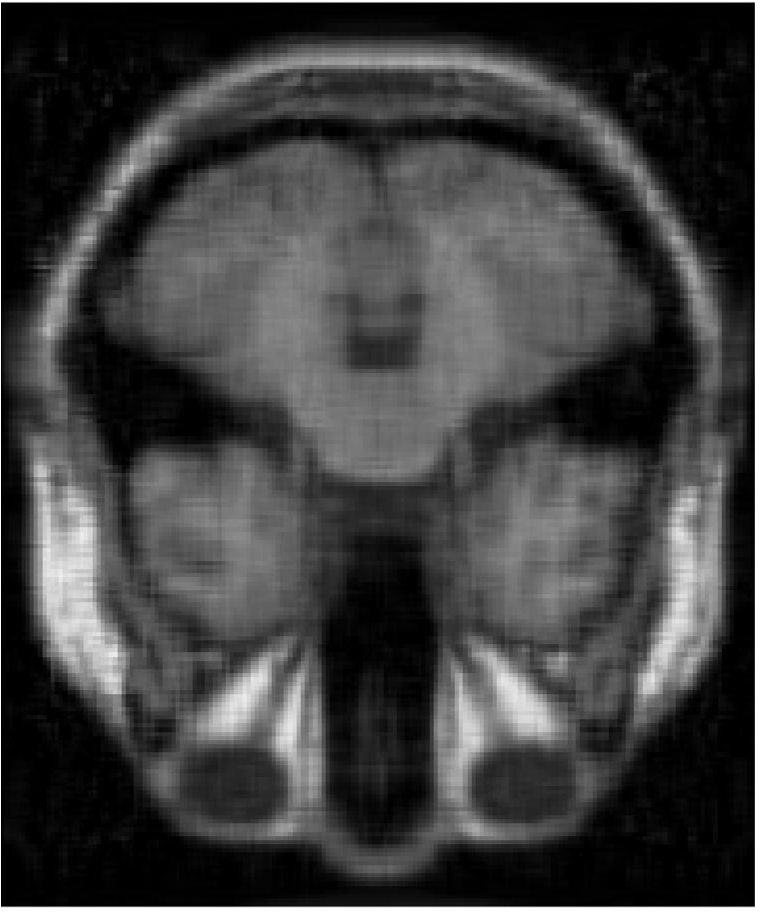}\vspace{1pt}
				\includegraphics[width=\linewidth]{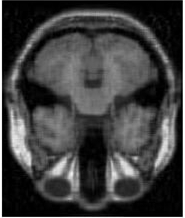}
			\end{minipage}
		}
		\hfill
		\subfloat[SVT]{
		\begin{minipage}[b]{0.135\linewidth}
			\centering				
			\includegraphics[width=\linewidth]{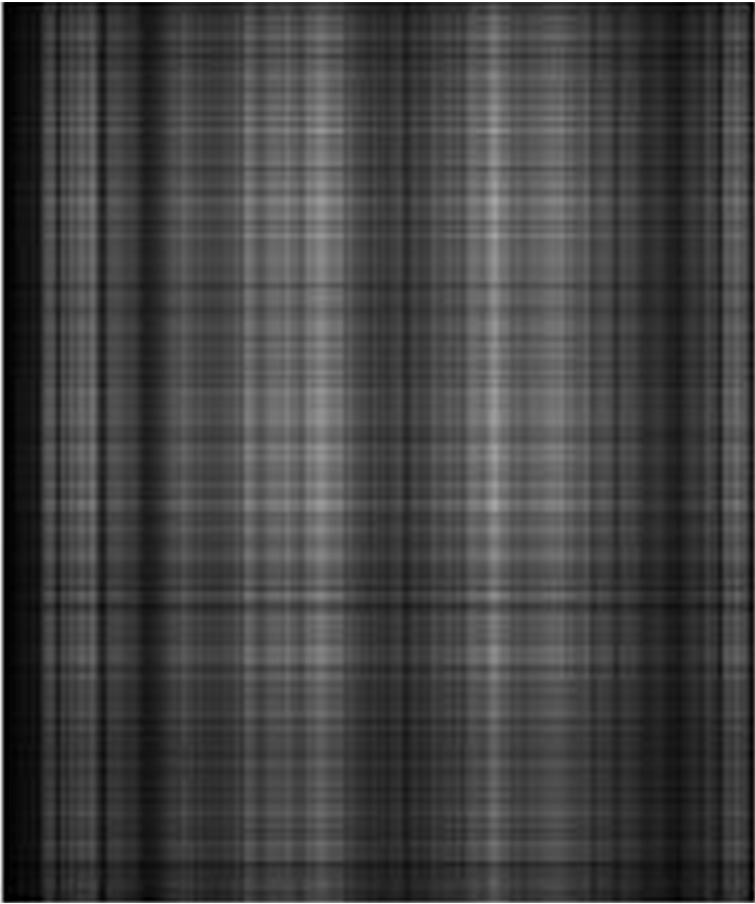}\vspace{1pt}
			\includegraphics[width=\linewidth]{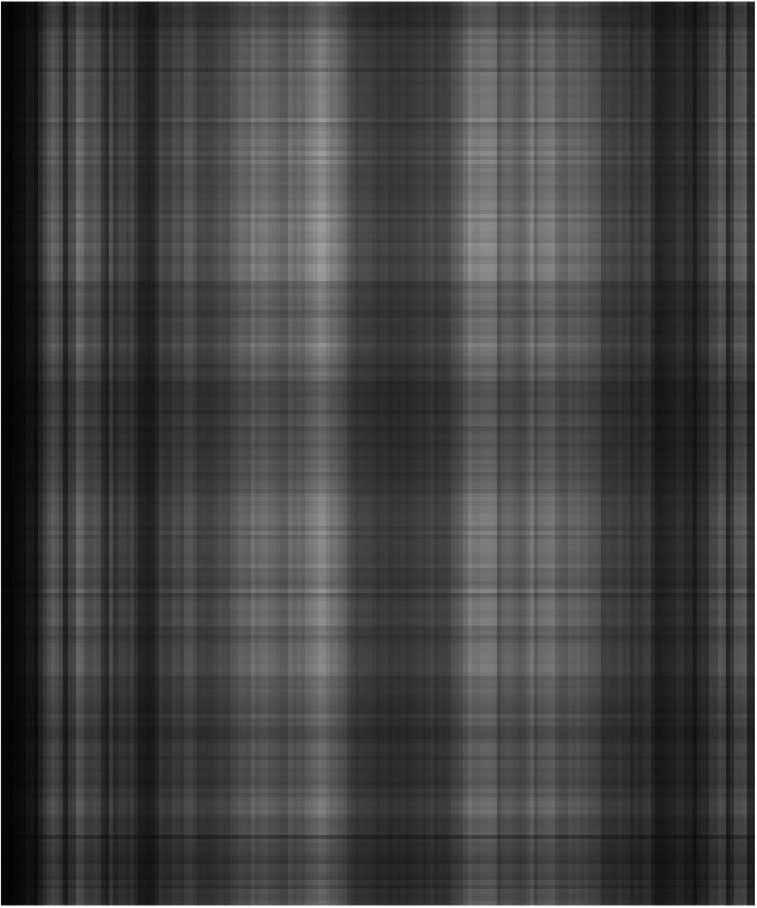}
		\end{minipage}
	}
		
	\end{minipage}
	\vfill
	\caption{ Results of  ``MRI'' recovery. The first column is the sample images with $sr =0.6,\, 0.8.$ The rest columns are the recovery images by the corresponding algorithms}
	\label{fig6}
\end{figure*}

\section{Conclusions }\label{sec5}
This paper mainly proposes the GIMSPG algorithm for the exact continuous model \cite{Yu_Q_Zhang_X} of matrix rank minimization problem and discuss its convergence under the flexible parameters condition. It is shown that the  singular values of the accumulation point have a common support set and the nonzero elements have a unified lower bound. Further,  the zero  singular values of the accumulation point can be achieved within finite iterations. Moreover, we prove that any accumulation point of the sequence generated by GIMSPG algorithm is a lifted stationary point of the continuous relaxation model under the flexible parameter constraint. Further, we generalize Euclidean distance to the Bregman distance in the GIMSPG algorithm and the   . Under the appropriate parameters, the range of the extrapolation parameters is large enough which also includes the sFISTA system (\ref{c}) with fixted restart.
 At last,  the superiority of the GIMSPG algorithm is verified on random data and real image data respectively.

 \begin{appendices}
 	 \section{Appendix }\label{sec6}
		\begin{lemma}\label{lemma3.33}
		For any $k\in \mathbb{N}$, suppose $\{X^k \}$ be the sequence generated by GIMSPG,
		we have
		\begin{align*}
			& H_{\delta_{k+1}}(X^{k+1}, X^{k},\mu_{k+1},\mu_k)-H_{\delta_{k}}(X^{k}, X^{k-1},\mu_{k},\mu_{k-1}) \notag\\
			\leq &((\frac{\tilde{L}-\varrho h_k}{2}+ \frac{h_k\alpha_{k}-\beta_k\tilde{L}}{2})\mu^{-1}_k
			+\frac{\alpha_{k}h_{k}\mu^{-1}_{k+1}}{2})\Vert X^{k+1}-X^{k}\Vert^2 \notag\\
			&+\frac{\tilde{L}\beta_k(\beta_k-1) }{2}\mu^{-1}_k{\Vert X^{k}-X^{k-1}\Vert}^2.
		\end{align*}
		When the parameters  $\alpha_k, \beta_k,h_k$ satisfying Assumption \ref{assumption_5},
		it holds that
		\begin{align}\label{3.11111}
			H_{\delta_{k+1}}(X^{k+1}, X^{k},\mu_{k+1},\mu_k)-H_{\delta_{k}}(X^{k}, X^{k-1},\mu_{k},\mu_{k-1})\leq \frac{-\varepsilon h_k  }{2}\mu^{-1}_{k+1}{\Vert X^{k+1}-X^{k}\Vert}^2
		\end{align}
		Moreover, the sequence
		$ H_{\delta_{k+1}}(X^{k+1}, X^{k},\mu_{k+1},\mu_k)$
		is a nonincreasing sequence.
	\end{lemma}
	
\begin{proof}
		By (\ref{3.7aa}), we have
		\begin{align}\label{3.133}
		&\langle X^{k+1}, \nabla \tilde{f}(Z^k,\mu_k)-h_k\mu^{-1}_k(Y_k-X_k)\rangle
			+h_k \mu^{-1}_k  D_{\phi}(X^{k+1},X^k)+\lambda{{\Phi}^{d^k}}(X^{k+1})\notag\\
			\leq &	\langle X^k, \nabla
 \tilde{f}(Z^k,\mu_k)-h_k\mu^{-1}_k(Y_k-X_k)\rangle+\lambda{{\Phi}^{d^k}}(X^k) .
		\end{align}
		Since $\nabla\tilde{f}(\cdot,\mu_k)$  has Lipschitz constant $\tilde{L}{\mu_k}^{-1}$, it follows from  the Definition \ref{def3.1}(v) that
		\begin{align}\label{3.144}
			\tilde{f}(X^{k+1},\mu_k)\leq \tilde{f}f(Z^{k},\mu_k)+\langle X^{k+1}-Z^k, \nabla \tilde{f}(Z^k,\mu_k)\rangle + \frac{1}{2}\tilde{L}{\mu_k}^{-1}{\Vert X^{k+1}-Z^k \Vert}^2.
		\end{align}	
	Moreover, $\tilde{f}(X^k,\mu_k)$ is convex with respect to $X^k$, we have
		\begin{align}\label{3.14aa}
			\tilde{f}(Z^{k},\mu_k) + \langle X^{k}-Z^k, \nabla \tilde{f}(Z^k,\mu_k) \rangle \leq\tilde{f}(X^{k},\mu_k)
		\end{align}
		Combining (\ref{3.133}), (\ref{3.144}) and (\ref{3.14aa}), we have
		\begin{align}
			& \widetilde{\mathcal F}^{d^k}(X^{k+1},\mu_k)-\widetilde{\mathcal F}^{d^k}(X^{k},\mu_k)  \\ \notag
			 \leq&  -h_k{\mu^{-1}_k}\langle Y^{k}-X^k, X^{k}-X^{k+1} \rangle
			-h_k{\mu^{-1}_k} D_\phi(X^{k+1},X^k)+\frac{\tilde{L}\mu^{-1}_k}{2} {\Vert {X}^{k+1}-Z^k \Vert}^2.\notag
		\end{align}
		Denote $ \Delta_k:= X^k-X^{k-1}$, then it has
		$ \alpha_k\Delta_k:= Y^k-X^k$, $\beta_k\Delta_k:=Z^k-X^k$, $\beta_k\Delta_k-\Delta_{k+1}:=Z^k-X^{k+1}$.
		That is,
		\begin{align}
			&\widetilde{\mathcal F}^{d^k}(X^{k+1},\mu_k)-\widetilde{\mathcal F}^{d^k}(X^{k},\mu_k) \notag \\
			& \leq h_k{\mu_k}^{-1} \langle \alpha_k\Delta_k, \Delta_{k+1} \rangle +\frac{\tilde{L} \mu^{-1}_k}{2} {\Vert\beta_k\Delta_k- \Delta_{k+1}\Vert}^2 - \frac{1}{2} \varrho h_k{\mu_k}^{-1}{\Vert \Delta_{k+1}\Vert}^2\notag \\
			&= \frac{(\tilde{L}-\varrho h_k)\mu^{-1}_k}{2}{\Vert \Delta_{k+1}\Vert}^2+\frac{\tilde{L}}{2}\beta^2_k\mu^{-1}_k{\Vert \Delta_{k}\Vert}^2 +
			(h_k\alpha_k-\beta_k\tilde{L})\mu^{-1}_k \langle \Delta_k, \Delta_{k+1} \rangle  \notag \\
			&\leq (\frac{ \tilde{L}-\varrho h_k }{2}+\frac{h_k\alpha_k-\beta_k\tilde{L}}{2})\mu^{-1}_k
			{\Vert \Delta_{k+1}\Vert}^2+(\frac{\tilde{L}\beta^2_k}{2}+\frac{h_k\alpha_k-
\beta_k\tilde{L}}{2})\mu^{-1}_k
			{\Vert \Delta_{k}\Vert}^2
		\end{align}
		where the second inequality comes from the Cauchy-Schwartz inequality.
		
		Letting $d^k=d^{X^k}$ and according to $\mathcal {\widetilde{F}}^{d^k}(X^{k+1},\mu_k)\ge \mathcal {\widetilde{F}}(X^{k+1},\mu_k)$, we have
		\begin{align}\label{3.155}
			&\widetilde {\mathcal F}(X^{k+1},\mu_k)
			+\frac{\alpha_{k+1}h_{k+1}\mu^{-1}_{k+1}}{2}{\Vert \Delta_{k+1}\Vert}^2
			-(\widetilde{\mathcal F}(X^{k},\mu_k)+\frac{\alpha_{k }h_{k }\mu^{-1}_{k}}{2}{\Vert \Delta_{k}\Vert}^2 ) \notag \\
			&\leq  ((\frac{\tilde{L}-\varrho h_k}{2}+\frac{h_k\alpha_k-\beta_k\tilde{L}}{2})\mu^{-1}_k
			+\frac{\alpha_{k+1}h_{k+1}\mu^{-1}_{k+1}}{2})
			{\Vert \Delta_{k+1}\Vert}^2  \notag\\
			&\quad+((\frac{\tilde{L}\beta^2_k+h_k\alpha_k-\beta_k\tilde{ L}}{2})\mu^{-1}_k-\frac{\alpha_kh_{k}\mu^{-1}_{k}}{2})
			{\Vert \Delta_{k}\Vert}^2
		\end{align}
		By the Definition \ref{def3.1}(iv) and the nonincreasing of $\mu_k$, we easily have that
		\begin{align}\label{3.15aa}
			\widetilde {\mathcal F}(X^{k},\mu_k)\leq \widetilde{\mathcal F} (X^{k},\mu_{k-1})+\kappa(\mu_{k-1}-\mu_k).
		\end{align}
		Together (\ref{3.155}) with (\ref{3.15aa}), we get
		\begin{align*}
			&H_{\delta_{k+1}}(X^{k+1}, X^{k},\mu_{k+1},\mu_k)-H_{\delta_{k}}(X^{k}, X^{k-1},\mu_{k},\mu_{k-1})  \notag \\
			\leq& ((\frac{\tilde{L}-\varrho h_k}{2}+\frac{h_k\alpha_k-\beta_k\tilde{L}}{2})\mu^{-1}_k
			+\frac{\alpha_{k+1}h_{k+1}\mu^{-1}_{k+1}}{2})
			{\Vert \Delta_{k+1}\Vert}^2+
			\frac{\tilde{L}(\beta^2_k-\beta_k)}{2}\mu^{-1}_{k}{\Vert \Delta_{k}\Vert}^2
		\end{align*}
		According to the parameters constraint in Assumption \ref{assumption_5}, we easily have that  $(\tilde{L}-\varrho h_k+h_k\alpha_k-\beta_k\tilde{L})\mu^{-1}_k+\alpha_{k}h_{k}\mu^{-1}_{k+1
		}<-\varepsilon h_k\mu^{-1}_{k+1}$, $\beta^2_k-\beta_k\leq 0$.
		Hence, the inequality (\ref{3.11111}) holds.
		So the desired result is obtained.
\end{proof}
 	\end{appendices}


\begin{thebibliography}{99}
	
	\bibitem{A_Argyriou}
	Argyriou, A., Evgeniou, T., Pontil, M.: Convex multi-task feature learning. Mach. Learn.  \textbf{73}, 243-272 (2008)	
	\bibitem{Beck_A_Teboulle_M}
  Beck, A., Teboulle, M.: A fast iterative shrinkage-thresholding algorithm for linear inverse problems. SIAM J. Imaging Sci. {\textbf{2}}, 183-202 (2009)
	\bibitem{Wei_B}
Bian,W.: Smoothing accelerated algorithm for constrained nonsmooth convex optimization problems (in Chinese). Sci. Sin. Math. {\textbf{50}}, 1651--1666 (2020)
   \bibitem{Bolte_J}
Bolte, J., Sabach, S., Teboulle, M., Vaisbourd, Y.: First-order methods beyond convexity and Lipschitz
gradient continuity with applications to quadratic inverse problems. SIAM J. Optim. \textbf{28}, 2131-2151
(2018)
	\bibitem {Bian_W_Chen_X}
 Bian, W., Chen, X.: A smoothing proximal gradient algorithm for nonsmooth convex regression with cardinality penalty. SIAM J. Numer. Anal. {\textbf{58}}, 858-883 (2020)
     \bibitem{Bouwmans}
   Bouwmans,T., Zahzah, E, H.: Robust PCA via Principal Component Pursuit: A Review for a Comparative Evaluation in Video Surveillance. Comput Vis. Image Underst.  {\textbf{122}}, 22-34 (2014)
  \bibitem{Cai_J}
     Cai, J., Candes, E., Shen, Z.: A singular value thresholding algorithm for matrix completion. SIAM J. Optim. {\textbf{20}}, 1956-1982 (2008)

     \bibitem{Chambolle_A_Dossal_C}
   Chambolle, A., Dossal, C.: On the convergence of the iterates of the ``fast iterative shrinkage/thresholding algorithm". J. Optim. Theory Appl. {\textbf{166}}, 968-982 (2015)

     \bibitem{Chen_X}
     Chen, X.: Smoothing methods for nonsmooth, nonconvex minimization. Math. Program. {\textbf{134}},
     71-99 (2012)


     \bibitem{Fornasier_M}
     Fornasier, M., Rauhut, H., Ward, R.: Low-rank matrix recovery via iteratively reweighted least
     squares minimization. SIAM J. Optim.  {\textbf{21}}, 1614-1640 (2011)


  \bibitem {He_Y}
  He, Y., Wang, F., Li, Y., Qin, J., Chen, B.: Robust matrix completion via maximum correntropy
  criterion and half-quadratic optimization. IEEE T. Signal Proces. {\textbf{68}}, 181-195 (2020)


  \bibitem {Ji_S}
  Ji, S., Ye, J.: An accelerated gradient method for trace norm minimization. Proceedings of the 26th annual international conference on machine learning. 457-464 (2009)


    \bibitem {Kulis_B}
  Kulis, B., Sustik, M.,A., Dhillon, I.,S.: Low-Rank Kernel Learning with Bregman Matrix Divergences. J. Mach. Learn. Res. \textbf{10}, 2009

  \bibitem{Lai_M}
  Lai, M., Xu,Y., Yin, W.: Improved iteratively reweighted least squares for unconstrained smoothed  q minimization. SIAM J. Numer. Anal. {\textbf{51}}, 927-957 (2013)

  \bibitem{A_S_Lewis}
  Lewis, A, S., Sendov, H, S.: Nonsmooth analysis of singular values. Part II: Applications. Set-Valued Anal. {\textbf{13}}, 243-264 (2005)

\bibitem{li_wenjing}
Li, W., Bian, W., Toh, K, C.: DC algorithms for a class of sparse group $\ell_0 $ regularized optimization problems (2021). arXiv:2109.05251


  \bibitem{Liang_J_and_C_B}
 Liang, J., Schonlieb, C, B.: ``Faster FISTA." European Signal Processing Conference. (2018)

  \bibitem {Lu_Z}
  Lu, Z., Zhang, Y., Lu, J.: $\ell _p $ Regularized low-rank approximation via iterative reweighted singular value minimization. Comput. Optim. Appl. {\textbf{68}}, 619-642 (2017)

     \bibitem{Ma_S}
    Ma, S., Goldfarb, D., Chen, L.: Fixed point and Bregman iterative methods for matrix rank minimization. Math. Program. {\textbf{128}}, 321-353 (2011)

     \bibitem{Ma_T}
     Ma, T., H., Lou, Y., Huang, T., Z.: Truncated $\ell_{1-2}$  models for sparse recovery and rank minimization. SIAM J. Imaging Sci. {\textbf{10}}, 1346-1380 (2017)


\bibitem{Mesbahi}
Mesbahi, M., Papavassilopoulos, G., P.: On the rank minimization problem over a positive semidefinite linear matrix inequality. IEEE T. Automat. Contr. {\textbf{42}}, 239-243(1997)

\bibitem{Nesterov_Y}
 Nesterov, Y.: Introductory Lectures on Convex Programming. Kluwer Academic Publisher, Dordrecht (2004)

\bibitem{Nesterov_Y0}
Nesterov Y. Gradient methods for minimizing composite functions[J]. Mathematical programming, 2013, 140(1): 125-161.
\bibitem{Nesterov_Y1}
  Nesterov, Y.: A method for solving the convex programming problem with convergence rate $O(1/k^2)$.
Dokl. Akad. Nauk. SSSR 269, 543¨C547 (1983)



\bibitem{ODonoghue}
 O'Donoghue, B., Candes, E.J.: Adaptive restart for accelerated gradient schemes. Found. Comput.
Math. {\textbf{15}}, 715-732 (2015)
\bibitem{Parikh_N_Boyd_S}
Parikh, N., Boyd, S.: Proximal algorithms. Found. Trends Mach. Learn. {\textbf{1}}, 127-239 (2014)

\bibitem{Recht}
Recht, B., Fazel, M., Parrilo, P.: Guaranteed minimum-rank solutions of linear matrix equations via
nuclear norm minimization. Siam Rev. {\textbf{52}}, 471-501 (2010)


\bibitem{Teboulle_M}
 Teboulle, M.: A simplified view of first order methods for optimization. Math. Program. \textbf{170}, 67-96
(2018)

\bibitem {Toh_K_C_Yun_S}
Toh, K., C., Yun, S.: An accelerated proximal gradient algorithm for nuclear norm regularized linear least squares problems. Pac. J. Optim. \textbf{6}, 615-640 (2010)


\bibitem{Wu_F_Bian_W_Xue_X}
Wu, F., Bian, W., Xue, X.: Smoothing fast iterative hard thresholding algorithm for $\ell_0$ regularized
nonsmooth convex regression problem (2021). arXiv:2104.13107

\bibitem{Wu_zhongming}
Wu, Z., Li, C., Li, M.,  Lim, A.: Inertial proximal gradient methods with Bregman regularization for a class of nonconvex optimization problems. J. Global Optim. \textbf{79}, 617-644 (2021)

\bibitem{Wu_Z_Li_M}
Wu, Z., Li, M.: General inertial proximal gradient method for a class of nonconvex nonsmooth optimization problems. Comput. Optim. Appl. \textbf{73}, 129-158(2019)

 \bibitem{Xu_H}
  Xu, H., Caramanis, C., Sanghavi, S.: Robust PCA via outlier pursuit. IEEE Trans. Inf. Theory. \textbf{58},
 3047-3064 (2012)

\bibitem{Yu_Q_Zhang_X}
Yu, Q., Zhang, X.: A smoothing proximal gradient algorithm for matrix rank minimization problem. Comput. Optim. Appl. 1-20 (2022)

\bibitem{Zhang_C_Chen_X}
Zhang, C., Chen, X.: Smoothing projected gradient method and its application to stochastic linear complementarity problems. SIAM J. Optim. \textbf{20}, 627-649(2009)

\bibitem{Zhang_J}
Zhang, J., Yang, X., Li, G.,  Zhang, K. The smoothing proximal gradient algorithm with extrapolation for the relaxation of $l_0$ regularization problem (2021). arXiv:2112.01114.

 \bibitem{Zhao_Q}
 Zhao, Q., Meng, D., Xu, Z., Yan, Y.: $L_{1}$-norm low-rank matrix factorization by variational Bayesian method. IEEE T. Neur. Net. Lear. \textbf{26}, 825-839(2015)

\bibitem{Zheng_Y}
Zheng, Y., Liu, G., Sugimoto, S., Yan, S., Okutomi, M.: Practical low-rank matrix approximation under robust $L_1$-norm. In 2012 IEEE Conference on Computer Vision and Pattern Recognition (pp. 1410-1417). IEEE



\end{thebibliography}
\end{document}